\documentclass[reqno, 12pt]{amsart}
\usepackage{amsmath,amssymb,mathrsfs,amsthm,amsfonts,mathtools}
\usepackage[inline]{enumitem} 				
\usepackage[usenames,dvipsnames]{xcolor}
\usepackage{graphicx} 
\usepackage[paper=letterpaper,margin=1in]{geometry}
\definecolor{hypercolor}{HTML}{003399}
\usepackage{hyperref}
 	
\hypersetup{colorlinks=true, linkcolor=hypercolor,citecolor=hypercolor}
\usepackage[margin=3pt, font=small]{caption}	
\newtheorem{thm}{Theorem}[section]
\newtheorem{lem}[thm]{Lemma}
\newtheorem{prop}[thm]{Proposition}
\newtheorem{innercustomprop}{Proposition}
\newenvironment{customprop}[1]
{\renewcommand\theinnercustomprop{#1}\innercustomprop}%
{\endinnercustomprop}
\newtheorem{cor}[thm]{Corollary}
\theoremstyle{definition}
\newtheorem{defn}[thm]{Definition}
\newtheorem{conj}[thm]{Conjecture}
\theoremstyle{remark}
\newtheorem{ex}[thm]{Example}
\newtheorem{rmk}[thm]{Remark}

\numberwithin{equation}{section}
\allowdisplaybreaks

\newcommand{\FA}{A}
\newcommand{\FF}{F}
\newcommand{\FG}{G}
\newcommand{\Ff}{f}
\newcommand{\Fg}{g}						
\newcommand{\Fh}{h}
\newcommand{\Fhh}{h_\diri}					
\newcommand{\Fp}{p}
\newcommand{\Fu}{u}		
\newcommand{\Fuu}{u_\diri}					
\newcommand{\Fv}{v}
\newcommand{\fPhi}{\Phi}
\newcommand{\fPsi}{\Psi}
\newcommand{\Fphi}{\phi}
\newcommand{\Fpsi}{\psi}
\newcommand{\Pgamma}{\gamma}
\newcommand{\Peta}{\eta}
\newcommand{\Pchi}{\chi}
\newcommand{\Pzeta}{\zeta}
\newcommand{\Pl}{\ell}		
\newcommand{\Mmu}{\mu}						
\newcommand{\Mnu}{\nu}
\newcommand{\SA}{A}
\newcommand{\SB}{B}
\newcommand{\SD}{D}

\newcommand{\SO}{O}
\newcommand{\SQ}{Q}
\newcommand{\SR}{R}

\renewcommand{\SS}{S}
\newcommand{\SY}{Y}
\newcommand{\SZ}{Z}
\newcommand{\Snet}{\Gamma}					

\newcommand{\eps}{\varepsilon}
\newcommand{\intvl}{\Lambda}				
\newcommand{\textL}{{\normalfont\text{L}}}
\newcommand{\textR}{{\normalfont\text{R}}}

\newcommand{\landscape}{\mathcal{L}}
\newcommand{\fixedpt}{\mathcal{H}}
\newcommand{\dens}{\rho}
\newcommand{\metric}{e}	
\newcommand{\diri}{d}
\newcommand{\rate}{I}						
\newcommand{\Rd}{\mathbb{R}_{\uparrow}^4}
\newcommand{\start}{a}						
\newcommand{\ed}{b}							
\newcommand{\Pgeo}{\pi}
\newcommand{\hypo}{\mathrm{Hypo}}
\newcommand{\EventT}{\mathcal{T}}			

\newcommand{\entbm}{q_{\normalfont\textsc{bm}}}
\newcommand{\entfbm}{r_{\normalfont\textsc{bm}}}
\newcommand{\Entbm}{Q_{\normalfont\textsc{bm}}}
\newcommand{\Engbm}{\mathcal{E}_{\normalfont\textsc{bm}}}
\newcommand{\Engbmn}[1]{\mathcal{E}_{\normalfont\textsc{bm},#1}}
\newcommand{\Entsp}{\mathscr{E}}			
\newcommand{\Entspp}[1]{\Entsp(#1)}			
\newcommand{\cone}{\mathcal{C}}				
\newcommand{\Stime}{\mathfrak{T}}			
\newcommand{\Sbdy}{\mathfrak{B}}			
\newcommand{\Sdom}{\mathfrak{D}}			
\newcommand{\dom}{E}						

\newcommand{\Metric}{\mathsf{e}}			
\newcommand{\Entp}{\mathsf{Ent}}			
\newcommand{\Hfn}{\mathsf{H}}				
\newcommand{\M}{\mathsf{M}}					
\newcommand{\Mnon}{\mathsf{M}_{\mathsf{non}}}
\newcommand{\Ment}{\mathsf{M}_{\mathsf{ent}}}
\newcommand{\graph}{\mathfrak{g}}			
\newcommand{\HL}{\mathsf{HL}}				
\newcommand{\HLbk}{\mathsf{HL}^{\mathsf{bk}}}
\newcommand{\R}{\mathbb{R}}
\newcommand{\Z}{\mathbb{Z}}

\newcommand{\Csp}{C}
\newcommand{\Lsp}{L}
\newcommand{\lsp}{\ell}
\newcommand{\Hsp}{H}

\newcommand{\Msp}{\mathscr{N}}				

\newcommand{\norm}[1]{\Vert #1\Vert}

\renewcommand{\P}{\mathbf{P}}
\newcommand{\ind}{\mathbf{1}}			
\newcommand{\supp}{\mathrm{supp}}		
\renewcommand{\d}{\mathrm{d}}		
\renewcommand{\bar}{\overline}
\newcommand{\und}{\underline}
\newcommand{\til}{\widetilde}

\usepackage{newtxtext}

\graphicspath{{./figs/}}

\title{Solving marginals of the LDP for the directed landscape}
\author{Sayan Das and Li-Cheng Tsai}

\address[Sayan Das]{Department of Mathematics, University of Chicago}
\address[Li-Cheng Tsai]{Department of Mathematics, University of Utah}

\subjclass[2020]{%
60F10
}%
\keywords{%
Burgers' equation, %
directed landscape, %
Kruzhkov entropy, %
Kardar--Parisi--Zhang universality class, %
large deviations%
}%

\begin{document}
\begin{abstract}
We identify the explicit rate function for the upper-tail LDP of the parabolic Airy process and characterize the minimizing metrics, or equivalently the limit shape of the directed landscape, under the upper-tail conditioning.
The LDP result answers Conjecture 10.1 in \cite{das24}, and the limit shape result gives the first proof of the limit shape in similar contexts beyond the case where the minimizer is piecewise linear.
The starting point of our proof is the metric-level LDP for the directed landscape from \cite{das24} that reduces our work to solving a variational problem. Our proof demonstrates how the variational problem fits into the Jensen-Varadhan theory of hydrodynamic large deviations \cite{jensen00,varadhan04}.
The proof is PDE-based and uses geometric arguments, connecting the variational problem to weak solutions of Burgers' equation, and may be generalized to those models in the KPZ universality class that enjoy similar metric-level LDPs proven in \cite{das24}.
\end{abstract}

\maketitle

\section{Introduction}
\label{s.intro}

The directed landscape is a random directed metric that is believed to be the universal scaling limit of models in the Kardar--Parisi--Zhang (KPZ) universality class \cite{kardar86,quastel2011introduction,corwin2012kardar,quastel2015one,ganguly2021random}. 
It was constructed in \cite{DOV18} from Brownian last passage percolation and has been shown to be the scaling limit of a handful of integrable models \cite{dv22,wu2023kpz,aggarwal2024scaling}.
Let $\Rd:=\{(s,y;t,x)\in \mathbb R^4, s<t\}$.
The directed landscape $\landscape$ is a random continuous function from $\Rd$ to $\R$. 
For fixed $s<t$ and $y$, the law of $\landscape(s,y;t,\cdot)$ is given by
\begin{align}
	\label{e.pairy}
	\landscape(s,y;t,\cdot) 
	\stackrel{\text{law}}{=} 
	(t-s)^{1/3} \operatorname{Airy}_2(\,\cdot\,) -\tfrac{(\,\cdot\,-y)^2}{t-s},
\end{align}
where $\operatorname{Airy}_2$ denotes the Airy${}_2$ process \cite{prahofer2002scale,corwin2014brownian}, and when $t-s=1$ and $y=0$, the right-hand side of \eqref{e.pairy} is the parabolic Airy process. 
The time-space marginal $\landscape(0,0;\cdot,\cdot)$ is given by the KPZ fixed point with the narrow-wedge initial data \cite{matetski2021kpz,nica2020one}.

In this paper, we identify the explicit rate functions for the upper-tail LDP of certain marginals of the directed landscape and characterize the corresponding minimizing metrics.
To set up the context, we recall the metric-level LDP for the directed landscape recently proven in \cite{das24}. 
Let $\mathcal{E}$ denote the set of all continuous functions $e:\Rd\to\R$ satisfying the (reverse) triangle inequality 
\begin{align}
	\label{e.triangle}
	\metric(s,y;u,z)+\metric(u,z;t,x) \le \metric(s,y;t,x),
	\qquad
	(s,y;u,z), (u,z;t,x)\in\Rd.
\end{align}
Equip $\mathcal{E}$ with the topology of uniform convergence on bounded sets.
Given a topological space $\mathcal{X}$, we call $J:\mathcal{X}\to[0,\infty]$ a \textbf{good rate function} if $J$ is lower-semicontinuous and if $\{J\leq \alpha\}$ is pre-compact for every $\alpha<\infty$.
Scale the directed landscape as $\landscape_\eps(s,y;t,x):=\eps\landscape(s,y\eps^{-1/2};t,x\eps^{-1/2})$.
Under such scaling, the law of large numbers of $\landscape_\eps$ is given by the \textbf{Dirichlet metric}
\begin{align}
	\label{e.diri}
	\diri(s,y;t,x) := -\tfrac{(x-y)^2}{t-s}.
\end{align}
It was shown in \cite{das24} that $\{\landscape_\eps\}_{\eps}$ satisfies an LDP with speed $\eps^{-3/2}$ and an explicit rate function $\rate$ that is good.
%

The work \cite{das24} offers several ways of expressing the rate function $\rate$.
The formulation most relevant to us uses the language of measures, which we now recall.
Throughout this paper, a \textbf{path} means an $\Hsp^1$ function on a closed time interval, and we use $\start(\Peta)$ and $\ed(\Peta)$ to denote the starting and ending time of the path $\Peta$, namely $\Peta\in\Hsp^1[\start(\Peta),\ed(\Peta)]$.
Let $\graph(\Peta):=\{(t,\Peta(t)):t\in[\start(\Peta),\ed(\Peta)]\}$ denote the graph of $\Peta$.
Paths being \textbf{internally disjoint} means their graphs are disjoint except possibly at common endpoints. 
Consider the space of finite signed measures on $\R^2$ supported on countably many paths
\begin{align}
	\label{e.Msp}
	\Msp 
	&:= 
	\Big\{ \mu=\sum_{\Pgamma\in\Snet} \dens_\Pgamma\,\delta_{\Pgamma} 
	\ : \ 
	\dens_\Pgamma \in \Lsp^{3/2} [\start(\Pgamma),\ed(\Pgamma)], \
	\sum_{\Pgamma\in\Snet} \norm{\dens_\Pgamma}_{\Lsp^{3/2}[\start(\Pgamma),\ed(\Pgamma)]} < \infty \Big\},
\end{align}
where $\Snet$ is a countable set of internally disjoint paths, and $\mu$ acts on $\FF\in\Csp_\mathrm{b}(\R^2)$ by
\begin{align}
	\langle\mu,\FF\rangle := \sum_{\Pgamma\in\Snet} \int_{[\start(\Pgamma),\ed(\Pgamma)]} \d t \, \dens_\Pgamma(t)\, \FF(t,\Pgamma(t)).
\end{align}
Equivalently, $\dens_\Pgamma\,\delta_{\Pgamma}$ denotes the pushforward of $\dens_\Pgamma(t)\,\d t$ under the map $t\mapsto (t,\Pgamma(t))$.
We call $\dens_{\Pgamma}$ the \textbf{density}.
Let $\Msp_+:=\{\Mmu\in\Msp:\Mmu \geq 0\}$ denote the subspace of positive measures. 
Hereafter, write $\dot{\Peta}=\frac{\d~}{\d t} \Peta$ for the time derivative and 
\begin{align}
	\label{e.connect}
	\text{ write }(s,y)\xrightarrow{\Peta}(t,x) \text{ for ``}\Peta\text{ is a path that connects }(s,y)\text{ to }(t,x)\text{''},
\end{align}
namely $\start(\Peta)=s$, $\ed(\Peta)=t$, $\Peta(s)=y$, and $\Peta(t)=x$.
Given a $\Mmu\in \Msp_+$, define the metric $\Metric_{\Mmu}$ and the length $|\,\cdot\,|_{\Metric_{\Mmu}}$ by
\begin{align}
	\label{e.Mmu}
	\Metric_{\Mmu}(s,y;t,x) := \sup\big\{ |\Peta|_{\Metric_{\Mmu}} : (s,y)\xrightarrow{\Peta}(t,x) \big\},
	\qquad
	|\,\Peta\,|_{\Metric_{\Mmu}} := \mu(\graph(\Peta))+|\Peta|_{\diri}\,,
\end{align}
where $|\Peta|_{\diri}:=-\int_{[\start(\Peta),\ed(\Peta)]} \d t \,\dot{\eta}^2$.
It was shown in \cite{das24} that $\Mmu\mapsto\Metric_{\Mmu}$ is a bijection from $\Msp_{+}$ to $\{\metric\in\mathcal{E}: \rate(\metric)<\infty\}$.
Namely, metrics with finite rates are in bijective correspondence with measures in $\Msp_{+}$.
Further, for any such metric, the rate function can be written as follows:
\begin{align}
	\label{e.rate}
	\text{For }\mu=\sum_{\Pgamma\in\Snet} \dens_\Pgamma\,\delta_{\Pgamma},
	\qquad
	\rate(\Metric_{\Mmu}) 
	=
	\sum_{\Pgamma\in\Snet} \int_{\start(\Pgamma)}^{\ed(\Pgamma)} \d t \, \frac{4}{3} \dens_{\Pgamma}^{3/2}.
\end{align}

To avoid ambiguity, we will refer to the right-hand side of \eqref{e.rate} as the rate function.
The right-hand side of \eqref{e.rate} is called the Kruzhkov entropy in \cite{das24}, but we reserve the name Kruzhkov entropy for $\entbm(\alpha):= \alpha^2/4$ and the name Kruzhkov entropy production for $\Entp_\pm$; both defined in Section~\ref{s.burgers.kruzkov}.
As will be seen there (in particular in \eqref{e.rate.M}), $\rate$ and $\Entp_\pm$ are indeed closely related.
However, there is a \emph{subtle difference} between the two, which we will discuss in Section~\ref{s.burgers.comparing}.


\subsection{Result for wedge initial data}
\label{s.intro.wedge}
Our main result concerns the LDP for $\landscape_{\eps}(0,0;1,\cdot)$.
We view this as the time-one LDP for the KPZ fixed point under the narrow-wedge initial data, which is equivalent to the LDP for the parabolic Airy process.
Let $\intvl$ be a union of finitely many bounded closed intervals and let $\Ff:\intvl\to\R$ be a Borel function such that $\Ff(x)\geq-x^2$ on $\intvl$.
The LDP from \cite{das24} together with the contraction principle gives
\begin{align}
	\label{e.contraction.wedge}
	&\limsup_{r\to 0}\limsup_{\eps\to 0} 
	\Big|\eps^{3/2} \log \P\big[ \norm{\landscape_{\eps}(0,0;1,\cdot)- \Ff}_{\Lsp^\infty(\intvl)} < r \big] + \rate_{0,0;1,\intvl}(\Ff) \Big| = 0,
	\\
	\label{e.rate.wedge}
	&\rate_{0,0;1,\intvl}(\Ff) :=\inf\big\{ \rate(\Metric_{\Mmu}) : \Mmu\in\Msp_+,\ \Metric_{\Mmu}(0,0;1,\cdot)|_{\intvl} = \Ff \big\},
	\quad
	\inf\emptyset := \infty.	
\end{align}

Our result gives an explicit description of $\rate_{0,0;1,\intvl}(\Ff)$ and characterizes the minimizer of \eqref{e.rate.wedge}.
To set up the notation, for $c\in(0,\infty)$, consider the space $\Entspp{c}:=\{ \Fphi\in\Csp(\R) : \Fphi|_{[-c,c]}\in\Hsp^1[-c,c], \ \Fphi(x)\geq -x^2 \text{ on } \R, \ \Fphi(x)=-x^2 \text{ outside of } [-c,c] \}$, let $\Entsp := \cup_{c\in(0,\infty)} \Entspp{c}$, and
\begin{align}
	\label{e.Engbm}
	\Engbm(\intvl,\Ff) 
	:= 
	\inf\Big\{ \frac{1}{4} \int_{\R} \d x \, \big( ( \partial_x \Fphi \,)^2 - 4x^2 \big) : \Fphi\in\Entsp, \ \Fphi|_{\intvl}=\Ff \Big\},
	\qquad
	\inf\emptyset := \infty.
\end{align}
When $\Engbm(\intvl,\Ff)<\infty$, the infimum in \eqref{e.Engbm} has a unique minimizer $\Ff_*\in\Entsp$, which is depicted in Figure~\ref{f.Ff*} and described in Section~\ref{s.pfmain.Engbm}.
Generalizing the notation in \eqref{e.connect}, for $B\subset\R^2$, we write $(s,y)\xrightarrow{\Peta}B$ for ``$\Peta$ is a path that connects $(s,y)$ to a point in $B$'', and similarly for $A\xrightarrow{\Peta}(t,x)$. 
Take the $\Ff_*(x)$ as the terminal data at $t=1$ and construct $\Fh_*(t,x)$, for $(t,x)\in(0,1]\times\R$, via the backward Hopf--Lax operator as
\begin{align}
	\label{e.Fh*}
	\Fh_*(t,x)
	:=
	\HLbk_{1\to t}\big[\Ff_*](x)
	:=
	\inf\big\{ -|\Pchi|_{\diri} + \Ff_*(\Pchi(1)) : (t,x)\xrightarrow{\Pchi}\{1\}\times\R \big\}.
\end{align}
\begin{thm}\label{t.main}
\begin{enumerate}
\item[]
\item \label{t.main.inf}
Let $\intvl$, $\Ff$, $\Engbm(\intvl,\Ff)$, $\Ff_*$, and $\Fh_*$ be as above, then
\begin{align}
	\label{e.t.main}
	\inf\big\{ \rate(\Metric_{\Mmu}) : \Mmu\in\Msp_+,\ \Metric_{\Mmu}(0,0;1,\cdot)|_{\intvl} = \Ff \big\} = \Engbm(\intvl,\Ff),
\end{align}
and when this is finite, any minimizer satisfies $ \Metric_{\Mmu}(0,0;t,x) = \Fh_*(t,x)$ on $(t,x)\in(0,1]\times\R$.
\item \label{t.main.measure}
When $\Ff$ is continuous, and piecewise linear or quadratic, the minimizer of the left hand side of \eqref{e.t.main} is unique and given by $\Mmu_*=\M[\Fh_*]$, where $\M[\,\cdot\,]$ is defined later in \eqref{e.M}.
\end{enumerate}
\end{thm}
\noindent
Under the setup of Theorem~\ref{t.main}\eqref{t.main.measure}, for the optimal profile $\Fh_*$, we prove that $\partial_x\Fh_*$ is a weak solution of the inviscid Burgers' equation with only non-entropic shocks (see Section \ref{s.burgers.burgers} for definitions) and the planted network metric $\Mmu_*$ is precisely supported on these shocks. It then follows from the analysis in \cite{das24} that the non-entropic shocks are $\Metric_{\Mmu_*}$ geodesics.

Theorem~\ref{t.main} has a corollary about the limit shape of the directed landscape.
\begin{cor}\label{c.main}
Notation as in Theorem~\ref{t.main} and let $\EventT(r):=\{ \norm{\landscape_\eps(0,0;1,\cdot)-\Ff}_{\Lsp^\infty(\intvl)} < r \}$.
\begin{enumerate}
\item \label{c.main.inf}
If $\Engbm(\intvl,\Ff)<\infty$, then for any compact $\SA\subset(0,1]\times\R$ and any $r'>0$,
\begin{align}
	\limsup_{r\to 0} \limsup_{\eps\to 0} 
	\eps^{-3/2} \log \P\Big[ \,\norm{\landscape_\eps(0,0;\cdot,\cdot)-\Fh_*}_{\Lsp^\infty(\SA)} \geq r' \Big| \EventT(r) \Big] < 0.
\end{align}
\item \label{c.main.measure}
If $\Ff$ is continuous, and piecewise linear or quadratic, for any compact $\SB\subset\R^4_{\uparrow}$ and $r'>0$,
\begin{align}
	\limsup_{r\to 0} \limsup_{\eps\to 0} 
	\eps^{-3/2} \log \P\Big[ \,\norm{\landscape_\eps-\Metric_{\Mmu_*}}_{\Lsp^\infty(\SB)} \geq r' \Big| \EventT(r) \Big] < 0.
\end{align}
\end{enumerate}
In words, conditioned on $\EventT(r)$ with $r\to 0$, Part~\ref{c.main.inf} asserts that $\landscape_\eps(0,0;\cdot,\cdot)$ concentrates around $\Fh_*$, and Part~\ref{c.main.measure} asserts that $\landscape_\eps$ concentrates around $\Metric_{\Mmu_{*}}$.
Both concentrations hold up to exponentially small probabilities and under their respective assumptions on $\Ff$.
\end{cor}


Our method is PDE-based and demonstrates how the above LDP problem fits into the broader LDP theory proposed by Jensen and Varadhan \cite{jensen00,varadhan04} in the context of totally asymmetric exclusion processes.
Our result further gives the limit shape of the landscape under the conditioning.
To the best of our knowledge, the PDE-based method is the only method that can characterize and prove the limit shape under the scope we consider here.
(There are some special cases where the minimizer $\Mmu_*$ is more tractable, as discussed in the next paragraph.)
The PDE formalism is not specific to the directed landscape, and one can consider generalizing this PDE-based method to other models in the KPZ universality class.
Of course, doing so requires establishing metric-level LDPs similar to that of \cite{das24}.
We note in passing that, to prove just the upper-tail LDP for the parabolic Airy process, which is the first half of Theorem~\ref{t.main}\eqref{t.main.inf}, there exists a different strand of methods based on the Brownian resampling property \cite{corwin2014brownian,ganguly2022sharp,dauvergne2023wiener}, but these methods do not seem to give the limit shape (except for the special cases where $\Mmu_*$ is more tractable).
%


We emphasize that the minimizing $\Mmu_*$ in Theorem~\ref{t.main}\eqref{t.main.measure} is generally \emph{not piecewise linear}.
When $\intvl$ consists of finitely many points and when $\Ff_*$ is concave, the minimizing $\Mmu_*$ is supported on the graph of finitely many linear paths $\Pgamma$ with constant $\dens_{\Pgamma}$.
This follows from \cite[Proposition~2.1]{das24}; see \cite{tsai2023high,lin23} for a similar result for the KPZ equation.
Roughly speaking, for such special configurations, the infimum in \eqref{e.t.main} can be solved by hand, giving the piecewise linear $\Mmu_*$ just described.
Beyond these special configurations, the infimum in \eqref{e.t.main} cannot be solved by hand (as far as we can tell), the support of $\Mmu_*$ generally consists of nonlinear paths $\Pgamma$, and the density $\dens_{\Pgamma}$ is generally nonlinear.
Example~\ref{ex.M} gives one such instance.

Let us describe our strategy of proving Theorem~\ref{t.main}.
First, proving Theorem~\ref{t.main} amounts to solving the variational problem in \eqref{e.t.main}.
For a suitable $\Mmu$, the function $\partial_{x}\Metric_{\Mmu}(0,0,t,x)$ is a weak solution of (the inviscid) Burgers' equation.
This is pointed out in \cite[Section~1.1]{das24}; see also \cite{bakhtin2013burgers}.
In this paper, we only consider a class of piecewise smooth weak solutions of Burgers' equation, which we call tractable solutions and define in Definition~\ref{d.tractable}.
This class of solutions, together with various approximation arguments, suffices for our proof.
Given what we said above, the variational problem in \eqref{e.t.main} can be written in terms of tractable solutions.
The key to solving this variational problem is an identity, stated in Proposition~\ref{p.key}, that relates the Kruzhkov entropy productions to $\Engbm$.
This identity gives a strong hint on what the minimizer of \eqref{e.t.main} should look like; see Section~\ref{s.burgers.proof} for a brief discussion on this.
A version of this identity appeared in \cite[Proposition~2.6]{quastel21}, but it was only considered for a special weak solution there.
Here, we observe that the identity generalizes to all tractable solutions and use it to solve the variational problem.
From a broader perspective, our proof is related to the Jensen--Varadhan picture of hydrodynamic large deviations \cite{jensen00,varadhan04}.
A major open problem under this picture is to approximate the Kruzhkov entropy production of general weak solutions of Burgers' equation.
This problem motivates or is related to several works in analysis, including \cite{DeLellis2003,DeLellis2003a,DeLellis2004,Golse2013,Lamy2018}.
Our work, on the other hand, focuses on connecting the directed landscape to Burgers' equation and does not touch on problems like this.

Hereafter, we will implicitly assume that $\Mmu\in\Msp_{+}$ is supported in $[0,1]\times\R$.
Doing so does not lose any generality, because Theorem~\ref{t.main} concerns the deviations of $\landscape_{\eps}(0,0;1,\cdot)$, and because $\Metric_{\Mmu}(0,0;1,\cdot)$ remains unchanged even if we redefine $\Mmu\in\Msp_{+}$ to be zero outside of $[0,1]\times\R$.
Accordingly, unless otherwise noted, \emph{time variables such as $s,t$ will be assumed to be within $[0,1]$ for the rest of the paper.}

\begin{figure}
\begin{minipage}[t]{.48\linewidth}
\fbox{\includegraphics[width=\linewidth]{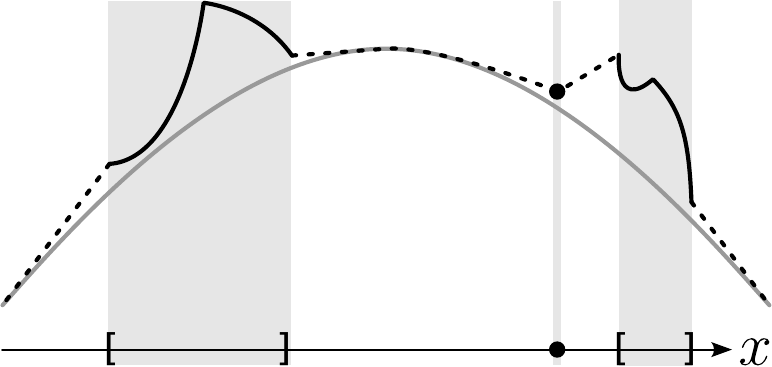}}
\caption{
	The set $\intvl$ is shown at the bottom.
	The gray curve is $-x^2$.
	The solid black curve is $\Ff$.
	The dashed and solid black curves together make $\Ff_*$.
}
\label{f.Ff*}
\end{minipage}
\hfill
\begin{minipage}[t]{.48\linewidth}
\fbox{\includegraphics[width=\linewidth]{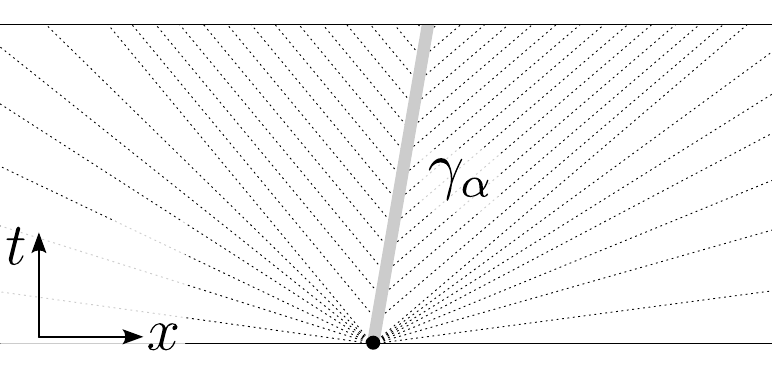}}
\caption{Example~\ref{ex.single}.
	The dot is $(0,0)$.
	The dashed paths are the characteristics of $\partial_{x}\Fh_*$.
	The gray path is $\Pgamma_{\alpha}$, which traces out the support of $\Mmu_{*}$.
}
\label{f.ex.single1}
\end{minipage}
\end{figure}

\subsection{Examples}
\label{s.intro.examples}
Here are a few examples to illustrate Theorem~\ref{t.main}\eqref{t.main.measure}.

\begin{ex}\label{ex.single}
Take $\intvl=\{\alpha\}$ and $\Ff(\alpha)=\beta-\alpha^2$, where $\alpha\in\R$ and $\beta>0$.
Since $\intvl$ consists of just one point, the assumption of Theorem~\ref{t.main}\eqref{t.main.measure} holds. 
One can solve explicitly 
\begin{align}
	\label{e.FA}
	\Ff_*(x) 
	&=
	\left\{\begin{array}{l@{,\ }l}
			- 2(\alpha-\sqrt{\beta}) x + (\alpha-\sqrt{\beta})^2 & \text{when } x\in[\alpha-\sqrt{\beta},\alpha]
			\\
			- 2(\alpha+\sqrt{\beta}) x + (\alpha+\sqrt{\beta})^2 & \text{when } x\in[\alpha,\alpha+\sqrt{\beta}]
			\\
			-x^2 & \text{when } |x-\alpha|>\sqrt{\beta}
	\end{array}\right\}
	=:
	\FA_{\alpha,\beta}(1,x),
\\
	\label{e.FA.}
	\Fh_*(t,x) &= t\FA_{\alpha,\beta}\big(1,\tfrac{x}{t}\big) =: \FA_{\alpha,\beta}(t,x),
\\
	\Mmu_{*} &= \beta\delta_{\Pgamma_{\alpha}},
	\text{ where }\Pgamma_{\alpha}(t):=\alpha t.
\end{align}
\end{ex}

Let us explain a way to visualize $\partial_{x}\Fh_*$ and $\Mmu_{*}$, which will be used in subsequent examples.
As will be seen in Section~\ref{s.tools.hopflax}, the derivative $\partial_x\Fh_*$ is a (non-entropy) solution of Burgers' equation, so it can be described by characteristics.
A characteristic of $\partial_x\Fh_*$ is a path $\Pchi$ such that
\begin{align}
	\label{e.ex.single.chara}
	\dot{\Pchi}(t)=-\tfrac{1}{2}(\partial_x\Fh_*)(t,\Pchi(t)) =  \text{constant}.
\end{align}
For Example~\ref{ex.single}, using \eqref{e.FA}--\eqref{e.FA.}, one sees that the characteristics are as depicted in Figure~\ref{f.ex.single1}.
The places where characteristics intersect are shocks, and by \eqref{e.ex.single.chara}, the function $\partial_x\Fh_*$ is not continuous at shocks.
The precise definition of $\Mmu_* := \M[\Fh_*]$ will be given in \eqref{e.M} in Section~\ref{s.burgers.kruzkov}.
Roughly speaking, $\Mmu_*$ is constructed to be supported on shocks and to have density being $1/16$ of the square of the jump of $\partial_x\Fh_*$:
\begin{align*}
	\Mmu_* := \M[\Fh_*] := \sum_{\Pl:\text{ shocks}} \frac{1}{16} \big( (\partial_{x}\Fh_*)_{x=\Pl(t)^-} - (\partial_{x}\Fh_*)_{x=\Pl(t)^+} \big)^2\, \delta_{\Pl}.
\end{align*}
For Example~\ref{ex.single}, the shocks trace out the path $\Pl=\Pgamma_\alpha$ and $(\partial_{x}\Fh_*)_{x=\Pl(t)^\pm}=-2(\alpha\pm\sqrt{\beta})$.
Hence $\Mmu_*=\frac{1}{16}(4\sqrt{\beta})^2\,\delta_{\Pgamma_{\alpha}}=\beta\,\delta_{\Pgamma_{\alpha}}$.

\begin{ex}\label{ex.M}
Take $\intvl=\{-\alpha,0,\alpha\}$, $\Ff(\pm\alpha):=\beta-\alpha^2$, and $\Ff(0):=\beta'$, for some $\alpha,\beta,\beta'>0$ such that $\beta'<\beta-\alpha^2$.
In this case, $\Ff_*$ takes an M shape as depicted in Figure~\ref{f.ex.M1}.
The characteristics of $\partial_x\Fh_*$ can be obtained by the ``shear-and-cut'' procedure described in Section~\ref{s.tools.hopflax} and the result gives either plot in Figure~\ref{f.ex.M2}, depending on the values of $\alpha,\beta,\beta'$.
The measure $\Mmu_{*}$ is then read off from the shocks of $\partial_{x}\Fh_*$, as explained previously and shown in Figure~\ref{f.ex.M2}.
\end{ex}

\begin{ex}\label{ex.lift}
Take $\intvl=[-\alpha,\alpha]$ and $\Ff(x):=-x^2+\beta$ on $\intvl$, for some $\alpha,\beta>0$.
In words, we seek to lift $\landscape_{\eps}(0,0;1,\cdot)|_{[-\alpha,\alpha]}$ up by $\beta$.
In this case, $\Ff_*$ is as depicted in Figure~\ref{f.ex.lift1}.
The characteristics of $\partial_x\Fh_*$ can be obtained by the ``shear-and-cut'' procedure described in Section~\ref{s.tools.hopflax} and the result gives either plot in Figure~\ref{f.ex.lift2}, depending on the values of $\alpha,\beta$.
The measure $\Mmu_{*}$ is then read off from the shocks of $\partial_{x}\Fh_*$; see Figure~\ref{f.ex.lift2}.
\end{ex}

\begin{figure}
\begin{minipage}[t]{.345\linewidth}
\fbox{\includegraphics[width=\linewidth]{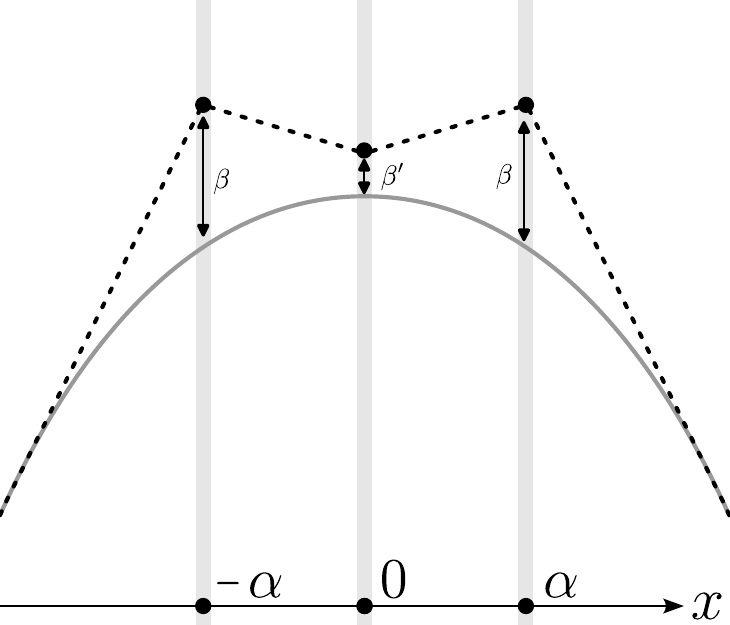}}
\caption{
Example~\ref{ex.M}.
The set $\intvl=\{-\alpha,0,\alpha\}$ is shown at the bottom.
The gray curve is $-x^2$. 
The three dots above the gray curve represent $\Ff$.
The dashed curve is $\Ff_*$.
}
\label{f.ex.M1}
\end{minipage}
\hfill
\begin{minipage}[t]{.63\linewidth}
\fbox{\includegraphics[width=.47\linewidth]{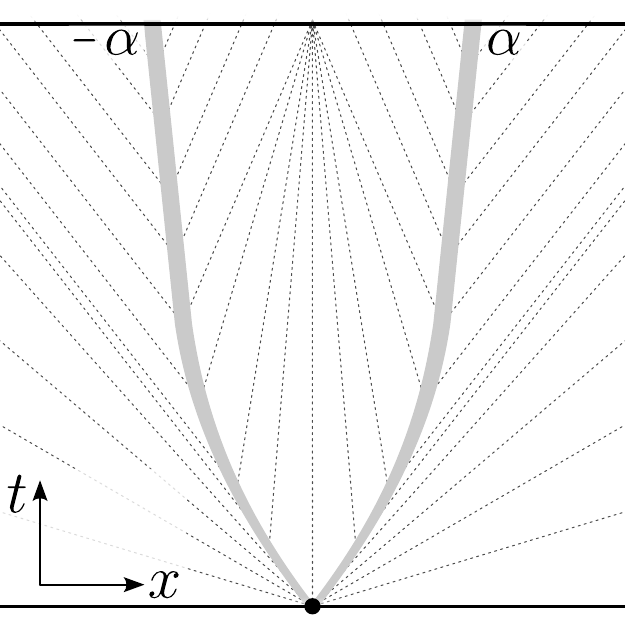}}\hfill\fbox{\includegraphics[width=.47\linewidth]{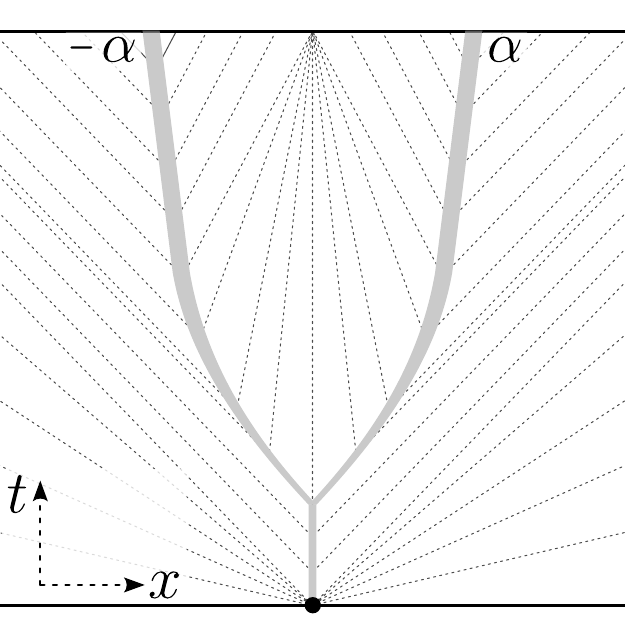}}
\caption{Example~\ref{ex.M}.
Possible configurations of the characteristics and $\Mmu_{*}$.
The dot is $(0,0)$.
The dashed paths are the characteristics of $\partial_{x}\Fh_*$.
The gray paths trace out the support of $\Mmu_{*}$, and the thickness of the paths indicates the size of the density of $\Mmu_{*}$ along those paths.
}
\label{f.ex.M2}
\end{minipage}
\end{figure}

\begin{figure}
\begin{minipage}[t]{.345\linewidth}
\fbox{\includegraphics[width=\linewidth]{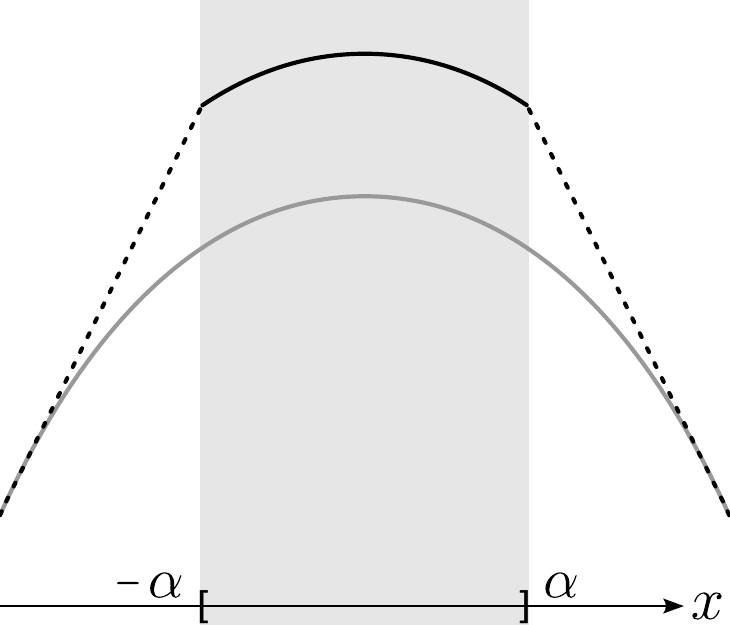}}
\caption{
	Example~\ref{ex.lift}.
	The set $\intvl=[-\alpha,\alpha]$ is shown at the bottom.
	The gray curve is $-x^2$. 
	The solid black curve is $\Ff$.
	The dashed and solid black curves together make $\Ff_*$.
}
\label{f.ex.lift1}
\end{minipage}
\hfill
\begin{minipage}[t]{.63\linewidth}
\fbox{\includegraphics[width=.47\linewidth]{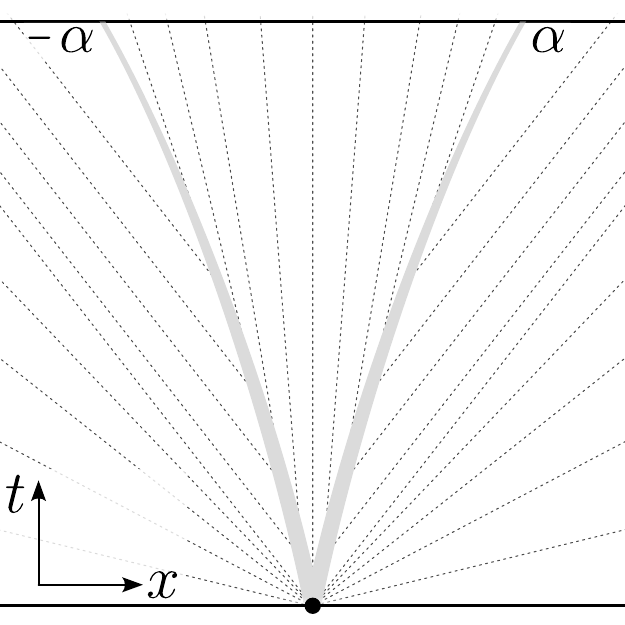}}\hfill\fbox{\includegraphics[width=.47\linewidth]{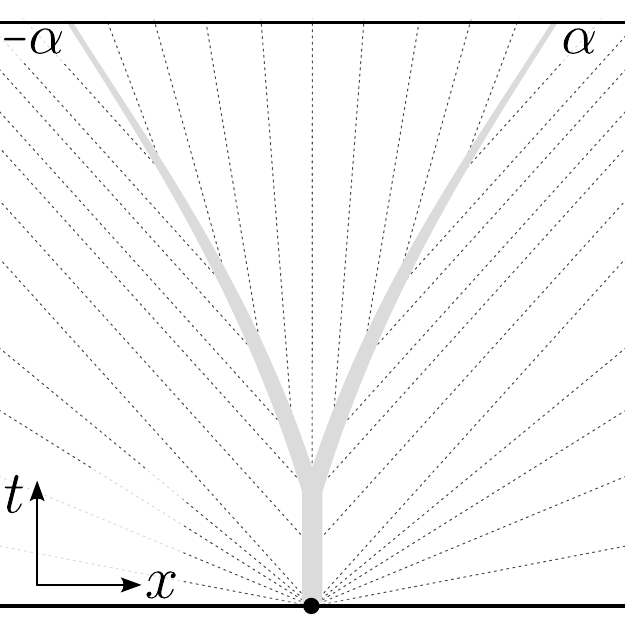}}
\caption{Example~\ref{ex.lift}.
	Possible configurations of the characteristics and $\Mmu_{*}$.
	The dot is $(0,0)$.
	The dashed paths are the characteristics of $\partial_{x}\Fh_*$.
	The gray paths trace out the support of $\Mmu_{*}$, and the thickness of the paths indicates the size of the density of $\Mmu_{*}$ along those paths.
}
\label{f.ex.lift2}
\end{minipage}
\end{figure}

\subsection{Result for multi-wedge initial data}
\label{s.intro.multi}
We next consider the multi-wedge initial data.
Take any nonempty \emph{finite set} $\SZ\subset\R$, a function $\Fg:\SZ\to\R$, and consider
\begin{align}
	\label{e.fixedpt}
	\fixedpt_{\eps}^{\SZ,\Fg}(t,x)
	:=
	\max_{z\in\SZ} \big\{ \landscape_{\eps}(0,z;t,x) + \Fg(z) \big\}.
\end{align}
This is the $\eps$-scaled KPZ fixed point with initial data $\Fg\ind_{\SZ} -\infty\ind_{\R\setminus\SZ}$, which we call multi-wedge.
We will only consider the finite-dimensional LDPs for $\fixedpt_{\eps}^{\SZ,\Fg}(1,\cdot)$.
Fix any nonempty \emph{finite set} $\SY\subset\R$ and $\Ff:\SY\to\R$ such that $\Ff(y)>\max_{z\in\SZ}\{ -(y-z)^2 + \Fg(z) \}$, for all $y\in\SY$.
The LDP from \cite{das24} together with the contraction principle gives
\begin{align}
	&
	\label{e.contraction.multi}
	\limsup_{r\to 0}\limsup_{\eps\to 0}\Big|
	\eps^{3/2} \log \P\big[ \norm{\fixedpt_{\eps}^{\SZ,\Fg}(1,\cdot)-\Ff}_{\lsp^\infty(\SY)}<r \big]
	+
	\rate_{0,\SZ;1,\SY}(\Fg;\Ff)
	\Big|
	=0,
\\
	&
	\label{e.rate.multi}
	\rate_{0,\SZ;1,\SY}(\Fg;\Ff)
	:=
	\inf\Big\{ \rate(\Metric_{\Mmu}) : \max_{z\in\SZ}\big\{ \Metric_{\Mmu}(0,z;1,y)+\Fg(z) \big\} = \Ff(y),\ \forall y\in\SY \Big\}.
\end{align}

Our result gives a decomposition of \eqref{e.rate.multi}.
Call $(\SY_{z})_{z\in\SZ}$ an \textbf{ordered partition of $\SY$} if $\{\SY_{z}\}_{z\in\SZ}$ is a partition of $\SY$ and $y<y'$ holds for all $(y,y')\in\SY_{z}\times\SY_{z'}$ with $z<z'$.
\begin{prop}\label{p.multi}
The infimum \eqref{e.rate.multi} is achieved by measures of the form $\Mmu = \sum_{z\in\SZ} \Mmu_{z}$ such that, for some ordered partition $(\SY_{z})_{z\in\SZ}$ of $\SY$,
\begin{enumerate}
	\item \label{p.mult.disjoint}
	the $\Mmu_{z}$s have disjoint supports and are supported on trees connecting $\SY_z$ to $z$,
	\item \label{p.mult.=}
	$\Metric_{\Mmu_{z}}(0,z;1,y)+\Fg(z)=\Ff(y)$ for all $z\in\SZ$ and $y\in\SY_{z}$, and
	\item \label{p.mult.<}
	$\Metric_{\Mmu_{z}}(0,z;1,y)+\Fg(z)\leq\Ff(y)$ for all $y\in\SY$ and $z\in \SZ$.
\end{enumerate}
\end{prop}
\noindent%
We mention the result \cite[Lemma~3.8]{rb21} in relation to the last proposition.
One could try to prove a multi-wedge analog of \eqref{e.t.main} using Theorem~\ref{t.main} and Proposition~\ref{p.multi}. 
Indeed, Condition~\eqref{p.mult.=} alone is equivalent to the setting of Theorem~\ref{t.main} with $\intvl=\SY_{z}$, $\Ff\mapsto\Ff-\Fg(z)$, and with the spatial origin being shifted to $z$.
Taking into account Conditions~\eqref{p.mult.disjoint} and \eqref{p.mult.<} requires further work.
We do not pursue this and just state a conjecture. 
\begin{conj}\label{conj.multi}
Let $\Entsp_z(c):=\{ \Fphi\in\Csp(\R)  :  \Fphi|_{[-c,c]}\in\Hsp^1[-c,c]; \ \Fphi(x)\geq -(x-z)^2+\Fg(z) \text{ on } \R; \ \Fphi(x)=-(x-z)^2+\Fg(z) \text{ outside of } [-c,c] \}$ and let $\Entsp_z := \cup_{c\in(0,\infty)} \Entsp_z(c)$.
The infimum in \eqref{e.rate.multi} is equal to
\begin{align}
    \label{e.conj.multi}
	\inf\Big\{ \sum_{z\in\SZ} \frac{1}{4} \int_{\R} \d x \Big( (\partial_x\Fphi_z)^2 - 4(x-z)^2 \Big) :
	(\Fphi_z)_{z\in\SZ}\in\prod_{z\in\SZ}\Entsp_{z}; \
	\Big(\max_{z\in\SZ} \Fphi_z\Big)\Big|_{\SY} = \Ff
	\Big\}.
\end{align}
\end{conj}
\noindent%
The analog of Conjecture~\ref{conj.multi} for TASEP is proven in \cite[Section~2.3]{quastel21} with the aid of determinantal formulas.
We outline a potential way to prove $\eqref{e.rate.multi}\leq\eqref{e.conj.multi}$, which seems to be the easier inequality given the results of the present paper.
Take any minimizing measure $\Mmu = \sum_{z\in\SZ} \Mmu_{z}$ and ordered partition $(\SY_{z})_{z\in\SZ}$ as in Proposition~\ref{p.multi}.
For each $z\in Z$ with $Y_z\neq\emptyset$, let $\Ff_{*z}$ be the (unique) minimizer of \eqref{e.Engbm} for $\intvl=Y_z$ and $-4x^2\mapsto -4(x-z)^2$ and let $\Fh_{*z}$ be given by $\Ff_{*z}$ similarly to \eqref{e.Fh*}.
If one can show that $\Mmu_z$ minimizes the left hand side of \eqref{e.t.main} for $\Ff=\Ff_{*z}$ and $\Metric_{\Mmu}(0,0;1,\cdot)|_{\intvl}\mapsto\Metric_{\Mmu}(0,0;1,\cdot)|_{Y_z}+g(z)$, by Theorem~\ref{t.main}\eqref{t.main.measure} it would follow that $\Fh_{*z}=\M[\Mmu_z]$, which together with Proposition~\ref{p.multi}\eqref{p.mult.=}--\eqref{p.mult.<} would give $(\max_{z}\Fh_{*z}(1))|_{Y}=(\max_{z}\Ff_{*z})|_Y=\Ff$.
Also, by Theorem~\ref{t.main}\eqref{t.main.inf}, $\rate(\metric_{\Mmu_{z}})=\frac{1}{4} \int_{\R} \d x \,((\partial_x\Ff_{*z})^2 - 4(x-z)^2)$, and by the disjointness in Proposition~\ref{p.multi}\eqref{p.mult.disjoint}, $\rate(\metric_{\Mmu})=\sum_{z\in Z}\rate(\metric_{\Mmu_{z}})$.%

\subsection{Some related literature}
\label{s.intro.literature}
For the directed landscape and KPZ fixed point, the one-point tail bounds follow from Tracy--Widom's original work \cite{tracy94}; see also \cite{ramirez2011beta} for a probabilistic derivation. 
The metric-level LDP in \cite{das24} was established using only one-point tail bounds and basic properties of the directed landscape.
At a finer level, the conditional limiting fluctuations and geodesics were studied in \cite{liu2022geodesic,liu2024conditional} using exact formulas from \cite{liu2022one}, and in \cite{ganguly2022sharp,ganguly2023brownian} (also for the KPZ equation) using geometric arguments relying on the Brownian resampling property from \cite{corwin2014brownian,corwin2016kpz}.


For the totally asymmetric simple exclusion process, or equivalently the exponential last passage percolation, the one-point LDPs were obtained in \cite{ls77,sep98a,sep98b,deuschel99,joh00}, the process-level upper-tail LDP was studied in \cite{jensen00,varadhan04} and proven in \cite{quastel21}, and the process-level lower-tail LDP was studied in \cite{ot19}.
The lower-tail LDP for the first passage percolation was established in \cite{bgs17} at the one-point level and recently in \cite{verges2024large} at the metric level.
For the KPZ equation, one-point tails were studied in the physics works \cite{ledoussal16long,ledoussal16short,krajenbrink17short,sasorov17,corwin18,krajenbrink18half,krajenbrink18simple,krajenbrink18systematic}
and in the mathematics works \cite{das21,kim21,corwin20lower,corwin20general,ghosal20,lin21half,cafasso22,tsai22exact},
and process-level large deviations were studied in the physics works
\cite{kolokolov07,kolokolov09,meerson16,meerson17,meerson18,kamenev16,hartmann19,krajenbrink21,krajenbrink22flat,krajenbrink23}
and the mathematics works \cite{lin21,gaudreaulamarre2023kpz,lin22,ganguly2022sharp,ganguly2023brownian,lin23,tsai2023high}.
For several other models in the KPZ universality class, the one-point LDPs have been established in \cite{georgiou13,janjigian15,emrah2017,janjigian19,das2022upper,das2023large}.


\subsection*{Outline}
In Section~\ref{s.burgers}, we describe the connection between the large deviations of the directed landscape and the weak solutions of Burgers' equation. 
In Section \ref{s.tools}, we prepare various tools and intermediate results necessary for the proof.
Finally, in Section~\ref{s.pfmain}, we prove Theorem \ref{t.main}, Corollary \ref{c.main}, and Proposition \ref{p.multi}.

\subsection*{Acknowledgments}
We thank the anonymous reviewers for their very careful reading and useful comments that improve the presentation. 

\subsection*{Funding}
The research of Tsai was partially supported by the NSF through DMS-2243112 and  the Alfred P Sloan Foundation through the Sloan Research
Fellowship FG-2022-19308.

\section{Burgers' equation and Kruzhkov entropy production}
\label{s.burgers}
Here we explain the relation between the large deviations of the directed landscape, the weak solutions of Burgers' equation, and the Kruzhkov entropy production.

\subsection{Weak solutions of Burgers' equation}
\label{s.burgers.burgers}
Let us briefly recall some general properties of Burgers' equation.
We refer to \cite[Chapters~3 and 11]{evans22} for an introduction to this topic.
Consider (the inviscid) Burgers' equation and its Hamilton--Jacobi equation:
\begin{align}
	\label{e.burgers}
	\partial_t \Fu = \tfrac{1}{4} \partial_x (\Fu^2),\qquad \partial_t \Fh = \tfrac{1}{4} (\partial_x \Fh)^2,
	\qquad
	(t,x)\in(0,1)\times\R,
\end{align}
which are related through $\Fu=\partial_x\Fh$.
We view $\Fh$ as the height function.
Call $\Fu\in\Lsp_\mathrm{loc}^\infty((0,1)\times\R)$ a weak solution of Burgers' equation if the first equation in \eqref{e.burgers} holds weakly on $(0,1)\times\R$.
No matter how smooth we require the initial data $\Fu(0,\cdot)$ to be, the solutions generally develop discontinuities later in time, and there generally exist multiple weak solutions.
Uniqueness is obtained only after imposing a so-called entropy condition.
For a given initial data in a suitable class, there exists exactly one weak solution satisfying the entropy condition, called the entropy solution.
While in the PDE literature, non-entropy solutions are often regarded as non-physical, in the study of the large deviations of the directed landscape, non-entropy solutions are \emph{highly relevant}.
The relevance of non-entropy solutions in this context is observed as early as in \cite{jensen00} and \cite{varadhan04}.

In this paper, we will only work with what we call tractable solutions.
As said, we will need to consider \emph{non-entropy} solutions.
To introduce tractable solutions, recall the Dirichlet metric $\diri$ from \eqref{e.diri} and consider
\begin{align}
	\label{e.Fhh.Fuu}
	\Fhh(t,x) := \diri(0,0;t,x) = - \tfrac{x^2}{t},
	\quad
	\Fuu(t,x) := \partial_x\Fhh(t,x) = -\tfrac{2x}{t},
	\quad
	(t,x)\in(0,1]\times\R.
\end{align}
These functions are smooth solutions of \eqref{e.burgers}.
They do become singular as $t\to 0$, but we will confine our attention to a light cone 
\begin{align}
	\label{e.cone}
	\cone(c) :=\{(t,x):t\in[0,1],|x| \leq ct\},
\end{align}
for $c<\infty$, on which $|\Fhh|$ and $|\Fuu|$ stay bounded.
\begin{defn}\label{d.tractable}
A \textbf{tractable partition of $\cone(c)$} consists of $(\Stime,\Sbdy,\Sdom)$, standing respectively for time points, boundary paths, and domains.
The set $\Stime=\{0=t_0<\ldots<t_{|\Stime|}=1\}$ is a partition of $[0,1]$.
The set $\Sbdy$ consists of finitely many $\Csp^1[t_{i-1},t_{i}]$ paths, $i=1,\ldots,|\Stime|$, and these paths do not intersect except at $t\in\Stime$.
We assume $\Sbdy$ contains the left and right boundaries of $\cone(c)\cap([t_{i-1},t_{i}]\times\R)$,  $i=1,\ldots,|\Stime|$, where by boundaries we mean parameterized curves.
The light cone $\cone(c)$ is partitioned by $\Stime\times\R$ and the paths in $\Sbdy$ into finitely many open connected regions.
Let $\Sdom$ denote the set of these open connected regions.
See Figure~\ref{f.tractable} for an illustration.

Call $\Fv$ with $\Fv:=\partial_x\Fg$ a \textbf{tractable solution} with respect to a tractable partition $(\Stime(\Fv),\Sbdy(\Fv),\Sdom(\Fv))$ of a light cone $\cone(c)$ if $\Fg\in\Csp((0,1]\times\R)$ solves $\partial_t\Fg=\frac{1}{4}(\partial_{x}\Fg)^2$ a.e.\ on $(0,1)\times\R$ and
\begin{enumerate}
	\item \label{d.tractable.cone}
	$\Fg=\Fhh$ outside of $\cone(c)$ and $\Fg\in\Csp(\cone(c))$ with the convention $\Fg(0,0):=0$,
	\item \label{d.tractable.domain}
	for each $\SD\in\Sdom(\Fv)$, we have $\Fg\in\Csp^2(\SD)$ and $\Fv\in\Lsp^\infty(\SD)$, and
	\item \label{d.tractable.bdy}
	for each $\Pl\in\Sbdy(\Fv)$, the limits $\Fv_{\Pl\pm}(t):=\lim_{x\to\Pl(t)^\pm}\Fv(t,x)$ exist for $t\in(\start(\Pl),\ed(\Pl))$ and are $\Csp(\start(\Pl),\ed(\Pl))\cap \Lsp^\infty(\start(\Pl),\ed(\Pl))$ functions.
\end{enumerate}
Call any $\Pl\in\Sbdy(\Fv)$ a \textbf{boundary path of $\Fv$} and any $\SD\in\Sdom(\Fv)$ a \textbf{smooth region of $\Fv$}.
\end{defn}

\begin{figure}
\fbox{\includegraphics[width=.5\linewidth]{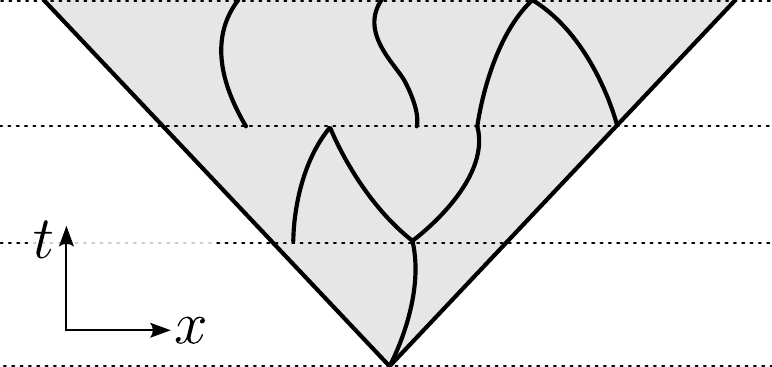}}
\caption{Tractable partition of a light cone.
	The dashed lines are $\Stime\times\R$.
	The solid paths are $\Pl\in\Sbdy$.
}
\label{f.tractable}
\end{figure}

Let us recall some basic properties of weak solutions of Burgers' equation, restricting the scope to a tractable solution $\Fv=\partial_x\Fg$.
Recall that a \textbf{characteristic} of $\Fv$ is a linear path $\Pchi$ such that
\begin{align}
	\label{e.characteristics}
	\dot{\Pchi}=-\tfrac{1}{2}\Fv|_{\Pchi} =  \text{constant}.
\end{align}
Hereafter, $\Fv|_{\Pchi}(t):=\Fv(t,\Pchi(t))$.
For any $\SD\in\Sdom(\Fv)$, since $\Fv$ is a $\Csp^1(\SD)$ solution, the region $\SD$ is filled with non-intersecting characteristics.
Next, recall that the places where $\Fv$ is discontinuous are called \textbf{shocks}, which can only occur on the boundary paths $\Pl\in\Sbdy(\Fv)$.
For the discussions of shocks, we exclude the starting and ending times $t=\start(\Pl),\ed(\Pl)$ when the limits $\Fv_{\Pl\pm}$ may not exist.
An \textbf{entropy shock} and \textbf{non-entropy shock} are respectively where $\Fv_{\Pl-}<\Fv_{\Pl+}$ and $\Fv_{\Pl-}>\Fv_{\Pl+}$.
We call $\Fv$ an entropy solution if it does not have non-entropy shocks.
Next, we derive some useful identities.
Since $\Fg\in\Csp((0,1]\times\R)$ by assumption, 
\begin{align}
	\label{e.Fg.conti}
	\lim_{x\to\Pl(t)^-} \Fg(t,x) = \Fg(t,\Pl(t)) = \lim_{x\to\Pl(t)^+} \Fg(t,x),
	\qquad
	\text{for all } t\in(\start(\Pl),\ed(\Pl)).
\end{align}
Take $t_1<t_2\in(\start(\Pl),\ed(\Pl))$, use \eqref{e.Fg.conti} to approximate $\Fg(t_i,\Pl(t_i))$ by $\Fg(t_i,\Pl(t_i)\pm\eps)$, use the $\Csp^1$ smoothness of the latter, and send $\eps\to 0^+$.
Doing so gives
$
\Fg(t_2,\Pl(t_2)) - \Fg(t_1,\Pl(t_1))
=
\int_{t_1}^{t_2} \d t \, ( \partial_{t} \Fg + \dot{\Pl}\partial_x \Fg  )|_{\Pl\pm}
=
\int_{t_1}^{t_2} \d t \, ( \frac{1}{4}\Fv^2 + \dot{\Pl}\Fv )|_{\Pl\pm}	 
$. 
Hence $\frac{\d~}{\d t}\Fg(t,\Pl(t))$ exists on $(\start(\Pl),\ed(\Pl))$ and 
\begin{align}
	\label{e.Fg.dot.}
	\dot{\Fg|_{\Pl}}
	:=
	\tfrac{\d~}{\d t} \Fg(t,\Pl(t))
	=
	\tfrac{1}{4} \Fv_{\Pl+}^2 + \Fv_{\Pl+}  \dot{\Pl}
	=
	\tfrac{1}{4} \Fv_{\Pl-}^2 + \Fv_{\Pl-} \dot{\Pl}.
\end{align}
Hereafter, we write $\Fg|_{\ell}\,(t)=\Fg(t,\ell(t))$. 
Note that the identity \eqref{e.Fg.dot.} holds for \emph{every} $t\in(\start(\Pl),\ed(\Pl))$.
From this, we further deduce
\begin{align}
	\label{e.rankine.hugoniot}
	&\dot{\Pl} = - \tfrac{1}{4} ( \Fv_{\Pl+} + \Fv_{\Pl-} ),
	\qquad
	\text{whenever } \Fv_{\Pl-}\neq\Fv_{\Pl+},
	\\
	\label{e.Fg.dot..}
	&( \dot{\Fg|_{\Pl}} + \dot{\ell}^2 ) \, \ind_{\Fv_{\Pl-}\neq\Fv_{\Pl+}} 
	= 
	\tfrac{1}{16} ( \Fv_{\Pl-} - \Fv_{\Pl+} )^2,
\end{align}
where \eqref{e.rankine.hugoniot} is the well-known Rankine--Hugoniot condition.
The identity \eqref{e.Fg.dot..} identifies the mass carried by a shock path and will be used repeatedly when we pass between Burgers solutions and measures.

Next, to see the connection between Burgers' equation and the large deviations, consider
\begin{align}
	\label{e.Hfn.Ufn}
	\Hfn[\Mmu](t,x) := \Metric_{\Mmu}(0,0;t,x), 
	\qquad (t,x)\in(0,1]\times\R.
\end{align}
This notation is convenient because the variational problem \eqref{e.t.main} concerns the terminal profile $\Metric_{\Mmu}(0,0;1,x)=\Hfn[\Mmu](1,x)$, while letting $t$ vary produces the full height function.
Under this notation, $\Fhh$ in \eqref{e.Fhh.Fuu} is $\Hfn[0]$.
For sufficiently ``nice'' $\Mmu$, the function $\partial_{x}\Hfn[\Mmu]$ is a tractable solution.

\begin{defn}\label{d.pwlinear.measure}
	Call $ \Mmu=\sum_{\Pgamma\in\Snet} \dens_{\Pgamma}\,\delta_{\Pgamma}\in\Msp_+$ \textbf{piecewise linear and cone-supported} if $\Snet$ consists of finitely many linear (meaning $\dot{\Pgamma}$ is constant) internally disjoint paths, each $\dens_{\Pgamma}$ is constant, and $\supp(\Mmu)\subset\cone(c)$ for some $c<\infty$.
\end{defn}

\begin{prop}\label{p.pwlinear.}
	For any piecewise linear and cone-supported $\Mmu$, $\partial_{x}\Hfn[\Mmu]$ is a tractable solution.
\end{prop}
\noindent{}%
A stronger version of Proposition~\ref{p.pwlinear.}, Proposition~\ref{p.pwlinear}, will be proven in Section~\ref{s.tools.pwlinear}.

\begin{rmk}
Proving Theorem~\ref{t.main} requires considering more general $\Mmu$s than those in Definition~\ref{d.pwlinear.measure}.
We achieve this by approximating the former by the latter.
It is of independent interest to understand how general $\Mmu$ can be for $\partial_{x}\Hfn[\Mmu]$ to still be a weak solution and to understand properties of $\partial_{x}\Hfn[\Mmu]$.
\end{rmk}

Here are some notation for the next example.
Given paths $\Peta$ and $\Peta'$ with $\ed(\Peta)=\start(\Peta')$ and $\Peta(\ed(\Peta))=\Peta'(\start(\Peta'))$, we write $\Peta\cup\Peta':=\Peta\ind_{[\start(\Peta),\ed(\Peta)]}+\Peta'\ind_{(\start(\Peta'),\ed(\Peta')]}$ for the concatenated paths.
Call a path $\Peta$ an \textbf{$\Metric_{\Mmu}$ geodesic} if $|\Peta|_{\Metric_{\Mmu}}=\Metric_{\Mmu}(\start,\Peta(\start);\ed,\Peta(\ed))$, where $\start=\start(\Peta)$ and $\ed=\ed(\Peta)$.

\begin{ex}\label{ex.single.}
An elementary yet useful example is $\Mnu_{\alpha,\beta}:=\beta\,\delta_{\Pgamma_{\alpha}}$, where $\Pgamma_{\alpha}(t):=\alpha t$, $\alpha\in\R$, and $\beta\geq0$.
One can show that, for any $(t,x)\in(0,1]\times\R$, the $\Metric_{\Mnu_{\alpha,\beta}}$ geodesic $(0,0)\xrightarrow{\Pgeo}(t,x)$ is of the form $\Pgeo = \Pgamma_\alpha|_{[0,s]}\cup\chi$, for some $s\in[0,t]$, where $\chi$ is the linear path that connects $(s,\Pgamma_\alpha(s))$ and $(t,x)$.
Recall $\FA_{\alpha,\beta}$ from \eqref{e.FA}--\eqref{e.FA.}.
One can further solve the height function explicitly to get
\begin{align}
	\Hfn[\Mnu_{\alpha,\beta}](t,x)
	=
	\FA_{\alpha,\beta}(t,x)
	:=
	t \FA_{\alpha,\beta}(1,\tfrac{x}{t}).
\end{align}
From this, we see that $\Fu:=\partial_x\Hfn[\Mnu_{\alpha,\beta}]$ has a non-entropy shock along $\graph(\Pgamma_\alpha)$.
\end{ex}

\subsection{Kruzhkov entropy}
\label{s.burgers.kruzkov}
We now move to the highly relevant notion of the Kruzhkov entropy. 
Take $\entbm(\alpha):= \alpha^2/4$ and view it as the entropy. 
This is the relevant choice of entropy because 
$\Fphi\mapsto \int \d x \, \entbm(\partial_x\Fphi)$ gives the large-deviation rate function of the Brownian motion with diffusivity 2, which is the invariant measure of the KPZ fixed point, modulo height shifts. 
Set $\entfbm(\alpha) := \alpha^3/12$, which will be seen to be the entropy flux.
Take any tractable solution $\Fv$. 
Within each smooth region $\SD\in\Sdom(\Fv)$, the function $\Fv$ is a $\Csp^1$ solution of Burgers' equation. 
Using the chain rule together with $\partial_t\Fv=\frac{1}{4}\partial_{x}(\Fv^2)$ gives
\begin{align}
	\label{e.ent.entf}
	\partial_t \entbm(\Fv)  - \partial_x \entfbm(\Fv) = 0,
	\quad
	\text{on } \SD,
\end{align}
or equivalently $\frac{\d~}{\d t} \int_{\alpha}^{\beta} \d x\,\entbm(\Fv(t,x)) = \entfbm(\Fv(t,\beta))-\entfbm(\Fv(t,\alpha))$, for all $\{t\}\times[\alpha,\beta]\subset\SD$.
This means, within $\SD$, the quantity $\entbm$ is conserved, and $\entfbm$ is the corresponding flux.
On the other hand, along any boundary path $\Pl\in\Sbdy(\Fv)$, 
\begin{align}
	\label{e.ent.production.path.}
	\big( \partial_t \entbm(\Fv)  - \partial_x \entfbm(\Fv) \big)\big|_{\Pl}
	=
	\big( \entbm(\Fv_{\Pl+})-\entbm(\Fv_{\Pl-}) \big)\, \dot{\Pl} \, \delta_{\Pl}
	- 
	\big( \entfbm(\Fv_{\Pl+})-\entfbm(\Fv_{\Pl-}) \big) \, \delta_{\Pl},
\end{align}
which is an element of $\Msp$ in \eqref{e.Msp}.
The equality \eqref{e.ent.production.path.} is interpreted in the weak sense.
Indeed, since $\entfbm(\Fv)=\entfbm(\Fv(t,x))$ can have a jump in $x$ at $\Pl(t)$, its $x$ derivative is equal to $(\entfbm(\Fv_{\Pl+})-\entfbm(\Fv_{\Pl-}))\delta_{\Pl}$ in the weak sense; similarly for $\partial_t \entbm(\Fv)$.
Using \eqref{e.rankine.hugoniot}, we simplify the right-hand side of \eqref{e.ent.production.path.} as
\begin{align}
	\label{e.ent.production.path}
	\big( \partial_t \entbm(\Fv)  - \partial_x \entfbm(\Fv) \big)\big|_{\Pl}
	=
	\tfrac{1}{48} \big(\Fv_{\Pl-}-\Fv_{\Pl+})^3 \, \delta_{\Pl} \in \Msp.
\end{align}
This is nonzero wherever shocks occur, and is positive at non-entropy shocks and negative at entropy shocks.
Accordingly, we consider the positive/negative entropy production
\begin{align}
	\label{e.ent.production.pm}
	\Entp_\pm(\Fv) 
	:= 
	\sum_{\Pl\in\Sbdy(\Fv)} \frac{1}{48} \int_{\start(\Pl)}^{\ed(\Pl)} \d t \, \big|\Fv_{\Pl-}-\Fv_{\Pl+}\big|^3 \ind_{\Fv_{\Pl\mp}>\Fv_{\Pl\pm}}.
\end{align}
Combining \eqref{e.ent.production.path.} over $\Pl\in\Sbdy(\Fv)$ and \eqref{e.ent.entf} over $\SD\in\Sdom(\Fv)$ gives
\begin{align}
	\label{e.ent.production}
	\int_{[0,1]\times\R} \d t \d x\, \big( \partial_t \entbm(\Fv)  - \partial_x \entfbm(\Fv) \big)
	= 
	\Entp_+(\Fv) - \Entp_-(\Fv).
\end{align}
The left-hand side is interpreted as $( \partial_t \entbm(\Fv)  - \partial_x \entfbm(\Fv) )\in\Msp$ acting on the constant function $1$.

The entropy productions $\Entp_\pm(\Fv)$ are closely related to the rate function $\rate$.
We now derive the identities (in \eqref{e.rate.M}) that link the two.
Recall that $\Msp_+$ denotes the subspace of \eqref{e.Msp} consisting of positive measures.
Let $\Fv=\partial_x\Fg$ be a tractable solution and consider 
\begin{align}
	\label{e.M}
	\M[\Fg] 
	:= 
	\sum_{\Pl\in\Sbdy(\Fv)} \big( \dot{\Fg|_{\Pl}} + \dot{\Pl}^2 \big) \, \ind_{\Fv_{\Pl-}\neq\Fv_{\Pl+}} \, \delta_\Pl
	=
	\sum_{\Pl\in\Sbdy(\Fv)} \frac{1}{16} \big( \Fv_{\Pl-} - \Fv_{\Pl+} \big)^2 \, \delta_\Pl
	\in
	\Msp_+,
\end{align}
where we used \eqref{e.Fg.dot..} in the last equality.
Decompose $\M[\Fg]$ into its non-entropy and entropy parts
\begin{align}
	\label{e.Mnon}
	\Mnon[\Fg] 
	&:= 
	\sum_{\Pl\in\Sbdy(\Fv)} \frac{1}{16} \big( \Fv_{\Pl-} - \Fv_{\Pl+} \big)^2 \ind_{\Fv_{\Pl-}>\Fv_{\Pl+}} \, \delta_\Pl,
	\\
	\label{e.Ment}
	\Ment[\Fg] 
	&:= 
	\sum_{\Pl\in\Sbdy(\Fv)} \frac{1}{16} \big( \Fv_{\Pl-} - \Fv_{\Pl+} \big)^2 \ind_{\Fv_{\Pl-}<\Fv_{\Pl+}} \, \delta_\Pl.
\end{align}
Indeed, $\rate(\Metric_{\Mmu})$ can also be viewed as a function of $\Mmu$.
We have been writing $\rate(\Metric_{\Mmu})$ (instead of $\rate(\Mmu)$) to keep the notation consistent with that of \cite{das24}.
The notation however becomes clumsy when taking $\M[\Fg]$, etc., as the measure.
Hence, for these measures, we write
\begin{align}
	&
	\rate\big(\M[\Fg]\big) := \rate\big(\Metric_{\M[\Fg]}\big),
	&&
	\rate\big(\Mnon[\Fg]\big) := \rate\big(\Metric_{\Mnon[\Fg]}\big),
	&&
	\rate\big(\Ment[\Fg]\big) := \rate\big(\Metric_{\Ment[\Fg]}\big).
\end{align}
Comparing \eqref{e.rate} and \eqref{e.ent.production.pm} shows that, for $\Fv=\partial_x\Fg$,
\begin{align}
	\label{e.rate.M}
	\rate\big(\Mnon[\Fg]\big)
	= 
	\Entp_+(\Fv),
	\qquad
	\rate\big(\Ment[\Fg]\big)
	= 
	\Entp_-(\Fv).
\end{align}

The key to our proof is the following identity, \eqref{e.key}, that relates entropy productions to $\Engbm$.
\begin{prop}\label{p.key}
For any tractable solution $\Fv=\partial_x\Fg$,
\begin{align}
	\label{e.key}
	\rate\big(\Mnon[\Fg]\big) - \rate\big(\Ment[\Fg]\big)
	=
	\Entp_+(\Fv) - \Entp_-(\Fv)
	=
	\frac{1}{4}\int_{\R} \d x \, \big( (\partial_x\Fg(1,x))^2 - 4x^2 \big).
\end{align}
\end{prop}
\begin{proof}
Take a $c\in(0,\infty)$ such that $\Fg(t,x)=\Fhh(t,x)$ outside of $\cone(c)$, which also gives $\Fv(t,x)=\Fuu(t,x)$ outside of $\cone(c)$.
Consider $\FF(t):=\int_{-2c}^{2c} \d x \, (\entbm(\Fv(t,x))-\entbm(\Fuu(t,x)))$ and note that $\FF(1)= $(RHS of \eqref{e.key}).
Given that $\Fv$ is a tractable solution, it is not hard to check that $\FF$ is piecewise $\Csp^1$ on $[0,1]$ and $\FF(0)=0$.
Hence
\begin{align}
	\label{e.p.key}
	\text{(RHS of \eqref{e.key})}
	= 
	\int_0^{1} \d t \, \dot{\FF}(t)
	= 
	\int_{[0,1]\times[-2c,2c]} \hspace{-15pt} \d t \d x \, \partial_t\entbm(\Fv)
	-
	\int_{[0,1]\times[-2c,2c]} \hspace{-15pt} \d t\d x \, \partial_t\entbm(\Fuu).
\end{align}
In the second last expression, $\partial_t\entbm(\Fv)\in\Msp$ and the integral is interpreted as an element of $\Msp$ acting on the function $1_{[0,1]\times[-2c,2c]}$.
Since $\Fuu(t,x)=-2x/t$ is a $\Csp^1$ solution of Burgers' equation, $\partial_t\entbm(\Fuu)=\partial_x\entfbm(\Fuu)$. 
Replace the last integral in \eqref{e.p.key} with $\int_0^1 \d t \,\entfbm(\Fuu)|_{-2c}^{2c}$.
Since $\Fv=\Fuu$ outside of $\cone(c)$, the integral is equal to $\int_0^1 \d t \,\entfbm(\Fv)|_{-2c}^{2c}=\int_{[0,1]\times[-2c,2c]} \d t \d x\,\partial_x(\entfbm(\Fv))$.
Again, the last expression is interpreted as an element of $\Msp$ acting on the function $1_{[0,1]\times[-2c,2c]}$.
We now have
\begin{align}
	\label{e.p.key.}
	\text{(RHS of \eqref{e.key})}
	= 
	\int_{[0,1]\times[-2c,2c]} \hspace{-15pt} \d t \d x \, \partial_t\entbm(\Fv)
	-
	\int_{[0,1]\times[-2c,2c]} \hspace{-15pt} \d t \d x \, \partial_x\entfbm(\Fv).
\end{align}
Combine the two integrals, note that $\partial_t\entbm(\Fv)-\partial_x\entfbm(\Fv)\in\Msp$ is supported on $\cone(c)\subset[0,1]\times[-2c,2c]$, and use \eqref{e.ent.production} and \eqref{e.rate.M}.
Doing so gives the desired result.
\end{proof}

\subsection{Comparing $\rate$ and $\Entp_\pm$}
\label{s.burgers.comparing}
It seems that the identities in \eqref{e.rate.M} equate the rate function $\rate$ to the entropy productions $\Entp_\pm$, but there is a catch.
Take any $\Mmu\in\Msp_+$ such that $\Fv:=\partial_{x}\Fg:=\partial_{x}\Hfn[\Mmu]$ is a tractable solution.
The identities in \eqref{e.rate.M} involve $\rate(\Ment[\Fg])$ and $\rate(\Mnon[\Fg])$, not $\rate(\Metric_{\Mmu})$, and hence do not equate $\rate(\Metric_{\Mmu})$ to the entropy productions of $\Fv=\partial_x\Metric_{\Mmu}(0,0;\cdot,\cdot)$.
Examples~\ref{ex.overshalled}--\ref{ex.badM} below show that, in general, $\rate(\Metric_{\Mmu})\neq\rate(\Mnon[\Fg])$ and $\rate(\Metric_{\Mmu})\neq\rate(\Ment[\Fg])$. 
Indeed, comparing these quantities boils down to comparing $\Mmu$ and $\Mnon[\Fg]$, and $\Mmu$ and $\Ment[\Fg]$.
We make a conjecture about how $\Mmu$ and $\Mnon[\Fg]$ compare.
\begin{conj}\label{conj.measures}
Given any $\Mmu\in\Msp_+$ such that $\Fv:=\partial_{x}\Fg:=\partial_{x}\Hfn[\Mmu]$ is a tractable solution, we have $\Mmu\lfloor_{\supp(\Mnon[\Fg])}=\Mnon[\Fg]$. 
Namely $\Mmu$ and $\Mnon[\Fg]$ coincide on the non-entropy shocks of $\Fv$.
In particular, $\rate(\Metric_{\Mmu})\geq \rate(\Mnon[\Fg])=\Entp_+(\Fv)$.
\end{conj}

Later in Proposition~\ref{p.pwlinear}, we will show that when $\Mmu$ is piecewise linear and cone-supported, $\rate(\Metric_{\Mmu})\geq\rate(\Mnon[\Fg])$.
This, as it turns out, suffices for our proof of Theorem~\ref{t.main}\eqref{t.main.inf}.
Also related to Conjecture~\ref{conj.measures} is Theorem~\ref{t.main}\eqref{t.main.measure}.
There, we prove that $\Mmu=\Mnon[\Fg]$ when $\Mmu$ is a minimizer of \eqref{e.t.main}, under the assumption that $\Ff$ is piecewise linear or quadratic.

Here are the examples that show that $\rate(\Metric_{\Mmu})\neq\rate(\Mnon[\Fg])$ and that $\rate(\Metric_{\Mmu})\neq\rate(\Ment[\Fg])$.
The examples require Lemma~\ref{l.pwlinear1}.
Recall $\FA_{\alpha,\beta}$ from \eqref{e.FA}--\eqref{e.FA.} and recall that $\Pgamma_{\alpha}(t):=\alpha t$.
\begin{ex}\label{ex.overshalled}
Take $\Mmu = \beta_1\,\delta_{\Pgamma_{\alpha_1}}+\beta_2\,\delta_{\Pgamma_{\alpha_2}}$ for any $\alpha_1\neq\alpha_2$.
For fixed $\alpha_1,\alpha_2,\beta_2$, when $\beta_1$ gets large enough, we have $\FA_{\alpha_1,\beta_1}(1,\cdot)>\FA_{\alpha_2,\beta_2}(1,\cdot)$ on $\R$.
Take any such $\beta_1$. 
By Lemma~\ref{l.pwlinear1}, we have $\Hfn[\Mmu]=\FA_{\alpha_1,\beta_1}$, so $\Mnon[\Fg]=\beta_1\delta_{\Pgamma_{\alpha_1}}$ and $\Ment[\Fg]=0$.
In particular,
\begin{align}
	\rate(\Metric_{\Mmu})
	=
	\tfrac{4}{3}(\beta_1^{3/2}+\beta_{2}^{3/2})
	>
	\tfrac{4}{3}\beta_1^{3/2}
	=
	\rate(\Mnon[\Fg])=\Entp_+(\Fv).
\end{align}
\end{ex}

\begin{ex}\label{ex.badM}
Take $\Mmu = \beta\,\delta_{\Pgamma_{\alpha}} + \beta\,\delta_{-\Pgamma_{\alpha}}$ with $0<\alpha<\sqrt{\beta}$.
By Lemma~\ref{l.pwlinear1},
\begin{align}
	\Hfn[\Mmu](t,x) 
	= 
	\FA_{\alpha,\beta}(t,x) \vee \FA_{-\alpha,\beta}(t,x) 
	= 
	t\, \big( \FA_{\alpha,\beta}(1,\tfrac{x}{t}) \vee \FA_{-\alpha,\beta}(1,\tfrac{x}{t}) \big). 
\end{align}
From this expression and \eqref{e.FA}--\eqref{e.FA.}, one can check that $\Ment[\Fg]=(\sqrt{\beta}-\alpha)^2\,\delta_0$, where $0$ denotes the constant path at $x=0$.
Hence $\rate(\Ment[\Fg])=\frac{4}{3}(\sqrt{\beta}-\alpha)^3$, which differs from $\rate(\Metric_{\Mmu})=\frac{8}{3}\beta^{3/2}$.
\end{ex}

\subsection{Idea of the proof}
\label{s.burgers.proof}
Let us explain the idea of the proof of Theorem~\ref{t.main}.
To focus on the idea, let us set aside finer technical details and consider the simpler case where $\intvl$ consists of finitely many points and $\Mmu$ is piecewise linear and cone-supported (Definition~\ref{d.pwlinear.measure}).
Put $\Fh=\Hfn[\Mmu]$ and assume $\Fh(1,\cdot)|_{\intvl}=\Ff$, as required by \eqref{e.t.main}. 
Proposition~\ref{p.pwlinear} in Section~\ref{s.tools.pwlinear} shows that $\partial_x\Fh$ is a tractable solution with $\rate(\Metric_{\Mmu})\geq\rate(\Mnon[\Fh])$.
Hence Proposition~\ref{p.key} yields
\begin{align}
\label{e.idea}
\begin{split}
	\rate(\Metric_{\Mmu}) 
	\geq
	\rate(\Mnon[\Fh])
	\geq& 
	\rate(\Mnon[\Fh]) - \rate(\Ment[\Fh]) 
	\\
	=&
	\frac{1}{4}\int_{\R} \d x \, \big( (\partial_x\Fh(1,x))^2 - 4x^2 \big)
	\geq
	\Engbm(\intvl,\Ff),
\end{split}
\end{align}
where the last inequality follows because $\Fh(1,\cdot)|_{\intvl}=\Ff$.
The first inequality becomes sharp when the chosen measure already equals $\Mnon[\Fh]$, so no extra mass is placed away from the non-entropy shock set.
Let us focus on the second and last inequalities in \eqref{e.idea}.
The second inequality becomes sharp if and only if $\Fh$ does not have entropy shocks, and the last becomes sharp if and only if $\Fh(1,\cdot)=\Ff_*$, where $\Ff_*$ is the minimizer of \eqref{e.Engbm}; see Figure~\ref{f.Ff*}.

Hence, the question becomes whether we can construct a tractable solution $\partial_x\Fh_*$ without entropy shocks such that $\Fh_*(1,\cdot)=\Ff_*$.
To answer this question, recall that, for a given initial data $\Fphi$ in a suitable class, the Hopf--Lax operator 
\begin{align}
	\label{e.hopflax}
	\Fh(t,x) 
	= 
	\HL_{0 \to t} \big[\Fphi\big](x)
	&:=
	\sup\big\{ |\Pchi|_{\diri} + \Fphi(\Pchi(0)) : \{0\}\times\R \xrightarrow{\Pchi} (t,x) \big\}
\\
	\label{e.hopflax.}
	&=
	\sup \Big\{ -\frac{(x-y)^2}{t} + \Fphi(y): y\in\R\Big\},
\end{align}
produces, through $\partial_x\Fh$, the entropy solution of Burgers' equation with the initial data $\partial_x\Fphi$; see \cite[Chapter~3]{evans22} for example.
(Let us assume that the solution is tractable: In Section~\ref{s.tools.hopflax}, we will specify a class of $\Fphi$ for this to hold.)
The second line follows from the fact that $|\Pchi|_{\diri}$ is maximized by the linear path among all paths $(s,y) \xrightarrow{\Pchi} (t,x)$, for fixed $(s,y,t,x)$.
Being an entropy solution, $\partial_x\Fphi$ does not have non-entropy shocks.
Now, if we \emph{reverse time}, the role of entropy shocks and non-entropy shocks exchange.
Namely, the backward Hopf--Lax operator
\begin{align}
	\label{e.hopflax.bk}
	\HLbk_{1\to t} \big[\Fphi\big](x)
	&:=
	\inf\big\{ -|\Pchi|_{\diri} + \Fphi(\Pchi(1)) : (t,x)\xrightarrow{\Pchi}\{1\}\times\R  \big\}
\\
	&=
	\inf \Big\{ \frac{(x-y)^2}{1-t} + \Fphi(y) : y\in\R \Big\}
\end{align}
produces a weak solution of Burgers' equation, $\partial_x\HLbk_{1\to t}[\Fphi](x)$, \emph{without entropy shocks}.
Given this property, the answer to the question is yes, by taking $\Fh_*(t,x)=\HLbk_{1\to t}[\Ff_*](x)$.

The above discussion explains why we expect the results in Theorem~\ref{t.main}.
Our proof follows the same line of reasoning, but needs to incorporate various approximation procedures to handle more general $\intvl$ and $\Mmu$ than those considered above.


\section{Tools}\label{s.tools}
Here we prepare some tools, notation, and preliminary results to facilitate the proof.
\subsection{Evolving piecewise linear or quadratic data by Hopf--Lax operators}
\label{s.tools.hopflax}
Let us recall a geometric picture, which we refer to as shear-and-cut, of understanding the time evolution of the Hopf--Lax operators.
The picture will be useful in understanding Definition~\ref{d.pwlinear} and the discussion that follows.
However, the actual proof pertaining to this discussion is all contained in Lemma~\ref{l.hopflax}, and for those we will appeal to the Hopf--Lax formula rather than the geometric picture, because the former is technically more straightforward.

Take a piecewise linear or quadratic $\Fphi\in\Csp(\R)$ such that $\Fphi=\Fhh(1,\cdot)$ outside of a bounded interval, view $\Fphi$ as the terminal data at time $1$, evolve it by the backward Hopf--Lax operator as in \eqref{e.hopflax.bk}, and let $\Fh(t,x):=\HLbk_{1\to t}[\Fphi](x)$ and $\Fu:=\partial_{x}\Fh$.
Under this setup, $\Fu$ can be described (fairly) explicitly by the ``shear-and-cut'' procedure (see \cite[Chapter~2]{whitham11} for example), which we now recall.
Take the piecewise linear function $\Fu(1,\cdot)=\partial_x\Fphi$ and consider its complete graph $\text{CG}(1)\subset\R^2$, which consists of the graph of $\Fu(1,\cdot)$ and vertical line segments at the jumps, as depicted in Figure~\ref{f.shearandcut1}.
To evolve $\Fu(1,\cdot)$ to $\Fu(t,\cdot)$ for $t<1$, apply the shearing $(x,\Fu)\mapsto (x+(1-t)\Fu,\Fu)$ to $\text{CG}(1)$.
The result gives a piecewise linear curve, which is denoted by $\text{CG}'(t)$ and depicted in Figure~\ref{f.shearandcut1}.
The curve $\text{CG}'(t)$ is not a complete graph of a function when it has overhangs.
At each overhang, make a vertical cut, in such a way that the regions bounded by the cut and $\text{CG}'(t)$, to the left and right of the cut, have the same area; see Figure~\ref{f.shearandcut2}.
These cuts convert $\text{CG}'(t)$ into $\text{CG}(t)$, and the latter is the complete graph of $\Fu(t,\cdot)$.
The scenario can actually be more complicated than what is illustrated in Figure~\ref{f.shearandcut2}, for example multiple overhangs sharing the same cut.
We refer to \cite[Chapter~2]{whitham11} for more discussion on this.
As mentioned, we will use the shear-and-cut procedure only as a picture for intuitive understanding, not in the proof.%


Let us examine properties of $\Fu$ under the shear-and-cut procedure.
First, the shearing procedure evolves linear functions according to 
\begin{align}
	\label{e.shearing}
	\alpha_0
	\xmapsto{\text{shearing }}
	\alpha_0,
	\qquad
	\tfrac{1}{\alpha_1} (x-\beta)
	\xmapsto{\text{shearing }}
	\tfrac{1}{1-t+\alpha_1}(x-\beta).
\end{align}
Here $\alpha_0,\alpha_1,\beta$ are constant and $\alpha_1\neq 0$.
Next, consider the vertical lines in $\text{CG}(1)$.
As seen in Figure~\ref{f.shearandcut2}, any vertical line associated with a downward jump remains vertical, while any vertical line associated with an upward jump evolves into a line of a finite slope.
In the latter case, the second relation in \eqref{e.shearing} is understood in the limiting case $\alpha_1\to 0$, which corresponds to a vertical line.
In particular, the (backward) time evolution removes all upward jumps in $\Fu(1,\cdot)$.
Since upward jumps correspond to entropy shocks, $\Fu$ does not have any entropy shocks.
Based on these properties, it is not hard to check that at every $t\in(0,1]$, $\Fu(t,\cdot)$ is piecewise linear, that $\Fu(t,x)=-2x/t$ outside of a light cone $\cone(c)$, and that $\Fu$ is bounded on $\cone(c)$.
Next, take any $\Pl(t)\in\R$ where $\Fu(t,\cdot)$ or $\partial_x\Fu(t,\cdot)$ is not continuous.
We seek to describe how $\Pl(t)$ evolves in $t$.
By \eqref{e.shearing}, $\Fu(t,x)|_{x=\Pl(t)^\pm}$ is either a constant $\alpha'$ or of the form $(x-\beta)/(-t+\alpha)$, where we absorb the $1$ in the denominator in \eqref{e.shearing} into $\alpha$. 
Using this and the fact that $\Fh$ solves $\partial_t\Fh=\frac{1}{4}(\partial_x\Fh)^2$ a.e., we see that $\Fh(t,x)|_{x=\Pl(t)^\pm}$ takes the form $\alpha' x+\alpha'{}^2t/4+\alpha''$ or $(x-\beta)^2/(2(-t+\alpha))+\alpha'''$.
This and the continuity of $\Fh(t,\cdot)$ give, for constant $\alpha$s and $\beta$s, one of the four equations
\begin{align}
	\left.\begin{array}{l}
		\alpha_{-} \Pl(t)+\tfrac{t}{4}\alpha'{}^2_{-} +\alpha''_{-}
		\\
		\tfrac{1}{2(-t+\alpha_{-})}(\Pl(t)-\beta_{-})^2+\alpha'''_{-}
	\end{array}\right\}
	=
	\left\{\begin{array}{l}
		\alpha_{+} \Pl(t)+\tfrac{t}{4}\alpha'{}^2_{+} +\alpha''_{+}
		\\
		\tfrac{1}{2(-t+\alpha_{+})}(\Pl(t)-\beta_{+})^2+\alpha'''_{+}\ .
	\end{array}\right.
\end{align}
Solving these equations shows that $\Pl(t)$ is of the form
\begin{align}
	\label{e.pwlinear.Pl}
	\Pl(t) = \Fp_1(t) \pm \sqrt{\Fp_2(t)},
	\qquad
	\Fp_1, \Fp_2 \text{ polynomials.}
\end{align}

\begin{figure}
\hfill
\begin{minipage}[t]{.38\linewidth}
\fbox{\includegraphics[width=\linewidth]{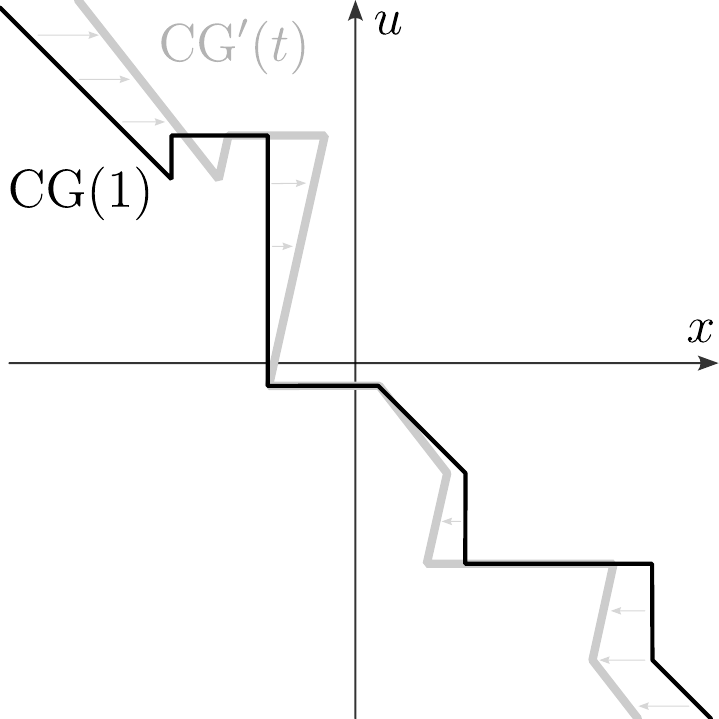}}
\caption{Shearing.
	The shearing maps $(x,\Fu)$ to $(x+(1-t)\Fu,\Fu)$.
}
\label{f.shearandcut1}
\end{minipage}
\hfill
\begin{minipage}[t]{.54\linewidth}
\fbox{\includegraphics[width=\linewidth]{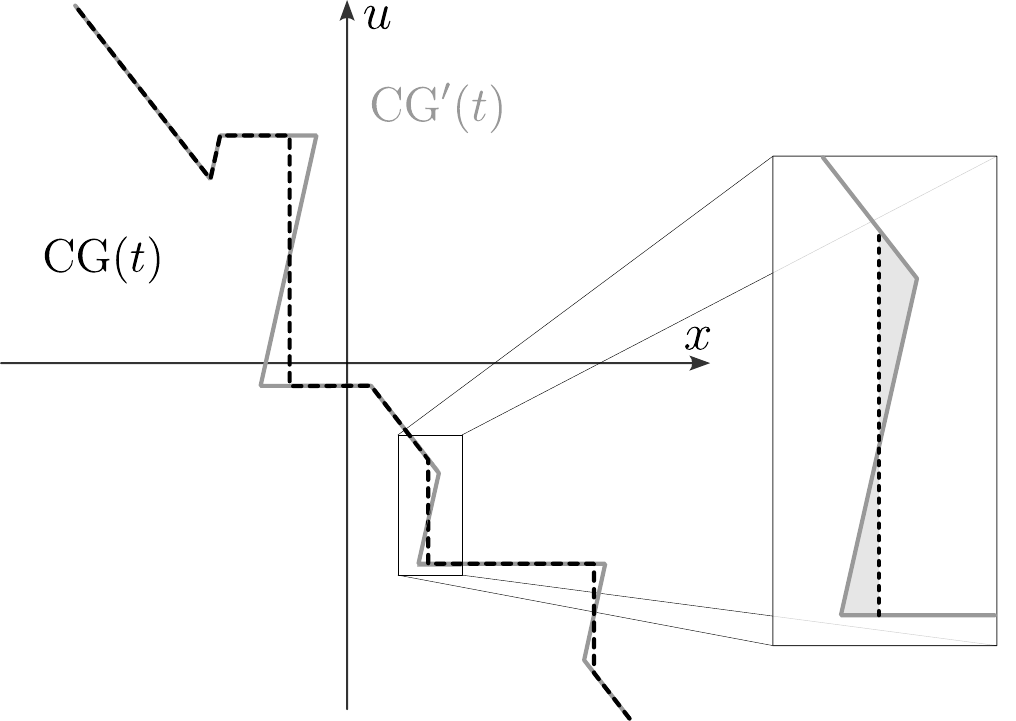}}
\caption{Cutting.
	In the insert, the two gray regions have the same area.
}
\label{f.shearandcut2}
\end{minipage}
\hfill
\end{figure}

The discussion above motivates the definition of a subclass of solutions that we call piecewise linear.
To facilitate the subsequent proof, it will be convenient to consider solutions not just on the entire $[0,1]\times\R$, but also on the smaller region $[s,s']\times\R$.

\begin{defn}\label{d.pwlinear}
Call $\Fv=\partial_x\Fg$ a \textbf{piecewise linear solution on $[s,s']\times\R$} with respect to a tractable partition $(\Stime(\Fv),\Sbdy(\Fv),\Sdom(\Fv))$ of a light cone $\cone(c)$ if $\Fg$ solves $\partial_t\Fg=\frac{1}{4}(\partial_x\Fg)^2$ a.e.\ on $[s,s']\times\R$, $s,s'\in\Stime(\Fv)$, and
\begin{enumerate}
\item \label{d.pwlinear.cone}
$\Fg=\Fhh$ on $([s,s']\times\R)\setminus\cone(c)$ and $\Fg\in\Csp(([s,s']\times\R)\cap\cone(c))$, 
\item \label{d.pwlinear.domain}
for each $\SD\in\Sdom(\Fv)$ such that $\SD\subset(s,s')\times\R$, the solution $\Fv$ is either constant on $\SD$, or is bounded and takes the form $(x-\beta)/(-t+\alpha)$ on $\SD$, for some $\alpha,\beta\in\R$; see Remark~\ref{r.pwlinear}\eqref{r.pwlinear.singular}.
\item \label{d.pwlinear.bdy}
each $\Pl\in\Sbdy(\Fv)$ with $[\start(\Pl),\ed(\Pl)]\subset[s,s']$ takes the form \eqref{e.pwlinear.Pl} and is continuous on $[\start(\Pl),\ed(\Pl)]$, and $\Fv_{\Pl-}-\Fv_{\Pl+}$ is either always positive, always negative, or always zero on $(\start(\Pl),\ed(\Pl))$.
\end{enumerate}
A \textbf{tractable solution on $[s,s']\times\R$} with respect to $(\Stime(\Fv),\Sbdy(\Fv),\Sdom(\Fv))$ is defined similarly.
\end{defn}

\begin{rmk}\label{r.pwlinear}
\begin{enumerate}
\item[]
\item
Indeed, a piecewise linear solution is a tractable solution.
The phrase piecewise linear refers to the property that $\Fv(t,x)$ is piecewise linear in $x$ for fixed $t$.
\item 
Because $\Fv$ is only defined on $[s,s']\times\R$, the portion of $(\Stime(\Fv),\Sbdy(\Fv),\Sdom(\Fv))$ that partitions $([0,s]\cup[s',1])\times\R$ is irrelevant, and
we could have defined the partition only on $[s,s']\times\R$.
\item \label{r.pwlinear.singular}
The function $(x-\beta)/(-t+\alpha)$ is singular at $t=\alpha$.
In this case, $\Fv$ being bounded on $\SD$ requires either $\bar{\SD}\cap(\{\alpha\}\times\R)=\emptyset$ or $\SD\subset\{(t,x):|x-\beta|<c'\,|t-\alpha|\}$ for some $c'<\infty$.
\end{enumerate}
\end{rmk}

The preceding discussion shows that $\partial_x(\HLbk_{1\to t}[\Fphi](x))$ is a piecewise linear solution without entropy shocks.
A similar argument works for the (forward) Hopf--Lax operator and gives a piecewise linear solution without non-entropy shocks.
These properties are summarized in Lemma~\ref{l.hopflax}\eqref{l.hopflax.bk}--\eqref{l.hopflax.}.
We will also need the property stated in Lemma~\ref{l.hopflax}\eqref{l.hopflax..}, which is readily verified from the shear-and-cut procedure above.
\begin{lem}\label{l.hopflax}
\begin{enumerate}
\item[]
\item \label{l.hopflax.bk}
Suppose $\Fphi\in\Csp(\R)$ is piecewise linear or quadratic, and equal to $\Fhh(1,\cdot)$ outside of a bounded interval.
Then $\partial_x(\HLbk_{1\to t}[\Fphi](x))$ is a piecewise linear solution on $[0,1]\times\R$ without entropy shocks.
\item \label{l.hopflax.}
Take any $s\in(0,1)$, and suppose $\Fphi\in\Csp(\R)$ is piecewise linear or quadratic, and equal to $\Fhh(s,\cdot)$ outside of a bounded interval.
Then $\partial_x(\HL_{s\to t}[\Fphi](x))$ is a piecewise linear solution on $[s,1]\times\R$ without non-entropy shocks.

\item \label{l.hopflax..}
Same setting as in Part~\eqref{l.hopflax.}. 
Put $\HL_{s\to t}[\Fphi](x):=\fPhi(t,x)$, and take any linear path $\Pgamma$ with $\start(\Pgamma)=s$.
Then $\fPhi|_{\Pgamma} (t)=\fPhi(t,\Pgamma(t))$ is continuous and piecewise rational.
\end{enumerate}
\end{lem}

\begin{proof}

The proof of Parts~\eqref{l.hopflax.bk}--\eqref{l.hopflax.} are similar, so we will only carry out the proof of Parts~\eqref{l.hopflax.}--\eqref{l.hopflax..}.

\eqref{l.hopflax.}~%
Given the assumption on $\Fphi$, there exists $a_1<\cdots<a_{n+1}$ and linear or quadratic functions $\Fphi_1,\ldots,\Fphi_n$ such that
\begin{align}
	\label{e.phi.piece}
	\Fphi = 
	\begin{cases}
		\Fphi_i, & \text{ on } I_i := [a_i, a_{i+1}], \\
		\Fhh(s,\,\cdot\,), & \text{ on } \R\setminus[a_1,a_{n+1}].
	\end{cases}
\end{align}
Define $\und{\Fphi}_i$ to be $\Fphi$ on $I_i$ and to be $-\infty$ outside $I_i$.
Define $\und{\Fphi}_d$ to be $\Fhh$ on $\R\setminus [a_1,a_{n+1}]$ and to be $-\infty$ on $(a_1,a_{n+1})$.
Inserting \eqref{e.phi.piece} into \eqref{e.hopflax.} gives
\begin{align}	
	\label{e.HL.max}
	\Phi(t,x)
	:=
	\HL_{s \to t} \big[\Fphi\big](x)
	=
	\big( \Phi_1 \vee \cdots \vee \Phi_n \vee \Phi_{d} \big) (t,x), 
\end{align}
where
\begin{align}
	\label{e.Phii.var}
	\Phi_i(t,x) := \sup_{y\in\R} \Big( -\frac{(x-y)^2}{t-s} + \und{\Fphi}_i(y)\Big),
	&&
	\Phi_d(t,x) := \sup_{y\in\R} \Big( -\frac{(x-y)^2}{t-s} + \und{\Fphi}_d(y)\Big).
\end{align}

Given that each $\Fphi_i$ is linear or quadratic and that $h_d(s,y):=-y^2/s$, it is straightforward, though tedious, to verify from \eqref{e.Phii.var} that $\Phi_i$ and $\Phi_d$ satisfy Definition~\ref{d.pwlinear}, for some light cone $\cone(c)$, with the following modifications:
\begin{enumerate}[label=(\alph*)]
\item \label{prop.a}
On $([s,1]\times\R)\setminus\cone(c)$, $\Phi_i(t,x)$ is equal to $h_d(t-s,x)$ (and not $h_d(t,x)$).
\item \label{prop.b}
Whenever defined, $\partial_x \Phi_i$ and $\partial_x \Phi_d$ are bounded on $\{(t,x): \Phi_i(t,x) \geq \phi(x)\}$ (but may not be bounded on every domain of their own tractable partitions).
\end{enumerate}
To reiterate, up to these modifications, $\Phi_i$ and $\Phi_d$ satisfy Definition~\ref{d.pwlinear}.
Using \ref{prop.a} and the property that $\Phi_d(t,x)=h_d(t,x)$ on $([s,1]\times\R)\setminus\cone(c)$ in \eqref{e.HL.max} shows that $\Phi$ satisfies Definition~\ref{d.pwlinear}\eqref{d.pwlinear.cone}.
Next, the functions $\Phi_i$ and $\Phi_d$ satisfy Definition~\ref{d.pwlinear}\eqref{d.pwlinear.domain} with the modification in \ref{prop.b} and with respect to their own tractable partitions.
After refining these tractable partitions, we obtain a tractable partition such that
\begin{enumerate}[label=(\alph*)]
\setcounter{enumi}{2}
\item \label{prop.c}
In the interior of every domain of this partition, $\Phi=\Phi_i$ for an $i$ that does not change within this domain.
\end{enumerate}
We see that $\Phi$ satisfies Definition~\ref{d.pwlinear}\eqref{d.pwlinear.domain}: The boundedness of $\partial_x\Phi$ holds because $\Phi(t,\cdot)\geq \Fphi$ (by considering $y=x$ in \eqref{e.hopflax.}), which together with \ref{prop.b} ensures that only the segments of $\Phi_i$ and $\Phi_d$ with bounded derivative will be picked up through \eqref{e.HL.max}.
Given that $\Phi$ satisfies Definition~\ref{d.pwlinear}\eqref{d.pwlinear.domain}, it also satisfies Definition~\ref{d.pwlinear}\eqref{d.pwlinear.bdy} by the argument given before \eqref{e.pwlinear.Pl}.

Let us show that $\partial_x\Phi$ has only entropy shocks, which amounts to saying that $\partial_x\Phi(t,\xi_-)<\partial_x\Phi(t,\xi_+)$ at every discontinuity point $(t,\xi)$ of $\partial_x\Phi$.
First, from \eqref{e.Phii.var}, it is not hard to verify that this property holds for $\Phi_i$ and $\Phi_d$.
Take any discontinuity point $(t,\xi)$ of $\partial_x\Phi$.
If in a spatial neighborhood of $\xi$, we have $\Phi=\Phi_\alpha$ for some $\alpha\in\{1,\ldots,n,d\}$, the desired property indeed holds for $\Phi$.
Otherwise, by \ref{prop.c} we have $\Phi(t,\,\cdot\,)=\Phi_\alpha(t,\,\cdot\,)$ on $(\xi-\varepsilon,\xi]$ and $\Phi(t,\,\cdot\,)=\Phi_\beta(t,\,\cdot\,)$ on $[\xi,\xi+\varepsilon)$ for some distinct $\alpha,\beta\in\{1,\ldots,n,d\}$ and $\varepsilon>0$.
Further, $\Phi_\alpha(t,\,\cdot\,)\geq \Phi_\beta(t,\,\cdot\,)$ on $(\xi-\varepsilon,\xi]$ because there $\Phi(t,\,\cdot\,)=\Phi_\alpha(t,\,\cdot\,)$, and similarly $\Phi_\beta(t,\,\cdot\,)\geq \Phi_\alpha(t,\,\cdot\,)$ on $[\xi,\xi+\varepsilon)$.
Using these properties together with the fact that $\Phi_\alpha$ and $\Phi_\beta$ have only entropy shocks shows that  $\partial_x\Phi(t,\xi_-)<\partial_x\Phi(t,\xi_+)$.

\smallskip

\eqref{l.hopflax..}~
From \eqref{e.Phii.var}, it is straightforward to verify that $\Phi_i(t,x)$ and $\Phi_d(t,x)$ are piecewise rational functions of $x,t$.
This together with \ref{prop.c} shows the desired property.
\end{proof}

\subsection{Matching}
\label{s.tools.matching}
Here and in the next section, we focus on results of a ``matching'' nature.
Roughly speaking, these results either assert or utilize the statement that two types of functions are the same.
Lemma~\ref{l.match} gives a criterion for two such functions to coincide, and the rest of Sections~\ref{s.tools.matching}--\ref{s.tools.pwlinear} utilizes Lemma~\ref{l.match} to establish results.

To state Lemma~\ref{l.match}, we need some notation. 
Given a $\Mmu\in \Msp_{+}$, we observe that the length of a path $\Peta$ under the metric $\Metric_{\Mmu}$ (defined in \eqref{e.Mmu}) can be written as
\begin{align}
	\label{e.length}
	|\Peta|_{\Metric_{\Mmu}}
	:=
	\Mmu(\graph\Peta) + |\Peta|_{\diri}
	=
	\int_{\start(\Peta)}^{\ed(\Peta)} \d t\, \sum_{\Pgamma\in\Snet} \dens_{\Pgamma}(t) \ind_{\Pgamma(t)=\Peta(t)}
	+
	\int_{\start(\Peta)}^{\ed(\Peta)} \d t\,\big( -\dot{\Peta}^2 \big).
\end{align}
Here $\Snet$ is the network appearing in \eqref{e.Msp}.
For a set $\SS$ of paths, write $\graph(\SS):=\cup_{\Pl\in\SS}\graph(\Pl)$.
Take any tractable partitions $(\Stime,\Sbdy,\Sdom)$ of a light cone $\cone(c)$.
Take any $\und{s}<\bar{s}\in\Stime$, let $\Sbdy'$ be a finite set of piecewise $\Csp^1[\und{s},\bar{s}]$ paths; assume that the paths do not intersect except at $t=\und{s},\bar{s}$ , that $\graph(\Sbdy')\subset\graph(\Sbdy)$, and that $\Sbdy'$ contains the left and right boundaries of $\cone(c)\cap([\und{s},\bar{s}]\times\R)$.
Let $\Sdom'$ denote the set of open connected regions in $\cone(c)\cap((\und{s},\bar{s})\times\R)$ partitioned by the paths in $\Sbdy'$.
For $\dom\in\Sdom'$ and $\FF\in\Csp(\bar{\dom})$, define the domain Hopf--Lax operator
\begin{align}
	\label{e.hopflax.domain}
	\HL_{\partial\dom}[\FF] (t,x)
	:=
	\sup\big\{ 
		|\Pchi|_\diri + \FF(\start,\Pchi(\start)) :  
		(\start,\Pchi(\start))\in\partial\dom, \ 
		\partial\dom\xrightarrow{\Pchi}(t,x), \
		\graph(\Pchi)\subset \bar{\dom}
	\}
\end{align}
for $(t,x) \in \dom$.
Take a $\Fv=\partial_{x}\Fg$ that is a tractable solution on $[\und{s},\bar{s}]\times\R$ with respect to $(\Stime,\Sbdy,\Sdom)$ and satisfies
\begin{align}
	\label{e.l.match.hopflax}
	\HL_{\partial\dom}[\Fg] =\Fg \ \text{ on } \dom,
	\qquad
	\text{for all } \dom\in\Sdom'.
\end{align}
For any $\Pl\in\Sbdy'$, because $\graph(\Sbdy')\subset\graph(\Sbdy)$ and because $\Fv$ is a tractable solution with respect to $(\Stime,\Sbdy,\Sdom)$, by completing the square in \eqref{e.Fg.dot.}, we have $\dot{\Fg|_\Pl}+\dot{\Pl}^2\geq 0$.
Given this, define 
\begin{align}
	\label{e.l.match.Mnu}
	\Mnu &:= \sum_{\Pl\in\Sbdy'} \big( \dot{\Fg|_\Pl}+\dot{\Pl}^2 \big)\, \delta_{\Pl}\ \in\Msp_+,
\\
	\label{e.l.match.Fh}
	\til{\Fh}(t,x) &:= \sup \big\{ |\Peta|_{\Metric_{\Mnu}} + \Fg(\und{s},\Peta(\und{s})) : \{\und{s}\}\times\R \xrightarrow{\Peta} (t,x) \big\},
	\quad
	\text{ for } (t,x)\in[\und{s},\bar{s}]\times\R.
\end{align}

The point of the next lemma is that, once a candidate height function is known to satisfy the appropriate local Hopf--Lax relations, it can be reconstructed globally from the measure placed on the chosen boundary paths.
When $\und{s}=0$, since $\Fg=\Fhh$ outside of $([0,\bar{s}]\times\R)\cap\cone(c)$, we have $\lim_{t\to 0}\Fg(t,x)=-\infty\ind_{x\neq 0}$.
Given this, we set $\Fg(0,\cdot):=-\infty\ind_{x\neq 0}$ in \eqref{e.l.match.Fh}, which amounts to requiring $\Peta(0)=0$.
\begin{lem}\label{l.match}
Under the preceding setup, $\Fg =\til{\Fh}$ on $(\und{s},\bar{s}]\times\R$.
\end{lem}
\begin{proof}
Let us prove that $\Fg\leq\til{\Fh}$ on $\cone(c)\cap([\und{s},\bar{s}]\times\R)$.
Fix any $\eps>0$ and consider $(t,x)\in\dom$ for some $\dom\in\Sdom'$.
Use \eqref{e.l.match.hopflax} to get a path $\partial\dom\xrightarrow{\chi}(t,x)$ such that $|\Pchi|_{\diri} \geq \Fg(t,x) - \Fg(\start,\Pchi(\start))-\eps$, where $\start:=\start(\Pchi)$.
If $\start=\und{s}$, then $\Pchi$ is admissible in \eqref{e.l.match.Fh}, and since $\Mnu$ in \eqref{e.l.match.Mnu} is supported on $\graph(\Sbdy')$ while $\Pchi$ only meets $\graph(\Sbdy')$ at its initial point, \eqref{e.length} gives $|\Pchi|_{\Metric_{\Mnu}}=|\Pchi|_{\diri}$. By \eqref{e.l.match.Fh}, $\Fg(t,x)\leq\til{\Fh}(t,x)+\eps$.
Otherwise $\Pchi(\start)=\Pl(\start)$ for some $\Pl\in\Sbdy'$, and we let $\Peta:=\Pl|_{[\und{s},\start]}\cup\Pchi$, where $\cup$ means concatenating the paths, defined before Example~\ref{ex.single.}.
This way $|\Peta|_{\Metric_{\Mnu}} = \Fg(\start,\Pchi(\start))-\Fg(\und{s},\Peta(\und{s})) + |\Pchi|_{\diri} \geq \Fg(t,x)-\Fg(\und{s},\Peta(\und{s}))-\eps$, because along the boundary segment $\Pl|_{[\und{s},\start]}$ the density term $\dot{\Fg|_\Pl}+\dot{\Pl}^2$ cancels the Dirichlet term $-\dot{\Pl}^2$ and leaves exactly the increment of $\Fg$.
Hence $\Fg(t,x)\leq\til{\Fh}(t,x)+\eps$.
So far we have proven $\Fg\leq\til{\Fh}+\eps$ on every $\dom\in\Sdom'$.
Since $\Fv$ is a tractable solution, $\Fg\in\Csp(\cone(c))$, and it is not hard to check that $\til{\Fh}\in\Csp(\cone(c))$.
Hence the inequality extends to $\cone(c)\cap([\und{s},\bar{s}]\times\R)$.
Since $\eps>0$ was arbitrary, the desired inequality follows.

Next, let us prove that $\Fg\geq\til{\Fh}$ on $\cone(c)\cap([\und{s},\bar{s}]\times\R)$.
Fix any $\eps>0$ and set $\tau:=\inf\{t\in[\und{s},\bar{s}]: \Fg(t,x)\leq\til{\Fh}(t,x)-\eps,\text{ for some } |x|\leq ct\}$, with the convention that $\inf\emptyset:=\infty$.
Our goal is to prove $\tau\geq \bar{s}$.
First, since $\Fg,\til{\Fh}\in\Csp(\cone(c))$ and since $\Fg(\und{s},x)=\til{\Fh}(\und{s},x)$ for $|x|\leq c\und{s}$ by \eqref{e.l.match.Fh}, we have $\tau>\und{s}$.
Given this, we next argue by contradiction. 
Assume $\tau\in(\und{s},\bar{s})$, so that
\begin{align}
	\label{e.l.match.contradict}
	\Fg(\tau,x) = \til{\Fh}(\tau,x) - \eps,
\end{align}
for some $|x|\leq c\tau$.
The supremum in \eqref{e.l.match.Fh} is in fact a maximum, where the existence of an optimal path can be obtained by the direct method in calculus of variations.
Take any optimal $\Peta$ in \eqref{e.l.match.Fh} for $(t,x)\mapsto(\tau,x)$ and consider
\begin{align}
	\label{e.l.match.s}
	\sigma := \sup\big\{ t\in[\und{s},\tau) : \Peta(t)=\Pl(t), \text{ for some } \Pl\in\Sbdy' \big\}.
\end{align}
Note that the set in \eqref{e.l.match.s} is nonempty. 
Otherwise $|\Peta|_{\Metric_{\Mnu}}=|\Peta|_\diri$, and it follows that $\til{\Fh}(\tau,x) \leq \HL_{\partial\dom}[\Fg](\tau,x)=\Fg(\tau,x)$ for the $\dom\in\Sdom'$ that contains $(\tau,x)$, which is prohibited by \eqref{e.l.match.contradict}.
We next argue for a contradiction in two separate cases.
\begin{description}
\item[Case 1, $\sigma<\tau$]
In this case, $(\tau,x)\in\dom$ for some $\dom\in\Sdom'$. 
By \eqref{e.l.match.s}, $(\sigma,\Peta(\sigma))\in\partial\dom$ and $\Peta|_{(\sigma,\tau)}\subset\dom$.
Write $\til{\Fh}(\tau,x)=|\Peta|_{[\,\und{s},\sigma]}|_\diri + \til{\Fh}(\sigma,\Peta(\sigma))$, and note that $\til{\Fh}(\sigma,\Peta(\sigma))<\Fg(\sigma,\Peta(\sigma))+\eps$ because $\sigma<\tau$.
Hence
\begin{align}
	\til{\Fh}(\tau,x) 
	< 
	|\Peta|_{[\und{s},\sigma]}|_\diri + \Fg(\sigma,\Peta(\sigma))+\eps
	\leq
	\HL_{\partial\dom}[\Fg](\tau,x) + \eps
	=
	\Fg(\tau,x)+\eps,
\end{align}
where the last equality followed by \eqref{e.l.match.hopflax}.
This contradicts \eqref{e.l.match.contradict}.

\item[Case 2, $\sigma=\tau$]
In this case, there exists $\Pl_0\in\Sbdy$ such that $\Peta(s_j)=\Pl_0(s_j)$ for a sequence $s_1<s_2<\ldots \to \tau$.
Take an $s:=s_k\in[0,\tau)$ from the sequence, with a large enough $k$ such that
\begin{align}
	\label{e.l.match.noninterseting}
	\mathrm{ConvexHull}(\graph(\Peta|_{[s,\tau]})) \cap \graph(\Pl) = \emptyset, \text{ for all } \Pl \in \Sbdy'\setminus\{\Pl_0\}.
\end{align}
Write $\til{\Fh}(\tau,x)=|\Peta|_{[s,\tau]}|_{\Metric_{\Mnu}} + \til{\Fh}(s,\Peta(s))$.
To get a handle on $|\Peta|_{[s,\tau]}|_{\Metric_{\Mnu}}$, consider $\SO:=\{t\in[s,\tau]:\Peta(t)\neq\Pl_0(t)\}$, which is open, and write it as a disjoint union $\SO=\cup_j(\start_j,\ed_j)$.
This way,
\begin{align}
	\label{e.l.match.length}
	\til{\Fh}(\tau,x)
	=
	|\Peta|_{[s,\tau]}|_{\Metric_{\Mnu}} + \til{\Fh}(s,\Peta(s))
	=
	\int_{[s,\tau]\setminus\SO} \d t \, \dot{\Fg|_{\Peta}} + \sum_{j} \big| \Peta|_{[\start_j,\ed_j]} \big|_{\diri} + \til{\Fh}(s,\Peta(s)).
\end{align}
By \eqref{e.l.match.noninterseting} and the definition of $(\start_j,\ed_j)$, we have $\graph(\Peta|_{(\start_j,\ed_j)})\cap\graph(\Sbdy')=\emptyset$.
Hence $\graph(\Peta|_{(\start_j,\ed_j)})\subset\dom$, for some $\dom\in\Sdom'$, and $(\start_j,\Pl_0(\start_j))\in\partial\dom$.
These properties together with \eqref{e.l.match.hopflax} imply
\begin{align}
	\Fg(t,\Peta(t)) = \HL_{\partial\dom}[\Fg](t,\Peta(t)) \geq|\Peta|_{[\start_j,t]}|_{\diri}+\Fg(\start_j,\Peta(\start_j)),
	\qquad
	\text{for all } t\in(\start_j,\ed_j).
\end{align}
Taking the $t\to\ed_{j}$ limit gives $\Fg(\ed_j,\Peta(\ed_j))\geq |\Peta|_{[\start_j,\ed_j]}|_{\diri}+\Fg(\start_j,\Peta(\start_j))$.
Rewrite this as $|\Peta|_{[\start_j,\ed_j]}|_{\diri} \leq \int_{\start_j}^{\ed_{j}}\d t\, \dot{\Fg|_{\Peta}}$, and insert it into \eqref{e.l.match.length}.
Doing so gives
$
\til{\Fh}(\tau,x) \leq \Fg(\tau,x) - \Fg(s,\Peta(s)) + \til{\Fh}(s,\Peta(s)).
$
The last term is $<\Fg(s,\Peta(s))+\eps$ because $s<\tau$.
Combining the last two inequalities gives $\til{\Fh}(\tau,x) < \Fg(\tau,x) +\eps$, which contradicts \eqref{e.l.match.contradict}.
\end{description}

So far we have proven $\Fg=\til{\Fh}$ on $\cone(c)\cap([\und{s},\bar{s}]\times\R)$, and we need to prove the same equality outside of the light cone, namely on $([\und{s},\bar{s}]\times\R)\setminus \cone(c)$.
Since $\Fg$ is equal to $\Fhh$ outside of the light cone, it suffices to prove that the same holds for $\til{\Fh}$.
Take any $(t,x)\in ([\und{s},\bar{s}]\times\R)\setminus \cone(c)$.
Since $\Mnu$ is supported in $\cone(c)$, any optimal path $\Peta$ in \eqref{e.l.match.Fh} starts from either $\SA_1:=\{(s,cs),(s,-cs):s\in[\und{s},\bar{s}]\}$ or $\SA_2:=\{\und{s}\}\times((-\infty,\und{s}]\cup[\und{s},\infty))$.
We have $\Fg=\til{\Fh}$ on $\SA_1$ because $\SA_1\subset\cone(c)$, and $\Fg=\til{\Fh}$ on $\SA_2$ by the definition of $\til{\Fh}$.
Also, $\til{\Fh}$ is equal to $\Fhh$ on $\SA_1\cup\SA_2$.
Hence, every admissible path in \eqref{e.l.match.Fh} starts from a point where the boundary value is already $\Fhh$, and maximizing the remaining Dirichlet contribution gives $\til{\Fh}(t,x)=-x^2/t=\Fhh(t,x)$.
\end{proof}

We next apply Lemma~\ref{l.match} to prove a reconstruction result.
Recall $\M[\Fg]$ from \eqref{e.M} and $\Hfn$ from \eqref{e.Hfn.Ufn}.
\begin{prop}\label{p.reproducing}
For any tractable solution $\Fv=\partial_x\Fg$, we have $ \Fg = \Hfn[\M[\Fg]] $.
\end{prop}
\begin{proof}
We begin with a reduction.
Consider the corresponding tractable partition $(\{t_0<\ldots<t_k\},\Sbdy(\Fv),\Sdom(\Fv))$ of a light cone $\cone(c)$ for $\Fv$.
Let $\Sbdy_{i}(\Fv)$ denote the set of those paths in $\Sbdy(\Fv)$ with starting time $t_{i-1}$ and ending time $t_{i}$, and let $\Sdom_{i}(\Fv)$ denote the set of those regions in $\Sdom(\Fv)$ that are subsets of $(t_{i-1},t_{i})\times\R$.
We seek to apply Lemma~\ref{l.match} with $(\Stime,\Sbdy,\Sdom)=(\Stime(\Fv),\Sbdy(\Fv),\Sdom(\Fv))$ and $(\und{s},\bar{s},\Sbdy',\Sdom')=(t_{i-1},t_{i},\Sbdy_{i},\Sdom_{i})$ and inductively over $i=1,\ldots,k$.
The induction is over the slabs $[t_{i-1},t_i]\times\R$: Assuming the desired identity has been established up to time $t_{i-1}$, we apply Lemma~\ref{l.match} on the next slab to extend it to time $t_i$.
Under this setup, the $\Mnu$ in \eqref{e.l.match.Mnu} is equal to $\M[\Fg]\lfloor_{[t_{i-1},t_i]\times\R}$.
Hence, once Lemma~\ref{l.match} applies, the desired result follows.

It hence remains only to verify \eqref{e.l.match.hopflax}.
Take any $\SD\in\Sdom_{i}(\Fv)$,  recall $\HL_{\partial\SD}[\Fg]$ from \eqref{e.hopflax.domain}, and note that the paths there can be taken to be linear because, for fixed endpoints, the Dirichlet functional is maximized by the linear path. 
Fix $(t_0,x_0)\in\SD$.
We hence consider all linear paths $\Pchi$s that satisfy $\partial\SD\xrightarrow{\Pchi}(t_0,x_0)$ with $\graph(\Pchi|_{(\start(\Pchi),t_0)})\subset\SD$.
Differentiating $\Fg(t,\Pchi(t))$ with the aid of $\partial_t\Fg=\frac{1}{4}\Fv^2$ gives 
\begin{align}
	\label{e.p.reproducing.}
	\dot{\Fg|_{\Pchi}}=\tfrac{1}{4}(\Fv|_{\Pchi}+2\dot{\Pchi})^2-\dot{\Pchi}^2,
	\qquad
	\dot{\Pchi} = \text{constant.}
\end{align}
Because $\Fv$ is a $\Csp^1(\SD)$ solution of Burgers' equation, every point in $\SD$ lies along a characteristic and the characteristics do not intersect within $\SD$. 
Among all the $\Pchi$s considered above, one of them is a characteristic, denoted $\Pchi_0$.
Combining \eqref{e.p.reproducing.} and \eqref{e.characteristics} for $\Pchi=\Pchi_0$ gives $\Fg(t_0,x_0)=|\Pchi_0|_{\diri}+\Fg(\start_0,\Pchi_0(\start_0))$, where $\start_0:=\start(\Pchi_0)$.
On the other hand, for any $\Pchi\neq\Pchi_0$, we must have $\Fv|_{\Pchi}(t)\neq -2\dot{\Pchi}$ for all $t\in(\start(\Pchi),t_0)$.
Otherwise the characteristic that passes through $(t,\Pchi(t))$ is $\Pchi$ itself.
This is prohibited because $\Pchi$ intersects with $\Pchi_0$ at $(t_0,x_0)$.
Hence, for any $\Pchi\neq\Pchi_0$, we have $\dot{\Fg|_{\Pchi}}(t)>-\dot{\Pchi}^2$ for all $t\in(\start(\Pchi),t_0)$.
Integrating this inequality gives $\Fg(t_0,x_0)>|\Pchi|_{\diri}+\Fg(\start,\Pchi(\start))$, where $\start:=\start(\Pchi)$.
To recap, we showed that $\Fg(t_0,x_0)=|\Pchi_0|_{\diri}+\Fg(\start_0,\Pchi_0(\start_0))$ and that, for all $\Pchi\neq\Pchi_0$, $\Fg(t_0,x_0)>|\Pchi|_{\diri}+\Fg(\start,\Pchi(\start))$.
This proves that $\HL_{\partial\SD}[\Fg](t_0,x_0)=\Fg(t_0,x_0)$.
\end{proof}

Let us state a few monotone properties that will be used.
They are readily checked from \eqref{e.Mmu}, \eqref{e.rate}, and \eqref{e.Hfn.Ufn}.
Hereafter, $\Mnu\leq\Mnu'$ means $\Mnu'-\Mnu\in\Msp_{+}$, and $\Mnu<\Mnu'$ means $0\neq (\Mnu'-\Mnu)\in\Msp_{+}$.
\begin{align}
	\label{e.monotone.weak}
	&\text{For } \Mnu\leq\Mnu'\in\Msp_{+}, \text{ we have } \rate(\Metric_{\Mnu})\leq\rate(\Metric_{\Mnu'}).
\\
	\label{e.monotone.height}
	&\text{For } \Mnu\leq\Mnu'\in\Msp_{+}, \text{ we have } \Hfn[\Mnu]\leq\Hfn[\Mnu'].
\\
	\label{e.monotone}
	&\text{for } \Mnu<\Mnu'\in\Msp_{+}, \text{ we have } \rate(\Metric_{\Mnu})<\rate(\Metric_{\Mnu'}).
\end{align}

\subsection{Piecewise linear construction}
\label{s.tools.pwlinear}
We next turn to the setting of Proposition~\ref{p.pwlinear.}.
Namely, we seek to establish properties of $\Hfn[\Mmu]$ and $\partial_{x}\Hfn[\Mmu]$ for a piecewise linear and cone supported $\Mmu$.
Let us begin by considering a more special $\Mmu$.
Recall the function $\FA_{\alpha,\beta}$ from \eqref{e.FA}--\eqref{e.FA.}.
\begin{lem}\label{l.pwlinear1}
Consider $\Mmu=\sum_{\Pgamma\in\Snet}\dens_{\Pgamma}\delta_{\Pgamma}$ where $\Snet$ consists of finitely many paths (that intersect only at $t=0$) of the form $\Pgamma(t)=\dot{\Pgamma}\,t$ with $\dot{\Pgamma}=$constant, and each $\dens_{\Pgamma}\geq 0$ is constant, and let $\Fh:=\Hfn[\Mmu]$ and $\Fu:=\partial_{x}\Fh$.
Then 
\begin{align}
	\label{e.l.pwlinear1}
	\Fh
	= 
	\max_{\Pgamma\in\Snet} \FA_{\dot{\Pgamma},\dens_{\Pgamma}},
\end{align}
$\Fu$ is a piecewise linear solution on $[0,1]\times\R$, and 
\begin{enumerate}
\item \label{l.pwlinear1.shock}
any $\Pl\in\Sbdy(\Fu)$ can have non-entropy shocks only if $\graph(\Pl)\subset\graph(\Pgamma)$ for some $\Pgamma\in\Snet$ and $\Fu_{\Pl\pm}=-2(\dot{\Pgamma}\pm\sqrt{\dens_{\Pgamma}})=$ constant.
\end{enumerate}
\end{lem}
\noindent{}%
A piecewise linear solution may be piecewise linear with respect to multiple tractable partitions.
The statement in Lemma~\ref{l.pwlinear1}\eqref{l.pwlinear1.shock} holds for the $\Sbdy(\Fu)$ in any tractable partition which $\Fu$ is piecewise linear with respect to.

\begin{proof}
We will first construct a function $\Fg$, verify that $\Fv=\partial_x\Fg$ is a piecewise linear solution and satisfies Property~\eqref{l.pwlinear1.shock} for $\Fu\mapsto\Fv$, and then show that $\Fg=\Hfn[\Mmu]$.

\smallskip
\noindent\textbf{Step~1.\ Constructing $\Fg$ and verifying the desired properties.}
Simply define
\begin{align}
	\label{e.l.pwlinear1.Fg}
	\Fg:=\max_{\Pgamma\in\Snet} \FA_{\dot{\Pgamma},\dens_{\Pgamma}}.
\end{align}
Recall from \eqref{e.FA}--\eqref{e.FA.} that $\FA_{\alpha,\beta}(t,x):=t\FA_{\alpha,\beta}(1,x/t)$ and that $\FA_{\alpha,\beta}(1,\cdot)$ is continuous, piecewise linear or quadratic, and coincides with $\Fhh(1,\cdot)$ outside of a bounded interval. 
From these properties, it is straightforward to check that $\Fv:=\partial_x\Fg$ is a piecewise linear solution on $[0,1]\times\R$ and satisfies Property~\eqref{l.pwlinear1.shock} for $\Fu\mapsto\Fv$.

\smallskip
\noindent\textbf{Step~2.\ Showing that $\Fg\leq\Hfn[\Mmu]$.}
Recall $\Mnu_{\alpha,\beta}$ from Example~\ref{ex.single.} and recall that $\FA_{\alpha,\beta} = \Hfn[\Mnu_{\alpha,\beta}]$.
Our $\Mmu$ here is $\sum_{\Pgamma\in\Snet}\Mnu_{\dot{\Pgamma},\dens_{\Pgamma}}$.
For each $\Pgamma\in\Snet$, using \eqref{e.monotone.height} for $\Mnu=\Mnu_{\dot{\Pgamma},\dens_{\Pgamma}}$ and $\Mnu'=\Mmu$ gives $\FA_{\dot{\Pgamma},\dens_\Pgamma} \leq \Hfn[\Mmu]$.
Taking the max over $\Pgamma\in\Snet$ in this inequality gives $\Fg\leq\Hfn[\Mmu]$.

\smallskip
\noindent\textbf{Step~3.\ Showing that $\Fg\geq\Hfn[\Mmu]$.}
The result of Step~1 asserts that $\Fv$ is a piecewise linear solution. 
Let $(\Stime(\Fv),\Sbdy(\Fv),\Sdom(\Fv))$ be the corresponding tractable partition of a light cone $\cone(c)$.
To prove $\Fg\geq\Hfn[\Mmu]$, we apply Lemma~\ref{l.match} with $(\Stime,\Sbdy,\Sdom)=(\Stime(\Fv),\Sbdy(\Fv),\Sdom(\Fv))$ and $(\und{s},\bar{s},\Sbdy',\Sdom')=(0,1,\Snet\cup\{\Pl_{\textL},\Pl_{\textR}\},\Sdom')$, where $\Pl_\textL$ and $\Pl_\textR$ denote the left and right boundaries of $\cone(c)$, and $\Sdom'$ consists of the open connected regions thus partitioned.
In Step~3-1 below, we will verify that our $\Fg$ in \eqref{e.l.pwlinear1.Fg} satisfies \eqref{e.l.match.hopflax}.
In Step~3-2 below, we will show that the corresponding $\Mnu$, defined in \eqref{e.l.match.Mnu}, satisfies $\Mnu \geq \Mmu$.
Since $\und{s}=0$ here, $\til{\Fh}$ defined in \eqref{e.l.match.Fh} matches $\Hfn[\Mnu]$ defined in \eqref{e.Hfn.Ufn}.
Hence, once Steps~3-1 and 3-2 are completed, $\Fg=\til{\Fh}=\Hfn[\Mnu]\geq \Hfn[\Mmu]$.

\smallskip
\noindent\textbf{Step~3-1.\ Verifying that $\Fg$ satisfies \eqref{e.l.match.hopflax}.}
Recall $\Mnu_{\alpha,\beta}$ from Example~\ref{ex.single.} and recall that $\FA_{\alpha,\beta} = \Hfn[\Mnu_{\alpha,\beta}]$.
From this, it is not hard to check that $\HL_{\partial\dom}[\FA_{\alpha,\beta}]=\FA_{\alpha,\beta}$ on $\dom$, for any $\dom$ of the form $\{(t,x): \alpha' t < x<\alpha t\}$ or $\{(t,x): \alpha t < x<\alpha' t\}$ where $\alpha'\in\R$.
Hence $\HL_{\partial\dom}[\FA_{\dot{\Pgamma},\dens_{\Pgamma}}]=\FA_{\dot{\Pgamma},\dens_{\Pgamma}}$ on $\dom$ for every $\dom\in\Sdom'$.
Also, as is readily checked from \eqref{e.hopflax.domain}, the operator $\HL_{\partial\dom}$ commutes with $\max$, namely $\max\HL_{\partial\dom}[\,\cdot\,] = \HL_{\partial\dom}[\max \,\cdot\,]$.
Combining the last two properties shows that $\Fg$ satisfies \eqref{e.l.match.hopflax}.

\smallskip
\noindent\textbf{Step~3-2.\ Showing that $\Mnu \geq \Mmu$.}
This amounts to showing $\dot{\Fg|_{\Pgamma}}+\dot{\Pgamma}^2 \geq \dens_\Pgamma$, for all $\Pgamma\in\Snet$.
Recall that $\Pgamma(t)=\dot{\Pgamma}\cdot t$ and that $\FA_{\alpha,\beta}(t,\alpha t)=t\FA_{\alpha,\beta}(1,\alpha)$.
Hence $(\Fg|_{\Pgamma})(t)$ and $(\FA_{\dot{\Pgamma},\dens_{\Pgamma}}|_{\Pgamma})(t)$ are linear in $t$ and take value $0$ at $t=0$.
This together with \eqref{e.l.pwlinear1.Fg} implies $\dot{\Fg|_{\Pgamma}} \geq \frac{\d~}{\d t}({\FA_{\dot{\Pgamma},\dens_{\Pgamma}}}|_{\Pgamma})$.
Explicit calculations give $\frac{\d~}{\d t}(\FA_{\alpha,\beta}(t,\alpha t)) = \FA_{\alpha,\beta}(1,\alpha)=\beta-\alpha^2$.
Hence $\dot{\Fg|_{\Pgamma}} \geq \dens_{\Pgamma}-\dot{\Pgamma}^2$.
\end{proof}

We next consider piecewise linear and cone-supported $\Mmu$.

\begin{lem}
\label{l.pwlinear}
For any piecewise linear and cone-supported $\Mmu=\sum_{\Pgamma\in\Snet} \dens_{\Pgamma}\,\delta_{\Pgamma}$  (Definition~\ref{d.pwlinear.measure}),
the function $\Fu=\partial_{x}\Hfn[\Mmu]$ is a piecewise linear solution on $[0,1]\times\R$, and
\begin{enumerate}
\item \label{l.pwlinear.shock}
any $\Pl\in\Sbdy(\Fu)$ can have non-entropy shocks only if $\graph(\Pl)\subset\graph(\Pgamma)$ for some $\Pgamma\in\Snet$ and $\Fu_{\Pl\pm}=-2(\dot{\Pgamma}\pm\sqrt{\dens_{\Pgamma}})=$ constant.
\end{enumerate}
\end{lem}

\begin{proof}
To set up the proof, first, without loss of generality, assume that there exist $0=\tau_0<\ldots<\tau_{k}=1$ such that each $\Pgamma\in\Snet$ satisfies $[\start(\Pgamma),\ed(\Pgamma)]=[\tau_{i-1},\tau_{i}]$, for some $i$, and is linear, and that the $\Pgamma$s do not intersect except at $t=\tau_0,\ldots,\tau_{k}$.
Since $\Snet$ is finite, this can always be achieved by cutting the paths shorter.
Let $\Snet_{i}:=\{ \Pgamma\in\Snet : [\start(\Pgamma),\ed(\Pgamma)]=[\tau_{i-1},\tau_{i}] \}$.

Let us use induction in $i$.
More precisely, we divide the region $[0,1]\times\R$ into $[\tau_{i-1},\tau_{i}]\times\R$ for $i=1,\ldots,k$ and prove the analog statements on each $[\tau_{i-1},\tau_{i}]\times\R$ inductively in $i$.
For $i=1$, Lemma~\ref{l.pwlinear1} gives the desired result.
Fix $i\geq 2$, assume the induction hypothesis that $\Fu$ is a piecewise linear solution on $[0,\tau_{i-1}]\times\R$, and set
\begin{align}
	\label{e.fPhi}
	\fPhi(t,x)
	&:=
	\HL_{\tau_{i-1}\to t}[\Fh(\tau_{i-1},\cdot)](x),
	\qquad
	(t,x)\in[\tau_{i-1},\tau_{i}]\times\R.
\end{align}
By Lemma~\ref{l.hopflax}\eqref{l.hopflax.} and the induction hypothesis, $\partial_x\fPhi$ is a piecewise linear solution on $[\tau_{i-1},\tau_{i}]\times\R$ without non-entropy shocks.

To implement the $i$th induction step, we follow the same strategy as the proof of Lemma~\ref{l.pwlinear1}.
The point is that the induction hypothesis has already encoded all earlier history in $\Fh(\tau_{i-1},\cdot)$, so on the new slab it suffices to compare the free Hopf--Lax evolution $\fPhi$ with the single-path contributions $\fPsi_{\Pgamma}$ for $\Pgamma\in\Snet_i$; their maximum records the best way a geodesic can pass through that slab.
We will first construct a function $\Fg\in\Csp([\tau_{i-1},\tau_{i}]\times\R)$, verify that $\Fv=\partial_x\Fg$ is a piecewise linear solution on $[\tau_{i-1},\tau_{i}]\times\R$ and satisfies Property~\eqref{l.pwlinear.shock} for $\Fu\mapsto\Fv$, and then show that $\Fg=\Hfn[\Mmu]$.

\smallskip
\noindent\textbf{Step~1.\ Constructing $\Fg$ and verifying the desired properties.}
We begin by motivating the construction of $\Fg$, which may otherwise seem obscure.
First, note that when $\Snet_{i}=\emptyset$, taking $\Fg=\fPhi$ completes the $i$th induction step: Indeed, $\partial_x\fPhi$ is a piecewise linear solution on $[\tau_{i-1},\tau_{i}]\times\R$ without non-entropy shocks by Lemma~\ref{l.hopflax}\eqref{l.hopflax.}, and $\fPhi=\Hfn[\Mmu]$ on $[\tau_{i-1},\tau_{i}]\times\R$ when $\Snet_{i}=\emptyset$.
In general, $\Snet_{i}\neq\emptyset$, and we construct $\Fg$ as
$
\Fg:=\fPhi \vee \max\{ \fPsi_{\Pgamma} : \Pgamma\in\Snet_{i}\},
$
for a suitable $\fPsi_{\Pgamma}(t,x)$ that ``records the action generated by $\Mmu$ along the path $\Pgamma$''.
Consider the restriction of $\Mmu$ onto $\graph(\Pgamma)$: $\Mmu\lfloor_{\graph(\Pgamma)}= \dens_{\Pgamma}\, \delta_{\Pgamma}$.
A natural choice for $\fPsi_{\Pgamma}$ could be
\begin{align}
	\label{e.p.pwlinear.fPsii}
	\Metric_{\Mmu|_{\graph(\Pgamma)}}(\tau_{i-1},\Pgamma(\tau_{i-1});t,x) + \fPhi|_{\Pgamma}(\tau_{i-1}).
\end{align}
Note that, at $(t,x)=(\tau_{i-1},\Pgamma(\tau_{i-1}))$, the first term in \eqref{e.p.pwlinear.fPsii} is equal to $0$, so adding the constant $\fPhi|_{\Pgamma}(\tau_{i-1})$ shifts it to the same value as $\fPhi$.
The choice \eqref{e.p.pwlinear.fPsii}, however, does not work.
Later in Step~3-2, we will need the property $\dot{\Fg|_{\Pgamma}}+\dot{\Pgamma}^2\geq\dens_{\Pgamma}$.
If we choose \eqref{e.p.pwlinear.fPsii} for $\fPsi_{\Pgamma}$, the resulting $\Fg$ will not satisfy the property.
(Referring to the construction of $\fPsi$ below, one sees that the property fails on the intervals $(\tau_{i-1},a_1),(b_1,a_2),\ldots$.)
With this in mind, we will construct $\fPsi_{\Pgamma}$ as a variant of \eqref{e.p.pwlinear.fPsii} geared toward satisfying the property.

We now complete the construction of $\Fg$ by constructing $\fPsi_{\Pgamma}$.
Consider
\begin{align}
\label{e.Fpsi}
\begin{split}
	\Fpsi_{\Pgamma}(t)
	:=
	&\sup  \big\{ \, (\fPhi|_{\Pgamma})(s) - (s-\tau_{i-1})(\dens_{\Pgamma}-\dot{\Pgamma}^2)
	\, : \, 
	s\in[\tau_{i-1},t] \big\} 
	\\
	& + (t-\tau_{i-1})(\dens_{\Pgamma}-\dot{\Pgamma}^2),
	\qquad
	t\in[\tau_{i-1},\tau_{i}].
\end{split}
\end{align}
In words, the function $\Fpsi_{\Pgamma}$ follows $\fPhi|_{\Pgamma}$ as long as the condition $\dot{\fPhi|_{\Pgamma}} \geq \dens_{\Pgamma}-\dot{\Pgamma}^2$ holds, and when the condition fails, $\Fpsi_{\Pgamma}$ makes a ``linear cut'' with velocity $\dens_{\Pgamma}-\dot{\Pgamma}^2$.
See Figure~\ref{f.Fpsi} for an illustration.
By Lemma~\ref{l.hopflax}\eqref{l.hopflax..}, $\fPhi|_{\Pgamma}(t)$ is continuous and piecewise rational.
This being the case, $\Fpsi_{\Pgamma}$ and $\fPhi|_{\Pgamma}$ differ only on finitely many disjoint intervals $(\start_1,\ed_1)$, \ldots, $(\start_M,\ed_{M})$, or $(\start_{1},\ed_{1})$, \ldots, $(\start_{M},\ed_{M}]$ with $\ed_{M}=\tau_{i}$.
The $\start$s, $\ed$s, and $M$ depend on $\Pgamma$, but to simplify notation we omit the dependence.
The idea is to only ``activate'' $\Mmu\lfloor_{\graph(\Pgamma)}$ on these intervals.
Let $\Pgamma_j:=\Pgamma|_{[\start_j,\ed_j]}$, consider the restriction of $\Mmu$ onto $\graph(\Pgamma_j)$, $\Mmu_{\Pgamma j}:=\dens_{\Pgamma}\, \delta_{\Pgamma_j}$, and let 
\begin{align}
	\label{e.fPsi.j}
	\fPsi_{\Pgamma j}(t,x)
	:=
	\Metric_{\Mmu_{\Pgamma j}}(\start_{j},\Pgamma(\start_{j});t,x) + \fPhi|_{\Pgamma}(\start_{j}).
\end{align}
Again, we add the constant $\fPhi|_{\Pgamma}(\start_{j})$ to make $\fPsi_{\Pgamma j}$ equal to $\fPhi$ at $(t,x)=(\start_{j},\Pgamma(\start_{j}))$.
This defines $\fPsi_{\Pgamma j}$ for $(t,x)\in(\start_j,\tau_{i}]\times\R$ and $(t,x)=(\start_{j},\Pgamma(\start_{j}))$, and we set $\fPsi_{\Pgamma j}(t,x):=-\infty$ for other values of $(t,x)\in[\tau_{i-1},\tau_{i}]\times\R$.
Finally, let
\begin{align}
	\label{e.fPsi}
	\fPsi_{\Pgamma} := \fPsi_{\Pgamma 1}\vee\cdots\vee	\fPsi_{\Pgamma M},
	\qquad
	\Fg:=\fPhi \vee \big\{ \fPsi_{\Pgamma} : \Pgamma\in\Snet_{i} \},
\end{align}
with the convention that $\fPsi_{\Pgamma}:=-\infty$ when $M=0$.

\begin{figure}
\centering
\fbox{\includegraphics[width=.8\linewidth]{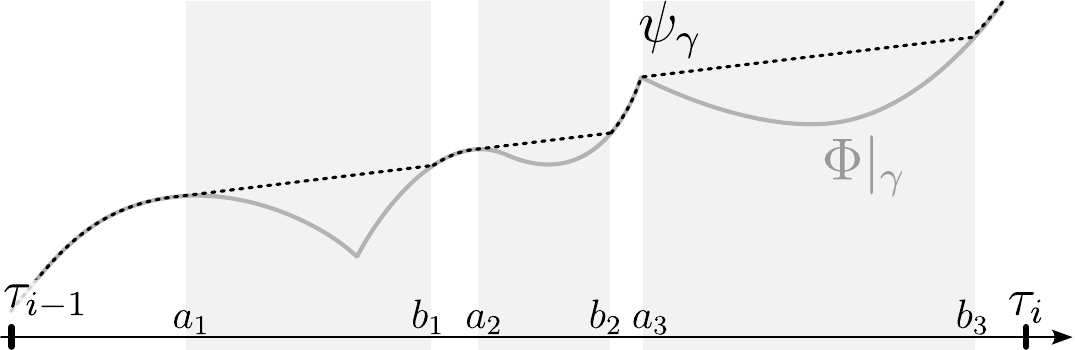}}
\caption{%
	The graphs of $\fPhi|_{\Pgamma}$ and $\Fpsi_{\Pgamma}$.
	The linear segments have velocity/slope $\dens_{\Pgamma}-\dot{\Pgamma}^2$. 
}
\label{f.Fpsi}
\end{figure}

Let us pause and make a conceptual remark about the above construction.
The construction essentially seeks to remove the ineffective part of $\Mmu\lfloor_{[\tau_{i-1},\tau_i]\times\R}$, namely the part of $\Mmu$ on $\graph(\Pgamma|_{(\tau_{i-1},a_1)})$, $\graph(\Pgamma|_{(b_1,a_2)})$, \ldots.
The reason is that, along these segments of $\Pgamma$, the Hopf--Lax evolution from time $\tau_{i-1}$, namely $\fPhi$ defined in \eqref{e.fPhi}, outcompetes the contribution from $\Mmu\lfloor_{[\tau_{i-1},\tau_i]\times\R}$.
We clarify that our construction does \emph{not} mean that all of the remaining part of $\Mmu$, namely the part of $\Mmu$ on $\graph(\Pgamma|_{[a_1,b_1]})$, $\graph(\Pgamma|_{[a_2,b_2]})$, \ldots, is visible to $\Fg$.
Refer back to the proof of Lemma~\ref{l.pwlinear1}.
In \eqref{e.l.pwlinear1}, it is possible that one of the terms on the right-hand side completely dominates the other, namely $\FA_{\dot{\Pgamma_1},\dens_{\Pgamma_1}} \geq \FA_{\dot{\Pgamma_2},\dens_{\Pgamma_2}}$, so that the $\Fg$ there does not ``see'' $\FA_{\dot{\Pgamma_2},\dens_{\Pgamma_2}}$.
The same or even more complicated ``takeover'' could happen in the $\fPsi_\gamma$ defined here.
What is important for the proof to work is the property $\dot{\Fg|_{\Pgamma}}+\dot{\Pgamma}^2\geq\dens_{\Pgamma}$, which will be used in Step 3-2 below.%

To prepare for verifying that $\Fg$ satisfies the desired property, let us derive some explicit formulas for $\fPsi_{\Pgamma j}$.
The function is described by $\Metric_{\Mmu_{\Pgamma j}}$ geodesics that start from $(\start_j,\Pgamma(\start_j))$.
Since $\dot{\Pgamma}$ and $\dens_{\Pgamma}$ are constant, these geodesics can be found explicitly, which we next describe.
Let
\begin{align}
	\label{e.fPsi.alpha}
	\alpha_{\Pgamma\pm} := \dot{\Pgamma} \pm \sqrt{\dens_{\Pgamma}}\ .
\end{align}
In Figure~\ref{f.fPsij}, the middle path is $\Pgamma|_{[\start_j,\ed_j]}$, and the other four paths are linear with velocities $\alpha_{\Pgamma-}$ and $\alpha_{\Pgamma+}$ and starting points $(\start_j,\Pgamma(\start_j))$ and $(\ed_j,\Pgamma(\ed_j))$.
These paths partition $(\tau_{i-1},\tau_{i})\times\R$ into the open regions $\SS_{\Pgamma j}$, $\SS_{\Pgamma j}'$, $\SR_{\Pgamma j-}$, and $\SR_{\Pgamma j+}$ depicted in Figure~\ref{f.fPsij}.
Now consider the $\Metric_{\Mmu_{\Pgamma j}}$ geodesic that connects $(\start_j,\Pgamma(\start_j))$ to $(t,x)$.
For $(t,x)\in\SS_{\Pgamma j}\cap((\start_j,\tau_{i}]\times\R)$ the geodesic is linear.
Further,
\begin{align}
	\label{e.fPsi.S.sup}
	\fPsi_{\Pgamma j}(t,x) 
	&= 
	\sup \big\{ |\Pchi|_{\diri} : (\start_j,\Pgamma(\start_j)) \xrightarrow{\Pchi} (t,x) \big\} + \fPhi|_\Pgamma(\start_j)
	\\
	\label{e.fPsi.S.formula}
	&= 
	-\tfrac{(x-\Pgamma(\start_j))^2}{(t-\start_j)_+}  + \fPhi|_\Pgamma(\start_j),
	\qquad
	(t,x)\in\SS_{\Pgamma j},
\end{align}
where $(\cdot)_+:=\max\{\cdot,0\}$.
When $(t,x)\in\SS_{\Pgamma j}\cap((\start_j,\tau_{i}]\times\R)$, the expressions in \eqref{e.fPsi.S.sup}--\eqref{e.fPsi.S.formula} follow from linearity of the geodesic, and
these expressions extend to all $(t,x)\in\SS_{\Pgamma j}$ under the conventions $\sup\emptyset=-\infty$, $0/0:=0$, and $\alpha^2/0:=+\infty$ for $\alpha\neq 0$.
Next, for $(t,x)\in\SS'_{\Pgamma j}$ the geodesic follows $\Pgamma$ until time $\ed_j$ and then becomes another linear path.
When following $\Pgamma$, the geodesic sees a change of its length by $(\ed_j-\start_j)(\dens_{\Pgamma}-\dot{\Pgamma}^2)$.
Also, by the definitions of $\start_{j}$, $\ed_{j}$, and $\Fpsi_{\Pgamma}$, we have $(\ed_j-\start_j)(\dens_{\Pgamma}-\dot{\Pgamma}^2)+\fPhi|_{\Pgamma}(\start_j)=\Fpsi_{\Pgamma}(\ed_j)$; see Figure~\ref{f.Fpsi}.
Therefore,
\begin{align}
	\label{e.fPsi.S'.sup}
	\fPsi_{\Pgamma j}(t,x) 
	&= 
	\sup \big\{ |\Pchi|_{\diri} + \Fpsi_{\Pgamma}(\ed_j) : (\ed_j,\Pgamma(\ed_j)) \xrightarrow{\Pchi} (t,x) \big\}
	\\
	\label{e.fPsi.S'.formula}
	&= 
	-\tfrac{(x-\Pgamma(\ed_j))^2}{t-\ed_j} + \Fpsi_{\Pgamma}(\ed_j),
	\qquad
	(t,x)\in\SS'_{\Pgamma j}.
\end{align}
For $(t,x)\in\SR_{\Pgamma j\pm}$ the geodesic follows $\Pgamma$ until time $s$, for some $s\in(\start_j,\ed_j)$, and then becomes the linear path with velocity $\alpha_{\Pgamma\pm}$, giving 
\begin{align}
\label{e.fPsi.R.formula}
\begin{split}
	\fPsi_{\Pgamma j}(t,x) 
	=&
	-2\alpha_{\Pgamma\pm}\,(x-(t-\start_{j})\dot{\Pgamma}-\Pgamma(\start_{j})) 
	\\
	&+ (t-\start_{j})(\dens_{\Pgamma}-\dot{\Pgamma}^2) 
	+ \fPhi|_{\Pgamma}(\start_{j}),
	\quad
	(t,x)\in\SR_{\Pgamma j\pm}.
\end{split}
\end{align}

\begin{figure}
\hfill
\begin{minipage}[t]{.48\linewidth}
\fbox{\includegraphics[width=\linewidth]{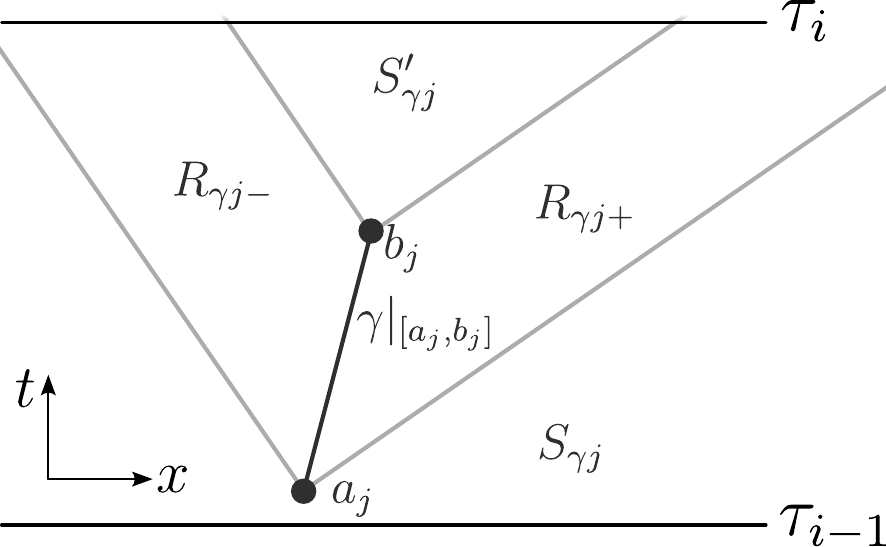}}
\caption{}
\label{f.fPsij}
\end{minipage}
\hfill
\setlength{\fboxsep}{0pt}
\begin{minipage}[t]{.48\linewidth}
\fbox{\includegraphics[width=\linewidth]{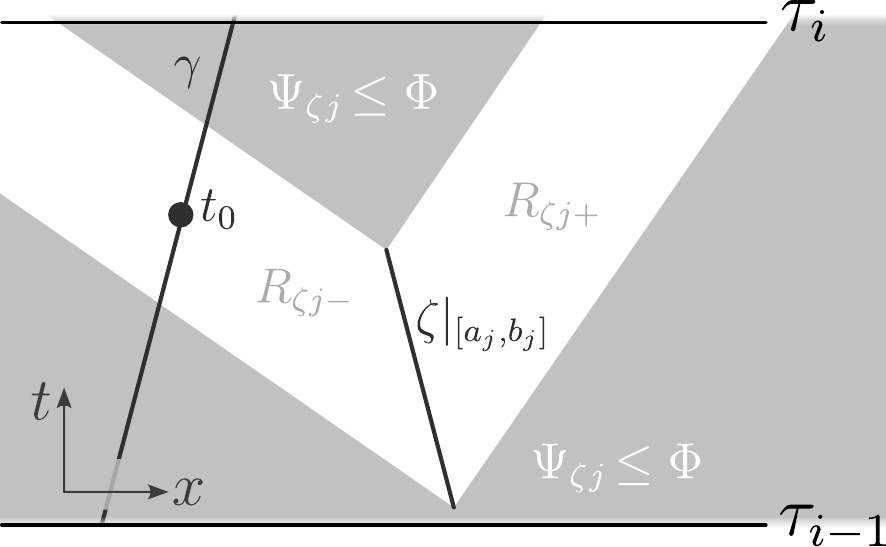}}
\caption{}
\label{f.fPsijj}
\end{minipage}
\end{figure}

We are now ready to verify that $\Fg$ satisfies the desired properties.
Recall that $\fPhi$ is a piecewise linear solution on $[\tau_{i-1},\tau_{i}]\times\R$ without non-entropy shocks.
Using this and the explicit formulas of $\fPsi_{\Pgamma j}$, namely \eqref{e.fPsi.S.formula}, \eqref{e.fPsi.S'.formula}, \eqref{e.fPsi.R.formula}, it is not hard to check that $\Fv=\partial_x\Fg$ (with $\Fg$ defined in \eqref{e.fPsi}) is a piecewise linear solution on $[\tau_{i-1},\tau_{i}]\times\R$.
Further, since $\fPhi$ does not have non-entropy shocks, non-entropy shocks of $\Fv$ can only occur at those $(t,x)$ such that $\Fg=\fPsi_{\Pgamma j}$ in a neighborhood of $(t,x)$, for some $\Pgamma\in\Snet_{i}$ and some $j=1,\ldots,M$.
As is readily checked from the explicit formulas of $\fPsi_{\Pgamma j}$, discontinuity of $\partial_x\fPsi_{\Pgamma j}$ happens only on $\graph(\Pgamma|_{[\start_j,\ed_j]})$, with $(\partial_x\fPsi_{\Pgamma j})(t,\Pgamma(t)^\pm)=-2\alpha_{\Pgamma\pm}=-2(\dot{\Pgamma}\pm\sqrt{\dens_{\Pgamma}})$.
Hence $\Fv$ satisfies Property~\eqref{l.pwlinear.shock} with $\Fu\mapsto\Fv$ on $[\tau_{i-1},\tau_{i}]\times\R$.

\smallskip
\noindent\textbf{Step~2.\ Showing that $\Fg\leq\Hfn[\Mmu]$ on $[\tau_{i-1},\tau_{i}]\times\R$.}
By the definition of $\fPhi$ and $\fPsi_{\Pgamma j}$ (see \eqref{e.fPhi} and \eqref{e.fPsi.j}), one readily checks that $\fPhi\leq\Hfn[\Mmu]$ and $\fPsi_{\Pgamma j}\leq\Hfn[\Mmu]$ on $[\tau_{i-1},\tau_{i}]\times\R$.
Since $\Fg$ is constructed as the maximum of $\fPhi$ and the $\fPsi_{\Pgamma j}$s, the desired result follows.

\smallskip
\noindent\textbf{Step~3.\ Showing that $\Fg\geq\Hfn[\Mmu]$ on $[\tau_{i-1},\tau_{i}]\times\R$.}
The result of Step~1 asserts that $\Fv$ is a piecewise linear solution on $[\tau_{i-1},\tau_{i}]\times\R$.
Let $(\Stime(\Fv),\Sbdy(\Fv),\Sdom(\Fv))$ be the corresponding tractable partition of a light cone $\cone(c)$.
To prove $\Fg\geq\Hfn[\Mmu]$, we apply Lemma~\ref{l.match} with $(\Stime,\Sbdy,\Sdom)=(\Stime(\Fv),\Sbdy(\Fv),\Sdom(\Fv))$ and $(\und{s},\bar{s},\Sbdy',\Sdom')=(\{\tau_{i-1},\tau_{i}\},\Snet\cup\{\Pl_{i\textL},\Pl_{i\textR}\},\Sdom')$, where $\Pl_{i\textL}$ and $\Pl_{i\textR}$ denote the left and right boundaries of $\cone(c)\cap([\tau_{i-1},\tau_{i}]\times\R)$, and $\Sdom'$ consists of the open connected regions thus partitioned.
Note that we can make $\graph(\Snet_{i})\subset\graph(\Sbdy(\Fv))$ by adding those $\Pgamma\in\Snet_{i}$ into $\Sbdy(\Fv)$.
Because the paths in $\Sbdy(\Fv)$ take the form~\eqref{e.pwlinear.Pl} and because the paths in $\Snet_{i}$ are linear, they intersect at most finitely many times. 
Hence, by cutting these paths shorter, we make sure that they form the boundary paths of a tractable partition. 

In Step~3-1 below, we will verify that our $\Fg$ in \eqref{e.l.pwlinear1.Fg} satisfies \eqref{e.l.match.hopflax}.
In Step~3-2 below, we will show that the corresponding $\Mnu$, defined in \eqref{e.l.match.Mnu}, satisfies $\Mnu \geq \Mmu\lfloor_{[\tau_{i-1},\tau_{i}]\times\R}$.
Under the current setup and given that $\Fg(\tau_{i-1},\cdot)=\Hfn[\Mmu](\tau_{i-1},\cdot)$ by the induction hypothesis, it is readily checked that the $\til{\Fh}$ defined in \eqref{e.l.match.Fh} is equal to $\Hfn[\,\Mnu+\Mmu\lfloor_{[0,\tau_{i-1}]\times\R}\,]$.
Hence, once Steps~3-1 and 3-2 are completed, $\Fg=\til{\Fh}=\Hfn[\,\Mnu+\Mmu\lfloor_{[0,\tau_{i-1}]\times\R}\,]\geq\Hfn[\Mmu]$ on $[\tau_{i-1},\tau_{i}]\times\R$.

\smallskip
\noindent\textbf{Step~3-1.\ Verifying that $\Fg$ satisfies \eqref{e.l.match.hopflax}.}
From \eqref{e.fPhi} and \eqref{e.fPsi.j}, it is not hard to check that $\HL_{\partial\dom}[\fPhi]=\fPhi$ and that $\HL_{\partial\dom}[\fPsi_{\Pgamma j}]=\fPsi_{\Pgamma j}$ on every $\dom\in\Sdom'$.
Since the operator $\HL_{\partial\dom}$ commutes with $\max$, \eqref{e.l.match.hopflax} holds.

\smallskip
\noindent\textbf{Step~3-2.\ Showing that $\Mnu \geq \Mmu\lfloor_{[\tau_{i-1},\tau_{i}]\times\R}$.}
Let us prepare some properties.
First, 
\begin{align}
	\label{e.p.pwlinear.step3-2.}
	\fPsi_{\Pzeta j} \leq \fPhi \qquad \text{on } \bar{\SS_{\Pzeta j}\cup\SS'_{\Pzeta j}},
	\qquad
	\text{for all } \Pzeta\in\Snet_{i} \text{ and } j=1,\ldots,M.
\end{align}
To see why, consider first $(t,x)\in\SS_{\Pzeta j}$.
Use \eqref{e.fPhi} to write
\begin{align}
	\fPhi|_{\Pzeta}(\start_j)
	=
	\sup\{|\Pchi'|_{\diri} + \Fh(\tau_{i-1},\Pchi'(\tau_{i-1})): \{\tau_{i-1}\}\times\R\xrightarrow{\Pchi'} (\start_j,\Pzeta(\start_j)) \}
\end{align}
and insert it into \eqref{e.fPsi.S.sup} for $\Pgamma\mapsto\Pzeta$ to get
\begin{align}
	\fPsi_{\Pzeta j}(t,x)
	=
	\sup\big\{ |\Pchi'\cup\Pchi|_{\diri} + \Fh(\tau_{i-1},\Pchi'(\tau_{i-1}))
	: 
	\{\tau_{i-1}\}\times\R\xrightarrow{\Pchi'} (\start_j,\Pzeta(\start_j)) \xrightarrow{\Pchi} (t,x) \big\}.
\end{align}
We see that this is $\leq\HL_{\tau_{i-1}\to t}[\Fh(\tau_{i-1},\cdot)](x)=\fPhi(t,x)$.
As for $(t,x)\in\SS'_{\Pzeta j}$, by the definition of $\ed_{j}$, we have $\Fpsi_{\Pzeta}(\ed_j)=\fPhi|_{\Pzeta}(\ed_j)$ (unless $\ed_j=\tau_{i}$, which implies $\SS'_{\Pzeta j}=\emptyset$).
Hence the same argument applies.
So far we have proven \eqref{e.p.pwlinear.step3-2.} on $\SS_{\Pzeta j}\cup\SS'_{\Pzeta j}$.
Indeed, $\fPhi$ is continuous and $\fPsi_{\Pzeta j}$ is continuous except where it is $-\infty$.
Hence \eqref{e.p.pwlinear.step3-2.} follows.
Next, given \eqref{e.p.pwlinear.step3-2.}, it is not hard check from the definitions of $\Fpsi_{\Pgamma}$ and $\fPsi_{\Pgamma}$ that
\begin{align}
	\label{e.p.pwlinear.step3-2..}
	(\fPsi_{\Pgamma}|_\Pgamma) \vee \fPhi|_{\Pgamma}
	=
	\Fpsi_{\Pgamma},
	\quad	
	\Pgamma\in\Snet_{i}.
\end{align}
Also, from the definition of $\Fpsi_{\Pgamma}$, it follows that the function is continuous, piecewise rational (so is differentiable everywhere except at finitely many points), and satisfies
\begin{align}
	\label{e.p.pwlinear.step3-2..-2}
	\dot{\Fpsi_{\Pgamma}} \geq \dens_{\Pgamma} - \dot{\Pgamma}^2 \text{ wherever } \dot{\Fpsi_{\Pgamma}} \text{ exists,}
	\qquad
	\Pgamma\in\Snet_{i}.
\end{align}

We now prove $\Mnu \geq \Mmu\lfloor_{[\tau_{i-1},\tau_{i}]\times\R}$.
This amounts to proving that $\dot{\Fg|_{\Pgamma}} \geq \dens_{\Pgamma} - \dot{\Pgamma}^2$, for all $\Pgamma\in\Snet_{i}$ and for a.e.\ on $[\tau_{i-1},\tau_{i}]$.
By the construction of $\Fg$ in Step~1, $\Fg|_{\Pgamma}$ is continuous and piecewise rational, so is differentiable everywhere except at finitely many points.
Fix any $\Pgamma\in\Snet_{i}$ and any $t_0\in(\tau_{i-1},\tau_{i})$ such that $\dot{\Fg|_{\Pgamma}}(t_0)$ and $\dot{\Fpsi_{\Pgamma}}(t_0)$ exist.
We show that $\dot{\Fg|_{\Pgamma}}(t_0) \geq \dens_{\Pgamma} - \dot{\Pgamma}^2$ in two cases separately.
\begin{description}
\item[Case~1. If $\Fg|_{\Pgamma}(t_0)=(\fPhi|_{\Pgamma}\vee\fPsi_{\Pgamma}|_{\Pgamma})(t_0)$]
Recall that $\Fg$ is the maximum of $\fPhi$ and $\fPsi_{\Pzeta}$ over $\Pzeta\in\Snet_{i}$.
This property, the assumption of Case~1, and \eqref{e.p.pwlinear.step3-2..} together give
\begin{align}
	\Fg|_{\Pgamma}(t_0)=\Fpsi_{\Pgamma}(t_0),
	\quad
	\text{ and }
	\quad
	\Fg|_{\Pgamma}(t)\geq\Fpsi_{\Pgamma}(t), \text{ for all } t \in (\tau_{i-1},\tau_{i}).
\end{align}
Combining this with the assumption on $t_0$ gives $\dot{\Fg|_{\Pgamma}}(t_0)=\dot{\Fpsi_{\Pgamma}}(t_0)$.
This together with \eqref{e.p.pwlinear.step3-2..-2} gives the desired inequality. 

\item[Case~2. If $\Fg|_{\Pgamma}(t_0)\neq (\fPhi|_{\Pgamma}\vee\fPsi_{\Pgamma}|_{\Pgamma})(t_0)$]
Since $\Fg$ is the maximum of $\fPhi$ and $\fPsi_{\Pzeta}$ over $\Pzeta\in\Snet_{i}$, in this case, there exist $\Pzeta\in\Snet_{i}\setminus\{\Pgamma\}$ and $j\in\{1,\ldots,M\}$, such that
\begin{align}
	\label{e.p.pwlinear.step3-2...}
	\Fg|_{\Pgamma}
	=
	\fPsi_{\Pzeta j}|_{\Pgamma}
	>
	(\fPsi_{\Pgamma}|_\Pgamma \vee \fPhi|_{\Pgamma})
	=
	\Fpsi_{\Pgamma},
	\qquad
	\text{in a neighborhood of } t_0.
\end{align}
Consider $\und{t}_0:=\sup\{ t\in[\tau_{i-1},t_0) : \fPsi_{\Pzeta j}|_{\Pgamma}(t)\leq\Fpsi_{\Pgamma}(t) \}$.
By \eqref{e.p.pwlinear.step3-2.} and  \eqref{e.p.pwlinear.step3-2..-2}, we have $(\und{t}_0,\Pgamma(\und{t}_0)),(t_0,\Pgamma(t_0))\in\bar{\SR_{\Pzeta j+}}$ or $(\und{t}_0,\Pgamma(\und{t}_0)),(t_0,\Pgamma(t_0))\in\bar{\SR_{\Pzeta j-}}$; see Figure~\ref{f.fPsijj} for an illustration.
Given this property, using \eqref{e.fPsi.R.formula} for $\Pgamma\mapsto\Pzeta$ and the linearity of $\Pgamma$ shows that $\fPsi_{\Pzeta j}|_{\Pgamma}(t)$ is \emph{linear} on $[\und{t}_0,t_0]$.
Also, by the definition of $\und{t}_0$, we have $\fPsi_{\Pzeta j}|_{\Pgamma}>\Fpsi_{\Pgamma} $ on $(\und{t}_0,t_0]$ and $\fPsi_{\Pzeta j}|_{\Pgamma}=\Fpsi_{\Pgamma}$ at $\und{t}_0$, so
\begin{align}
	\label{e.p.pwlinear.step3-2....}
	\fPsi_{\Pzeta j}|_{\Pgamma}(t_0) - \fPsi_{\Pzeta j}|_{\Pgamma}(\und{t}_0) 
	>
	\Fpsi_{\Pgamma}(t_0) - \Fpsi_{\Pgamma}(\und{t}_0) 
	= 
	\int_{\und{t}_0}^{t_0} \d t \, \dot{\Fpsi_{\Pgamma}}.
\end{align}
Divide \eqref{e.p.pwlinear.step3-2....} by $t_0-\und{t}_0$; in the leftmost expression, use the property that $\fPsi_{\Pzeta j}|_{\Pgamma}(t)$ is linear on $[\und{t}_0,t_0]$; in the rightmost expression, use \eqref{e.p.pwlinear.step3-2..}.
Doing so gives $\frac{\d~}{\d t}({\fPsi}_{\Pzeta j}|_{\Pgamma})(t_0)>\dens_{\Pgamma} - \dot{\Pgamma}^2$.
The left-hand side is equal to $\dot{\Fg|_{\Pgamma}}(t_0)$ by \eqref{e.p.pwlinear.step3-2...}, so the desired inequality follows.
\end{description}
We have shown $\dot{\Fg|_{\Pgamma}}(t_0) \geq \dens_{\Pgamma} - \dot{\Pgamma}^2$ in both cases and hence have completed the proof.
\end{proof}

We extract what we need from Lemma~\ref{l.pwlinear} for later and put it in Proposition~\ref{p.pwlinear}.

\begin{customprop}{\ref*{p.pwlinear.}'}\label{p.pwlinear}
Take any piecewise linear and cone-supported $\Mmu$.
\begin{enumerate}
	\item \label{p.pwlinear.rate}
	The function $\Fu:=\partial_x\Hfn[\Mmu]$ is a tractable solution, with $\Mmu\geq\Mnon[\Fh]$ and $\rate(\Metric_{\Mmu})\geq\rate(\Mnon[\Fh])$.
	\item \label{p.pwlinear.path}
	For any $(t,x)\in(0,1]\times\R$, there exists a path $(t,x)\xrightarrow{\Pgeo}\{1\}\times\R$ such that $|\Pgeo|_{\Metric_{\Ment[\Fh]}}+\Fh(t,x)=\Fh(1,\Pgeo(1))$.
\end{enumerate}
\end{customprop}
\begin{proof}
\eqref{p.pwlinear.rate}\
The first statement follows from Lemma~\ref{l.pwlinear} because a piecewise linear solution is a tractable solution.
The second statement follows from Lemma~\ref{l.pwlinear}\eqref{l.pwlinear.shock} and \eqref{e.Mnon}.
The last statement follows from the second statement and \eqref{e.monotone.weak}.

\eqref{p.pwlinear.path}\
Let $(\Stime(\Fu),\Sbdy(\Fu),\Sdom(\Fu))$ be the tractable partition associated with $\Fu$.
To construct the path $\Pgeo$, start at $(t,x)$; follow a characteristic when being in $\SD\in\Sdom(\Fu)$ and follow $\Pl\in\Sbdy(\Fu)$ when hitting such an $\Pl$.
This procedure terminates in finite steps.
We need to ensure that $\Pgeo$ does not take any of $\Pl\in\Sbdy(\Fu)$ with $\Fu_{\Pl-}>\Fu_{\Pl+}$.
By Lemma~\ref{l.pwlinear}\eqref{l.pwlinear.shock}, any $\Fu_{\Pl-}>\Fu_{\Pl+}$ happens only when $\Fu_{\Pl-},\Fu_{\Pl+},\dot{\Pl}$ are all constant.
In this case, every $(t,\Pl(t))$ with $t\in(\start(\Pl),\ed(\Pl))$ is connected to a characteristic that is emanating from $\graph(\Pl)$; see Figure~\ref{f.antishock} for an illustration.
By taking a sequence $t_k\downarrow \start(\Pl)$, we see that the point $(\start(\Pl),\Pl(\start(\Pl)))$ is connected to either a characteristic or an $\Pl'\in\Sbdy(\Fu)$ with $\Fu_{\Pl'-}\leq\Fu_{\Pl'+}$; see Figure~\ref{f.antishock} for an illustration.
This being the case, we can construct $\Pgeo$ such that $\Pgeo=\Pgeo_1\cup\ldots\cup\Pgeo_{M}$, where each $\Pgeo_j$ is a segment of
\begin{enumerate}
\item a characteristic within a $\SD\in\Sdom(\Fu)$, or
\item a path $\Pl'\in\Sbdy(\Fu)$ with $\Fu_{\Pl'-}\leq\Fu_{\Pl'+}$.
\end{enumerate}
Let $t=s_0<s_1<\ldots<s_M=1$ denote the starting and ending time of these $\Pgeo_j$s.
\begin{enumerate}
\item Using $\partial_t\Fh=\frac{1}{4}\Fu^2$ and \eqref{e.characteristics} gives $\Fh(s_j,\Pgeo(s_j))-\Fh(s_{j-1},\Pgeo(s_{j-1}))=|\Pgeo_j|_{\diri}$, which is equal to $|\Pgeo_j|_{\Metric_{\Ment[\Fh]}}$ because $\graph(\Pgeo_j|_{(s_{j-1},s_j)})\subset\SD$ for some $\SD\in\Sdom(\Fu)$.
\item Using \eqref{e.Fg.dot..} and \eqref{e.Ment} gives $\Fh(s_j,\Pgeo(s_j))-\Fh(s_{j-1},\Pgeo(s_{j-1}))=|\Pgeo_j|_{\Metric_{\Ment[\Fh]}}$.
\end{enumerate}
Adding these equalities over $j=1,\ldots,M$ gives the desired result.
\end{proof}

\begin{figure}
\fbox{\includegraphics[width=.4\linewidth]{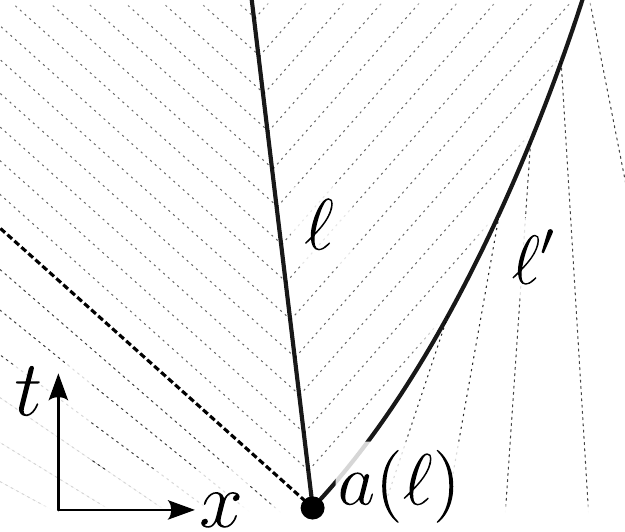}}
\caption{%
	The dashed paths are characteristics.
}
\label{f.antishock}
\end{figure}

\subsection{Approximation}
\label{s.tools.approx}
Here we prepare the tools for approximating paths and measures.

\begin{lem}\label{l.approx.path}
Take any sequence of measures $\{\Mmu_{n}\}_{n}\subset\Msp_+$ and any sequence of paths $\{\Peta_{n}\}_{n}\subset\Hsp^1[\start,\ed]$ such that $\inf_{n} |\Peta_{n}|_{\Metric_{\Mmu_{n}}}>-\infty$, $\sup_n \rate(\Metric_{\Mmu_n})<\infty$, and $\Peta_n(\start)=x$ is independent of $n$.
There exists a subsequence $\{\Peta_{n_j}\}_j$ and a path $\Peta\in\Hsp^1[\start,\ed]$ such that 
\begin{align}
\label{e.l.approx.path}
\begin{split}
	&\limsup_{j\to\infty} \norm{\Peta_{n_j}-\Peta}_{\Lsp^\infty[\start,\ed]} =0,
\\
	&\limsup_{j\to\infty} |\Peta_{n_j}|_{\Metric_{\Mmu_{n_j}}} \leq  |\Peta|_{\diri} + (\ed-\start)^{1/3}\Big(\frac{3}{4}\limsup_{n\to\infty}\rate(\Metric_{\Mmu_{n}})\Big)^{2/3}.
\end{split}
\end{align}
\end{lem}
\begin{proof}
Since $\{\eta\in\Hsp^1[\start,\ed]:\eta(\start)=x\}$ is compactly embedded into $\Csp[\start,\ed]$, there exists a subsequence $\{\Peta_{n_j}\}_j$ and an $\Peta$ such that the first limit in \eqref{e.l.approx.path} holds.
Next, use \eqref{e.length} for $\Peta\mapsto\Peta_{n_j}$ and consider the $j\to\infty$ limsups of the first and second integrals.
The first limsup is at most $\limsup_{n}(b-a)^{1/3}(\frac{3}{4}\rate(\Metric_{\Mmu_{n}}))^{2/3}$ by H\"{o}lder's inequality.
Indeed, writing $f_j(t):=\sum_{\Pgamma\in\Snet}\dens_{\Pgamma}(t)\ind_{\Pgamma(t)=\Peta_{n_j}(t)}$, we have $\int_{\start}^{\ed} f_j(t)\,\d t\leq \|f_j\|_{\Lsp^{3/2}[\start,\ed]}\,(\ed-\start)^{1/3}$.
Because the supporting paths are internally disjoint, $f_j(t)^{3/2}=\sum_{\Pgamma\in\Snet}\dens_{\Pgamma}(t)^{3/2}\ind_{\Pgamma(t)=\Peta_{n_j}(t)}\leq\sum_{\Pgamma\in\Snet}\dens_{\Pgamma}(t)^{3/2}$, so this term is controlled by the same $L^{3/2}$ quantity that appears in $\rate(\Metric_{\Mmu_{n_j}})$.
The second limsup is at most $|\Peta|_{\diri}$ because $\Pgamma\mapsto\int_{[\start,\ed]}\d t\,\dot{\Pgamma}^2$ is lower semicontinuous under the uniform topology.
\end{proof}

\begin{defn}\label{d.approx}
Take any $\Mmu=\sum_{\Pgamma\in\Snet}\dens_\Pgamma\,\delta_{\Pgamma}\in\Msp_+$ such that $\rate(\Metric_{\Mmu})<\infty$.
Let us construct a piecewise linear approximation of $(\Mmu,\Snet)$ as follows.
First we need to perform a surgery on $\Snet$.
Examine any $\Pgamma\in\Snet$, if $\graph(\Pgamma)$ intersects with $\{0\}\times\R$ at any point other than $(0,0)$, divide the path $\Pgamma$ into $\Pgamma|_{[1/2,\ed(\Pgamma)]}$, $\Pgamma|_{[1/3,1/2]}$, $\Pgamma|_{[1/4,1/3]}$, \ldots.
Denote the post-surgery set still by $\Snet$.
This surgery does not change $\Mmu$ and ensures that for every $\Pgamma\in\Snet$, there exists a $c_\Pgamma<\infty$ such that $\graph(\Pgamma)\subset\cone(c_\Pgamma)$.
Next, index the paths in $\Snet$ as $\Pgamma_1,\Pgamma_2,\ldots$ and put $[\start(\Pgamma_i),\ed(\Pgamma_i)]=[\start_i,\ed_i]$.
For each $\Pgamma_{i}$, construct a sequence of piecewise linear paths $\Pgamma_{i1},\Pgamma_{i2},\ldots\in\Hsp^1[\start_{i},\ed_{i}]$ with  $\Pgamma_{in}(\start_{i})=\Pgamma_{i}(\start_{i})$ and $\Pgamma_{in}(\ed_{i})=\Pgamma_{i}(\ed_{i})$ for all $n$, and a sequence of piecewise constant $\dens_{i1},\dens_{i2},\ldots\in\Lsp^1[\start_{j},\ed_{j}]$, such that
\begin{align}
	\label{e.approx}
	\sum_{i=1}^\infty \int_{\start_i}^{\ed_i} \d t \, | \dot{\Pgamma_{in}} - \dot{\Pgamma_{i}} |^{2} \xrightarrow{n\to\infty} 0,
	\qquad
	\sum_{i=1}^\infty \int_{\start_i}^{\ed_i} \d t \, 
	\big| \dens_{in} - \dens_{\Pgamma_i} \big|^{3/2}
	\xrightarrow{n\to\infty} 0.	 
\end{align}
Let $\Mmu_{n}:=\sum_{i=1}^n \dens_{in}\delta_{\Pgamma_{in}}$.
This way, each $\Mmu_{n}$ is piecewise linear and cone-supported (Definition~\ref{d.pwlinear.measure}), and $\rate(\Metric_{\Mmu_{n}})\to \rate(\Metric_{\Mmu})$ as $n\to\infty$.
\end{defn}

\begin{lem}\label{l.approx.metric}
For the $\Mmu$ and $\Mmu_n$ as in Definition~\ref{d.approx} and for all $(s,y,s',y')\in\R^4_{\uparrow}$,
\begin{align}
	\limsup_{n\to\infty}\Metric_{\Mmu_{n}}(s,y;s',y')\leq \Metric_{\Mmu}(s,y;s',y').
\end{align}
\end{lem}
\begin{proof}
Take any $(s,y,s',y')\in\R^4_{\uparrow}$ and any $\Metric_{\Mmu_{n}}$ geodesic $(s,y)\xrightarrow{\Pgeo_{n}}(s',y')$.
Assume $\limsup_{n}|\Pgeo_{n}|_{\Mmu_{n}}>-\infty$, otherwise the desired result follows immediately. 
Given the assumption, the proof of Lemma~\ref{l.approx.path} shows that, after passing to a subsequence, $\Pgeo_{n}$ converges pointwise (uniformly in fact) to a path $\Pgeo$ and $\limsup_{n}\int_{[s,s']} \d t (-\dot{\Pgeo}^2_{n})\leq\int_{[s,s']} \d t (-\dot{\Pgeo}^2)$.
Using this and \eqref{e.length} for $\Peta\mapsto\Pgeo_{n}$ give
\begin{align}
	\label{e.l.approx.metric.}
	\limsup_{n\to\infty}\Metric_{\Mmu_{n}}(s,y;s',y')
	\leq
	\limsup_{n\to\infty}
	\int_{s}^{s'} \d t\, \sum_{i=1}^{\infty} \dens_{\Pgamma_{in}} \ind_{\Pgamma_{in}=\Pgeo_{n}}
	+
	\int_{s}^{s'} \d t\,\big( -\dot{\Pgeo}^2 \big).
\end{align}
Next, by construction, for every fixed $i$ and as $n\to\infty$, the path $\Pgamma_{in}$ converges to $\Pgamma_{i}$ pointwise (uniformly in fact). 
This and the pointwise convergence of $\Pgeo_{n}$ to $\Pgeo$ together give
$
\limsup_{n} \sum_{i=1}^\infty \ind_{\Pgamma_{in}=\Pgeo_{n}}
\leq
\sum_{i=1}^\infty \ind_{\Pgamma_{i}=\Pgeo}.
$
Combining this with \eqref{e.approx} shows that the second limsup in \eqref{e.l.approx.metric.} is at most $\int_{[s,s']}\d t\, \sum_{i=1}^{\infty} \dens_{i} \ind_{\Pgamma_{i}=\Pgeo}$.
Hence $\limsup_{n} \Metric_{\Mmu_{n}}(s,y;s',y') \leq |\Pgeo|_{\Metric_{\Mmu}}$.
Being the pointwise limit of $\Pgeo_{n}$, the path $\Pgeo$ connects $(s,y)$ to $(s',y')$, so $|\Pgeo|_{\Metric_{\Mmu}}\leq\Metric_{\Mmu}(s,y;s',y')$ and the desired result follows.
\end{proof}

\section{Proof of Theorem~\ref{t.main}, Corollary~\ref{c.main}, and Proposition~\ref{p.multi}}
\label{s.pfmain}

\subsection{Basic properties related to $\Engbm$}
\label{s.pfmain.Engbm}
To streamline the proof, let us prepare some notation and properties related to $\Engbm$.
Recall that $\intvl\subset\R$ is a finite union of closed intervals, $\Ff:\intvl\to\R$ is Borel with $\Ff\geq\Fhh(1,\cdot)$,
\begin{align}
	\label{e.Entspp}
	\Entspp{c}
	:=
	\big\{ \Fphi\in\Csp(\R) 
		: 
		\Fphi|_{[-c,c]}\in\Hsp^1[-c,c], \
		\Fphi\geq\Fhh(1,\cdot),\
		\Fphi|_{\R\setminus[-c,c]} = \Fhh(1,\cdot)|_{\R\setminus[-c,c]}
	\big\},
\end{align}
and $\Entsp := \cup_{c\in(0,\infty)} \Entspp{c}$.
For $\Fphi\in\Entsp$, set
\begin{align}
	\label{e.Entbm}
	\Entbm(\Fphi)
	:=
	\frac{1}{4} \int_{\R} \d x\, \big( (\partial_x\Fphi)^2 - 4x^2 \big).
\end{align}
This way, $\Engbm(\intvl,\Ff) = \inf\{ \Entbm(\Fphi) : \Fphi\in\Entsp, \ \Fphi|_{\intvl} = \Ff \}$.
When $\Engbm(\intvl,\Ff)<\infty$, the infimum has a unique minimizer.
To see why, take any $\phi_1\neq\phi_2\in \Entspp{c}$ with $\Entbm(\phi_1)=\Entbm(\phi_2)$, consider the average $\phi:=(\phi_1+\phi_2)/2\in \Entspp{c}$, and compute
\begin{align}
	\tfrac{1}{2}\Entbm(\phi_1) + \tfrac{1}{2}\Entbm(\phi_2) -  \Entbm\big(\phi\big)
	=
	\int_{-c}^{c} \d x \, \tfrac{1}{2} \big( \partial_x \phi_1 - \partial_x \phi_2  \big)^2\ .
\end{align}
The integral is strictly positive, otherwise $\partial_x\phi_1=\partial_x\phi_2$ a.e., which forces $\phi_1=\phi_2$.
Hence $\Entbm(\phi_1)=\Entbm(\phi_2)>\Entbm(\phi)$, so the minimizer has to be unique.
The minimizer, denoted by $\Ff_*$, can be obtained geometrically, as illustrated in Figure~\ref{f.Ff*}.
Roughly speaking, $\Ff_*$ is obtained by interpolating $\Ff$ outside of $\intvl$ constrained to stay above the parabola $\Fhh(1,x)=-x^2$.
We call this interpolation \textbf{parabola-constrained interpolation} and describe it next.
Let $\hypo(\Fphi):=\{(x,\alpha):\alpha\leq\Fphi(x)\}$ denote the hypograph of $\Fphi$; decompose $\R\setminus\intvl$ into disjoint open intervals;
on any one of those intervals that is bounded, denoted $(\alpha',\beta')$, consider the convex hull of $\hypo(\Fhh(1,\cdot)|_{[\alpha',\beta']})\cup\{(\alpha',\Ff(\alpha')),(\beta',\Ff(\beta'))\}$ and make $\Ff_*|_{(\alpha',\beta')}$ the function whose hypograph is this convex hull; on the leftmost interval, denoted $(-\infty,\alpha')$, let $\Ff_*|_{(-\infty,\alpha')}$ be the function whose hypograph is the convex hull of $\hypo(\Fhh(1,\cdot)|_{(-\infty,\alpha']})\cup\{(\alpha',\Ff(\alpha'))\}$; do similarly on the rightmost interval.

Next we turn to the finite-dimensional analog of $\Engbm$.
Take any $\vec{y}=(y_1<\ldots<y_k)\in\R^{k}$ and $\vec{\alpha}\in\R^{k}$ with $\alpha_i\geq -y_i^2$ for all $i$, and consider
\begin{align}
	\label{e.Engbmn}
	\Engbmn{k}(\vec{y},\vec{\alpha})
	&:=
	\inf\{\Entbm(\Fphi) : \Fphi\in\Entsp, \Fphi(y_i)=\alpha_i,i=1,\ldots,k\}.
\end{align}
Just like before, the infimum has a unique minimizer, obtained by parabola-constrained interpolating $(y_1,\alpha_1)$, \ldots, $(y_k,\alpha_k)$.
By analyzing how the minimizer varies in $(\vec{y},\vec{\alpha})$, it is not hard to check that
\begin{align}
	\label{e.Engbmn.conti}
	\Engbmn{k}	(\vec{y},\vec{\alpha})
	\text{ is continuous on }
	\big\{ (y_1<\ldots<y_k,\alpha_1,\ldots,\alpha_k)\in\R^{2k} : \alpha_i\geq -y_i^2 \ \forall i \big\}.
\end{align}
To streamline the notation, for any finite set $\SY=\{y_1<\ldots<y_{|\SY|}\}\subset\R$ and $\Fphi\in\Entsp$, write $\vec{\SY}:=(y_1,\ldots,y_{|\SY|})$ and $\Fphi(\vec{\SY}):=(\Fphi(y_1),\ldots,\Fphi(y_{|\SY|}))$, so that $\SY$ is the set and $\vec{\SY}$ is the ordered tuple of its elements, fix an $\FF$ and hence $\Ff_*$, and let $\Ff_{*\SY}$ be the unique minimizer of \eqref{e.Engbmn} for $(\vec{y},\vec{\alpha})=(\vec{\SY},\Ff_*(\vec{\SY}))$.
More explicitly,
\begin{align}
	\label{e.Ff*finite}
	\Engbmn{|\SY|}\big(\vec{\SY},\Ff_*(\vec{\SY})\big)
	&:=
	\inf\{\Entbm(\Fphi) : \Fphi\in\Entsp, \Fphi=\Ff_* \text{ on } \SY\}
	=
	\Entbm(\Ff_{*\SY}).
\end{align}
Recall $\Lambda$ is the union of finitely many closed intervals, and write $\Lambda_s$ for the union of those singleton intervals, if any.
Now take $\SY$ to be $\intvl_{m}:=\intvl\cap(\{i/m\}_{i\in\Z}\cup\Lambda_s)$ and consider 
\begin{align}
	\label{e.Ff*m}
	\Engbmn{|\intvl_{m}|}\big(\vec{\intvl}_{m},\Ff(\vec{\intvl}_{m})\big)
	&:=
	\inf\{\Entbm(\Fphi) : \Fphi\in\Entsp, \Fphi(y)=\Ff(y), \ \forall y\in\intvl_{m}\}
	=
	\Entbm(\Ff_{*\intvl_{m}}).
\end{align}
The minimizer $\Ff_{*\intvl_{m}}$ is obtained by parabola-constrained interpolating $(y,\Ff(y))_{y\in\intvl_{m}}$.
From this, it is not hard to check that there exists a $c<\infty$ such that
\begin{align}
	\label{e.Ff*m.infty}
	&\text{ if } \Ff\notin\Csp(\intvl), \text{ then } \Entbm(\Ff_{*\intvl_{m}})\to\infty \text{ as } m\to\infty,
	\\
	\begin{split}
		\label{e.Ff*m.convg}
		&\text{ if } \Ff\in\Csp(\intvl),  \text{ then }
		\Ff_{*\intvl_{m}}\to\Ff_*\text{ uniformly on } [-c,c] \text{ as } m\to\infty,
		\\
		&\hphantom{\text{ if } \Ff\in\Csp(\intvl),} 
		\text{ and } \Ff_{*\intvl_{m}}=\Ff_*=\Fhh \text{ outside of } [-c,c] \text{ for all } m.
	\end{split}
\end{align}
Also, consider
\begin{align}
	\label{e.Engbm.G}
	\FG_{m}(\lambda)
	:=
	\inf\big\{ \Entbm(\Fphi) : \Fphi\in\Entsp, \ \Fphi = \Ff_* \text{ on } \intvl_{m}, \ \norm{\Fphi-\Ff_*}_{\Lsp^\infty(\R)}\geq\lambda \big\}.
\end{align}
If $\Engbm(\intvl,\Ff)<\infty$, we have
\begin{align}
	\label{e.Engbm.bd}
	\liminf_{m\to\infty} \FG_{m}(\eps) \gneqq \Entbm(\Ff_*)=\Engbm(\intvl,\Ff),
	\quad
	\text{for all } \eps>0. 
\end{align}
This can be shown by taking a sequence from the defining set of $\FG_m(\eps)$ whose energies converge to the liminf in \eqref{e.Engbm.bd} and extracting a subsequential limit $\Ff_{0}$ by compactness.
The limit $\Ff_{0}$ belongs to $\Entsp$ but differs from $\Ff_*$ by construction, so $\Entbm(\Ff_0)>\Entbm(\Ff_*)$ by the uniqueness of the minimizer in \eqref{e.Engbm}, which was shown after \eqref{e.Entbm}.

\subsection{Proof of Theorem~\ref{t.main}\eqref{t.main.inf}, proving $\leq$ in \eqref{e.t.main}}
\label{s.pfmain.<}
Recall that $\intvl_{m}:=\intvl\cap(\{i/m\}_{i\in\Z}\cup\intvl_s)$ and that $\Ff_{*\intvl_{m}}$ is the minimizer of the infimum in \eqref{e.Ff*m}. 
Let $\Fh_{*m}(t,x):=\HLbk_{1\to t}[\Ff_{*\intvl_{m}}](x)$ and $\Mmu_{*m}:=\M[\Fh_{*m}]$.
Since $\Ff_{*\intvl_{m}}$ is obtained by parabola-constrained interpolating $(y,\Ff(y))_{y\in\intvl_{m}}$, the function is continuous, piecewise linear or quadratic, and equal to $\Fhh(1,\cdot)$ outside of a bounded interval.
Given these properties, Lemma~\ref{l.hopflax}\eqref{l.hopflax.bk} asserts that $\Fh_{*m}$ is a tractable solution without entropy shocks, so $\Mmu_{*m}=\Mnon[\Fh_{*m}]$ and $\Ment[\Fh_{*m}]=0$.
Also, Proposition~\ref{p.reproducing} gives that $\Hfn[\Mmu_{*m}]=\Fh_{*m}$.
Now apply Proposition~\ref{p.key} to get
\begin{align}
	\label{e.pfmain<.}
	\rate(\Metric_{\Mmu_{*m}}) - 0 
	= \Entbm(\Fh_{*m}(1,\cdot)) 
	= \Entbm(\Ff_{*\intvl_{m}}) 
	= \Engbmn{|\intvl_{m}|}\big(\vec{\intvl}_m,\Ff(\vec{\intvl}_m)\big) 
	\leq
	\Engbm(\intvl,\Ff),
\end{align}
where the last inequality follows because $\intvl_{m}\subset\intvl$.
If $\Ff\notin\Csp(\intvl)$, using \eqref{e.Ff*m.infty} gives the desired result.
If $\Ff\in\Csp(\intvl)$, using \eqref{e.Ff*m.convg} gives $\Metric_{\Mmu_{*m}}(0,0;1,\cdot)\to \Ff_*$ uniformly on $\R$.
This together with the property that $\rate$ is a good rate function gives
\begin{align}
	\label{e.pfmain<..}
	\inf\big\{ \rate(\Metric_{\Mmu}) : \Mmu\in\Msp_+,\ \Metric_{\Mmu}(0,0;1,x) = \Ff(x) \text{ for } x\in\intvl \big\}
	\leq
	\liminf_{m\to\infty} \rate(\Metric_{\Mmu_{*m}}).
\end{align}
Combining \eqref{e.pfmain<.}--\eqref{e.pfmain<..} gives the desired result.

\subsection{Proof of Theorem~\ref{t.main}\eqref{t.main.inf}, proving $\geq$ in \eqref{e.t.main}}
\label{s.pfmain.>}
We begin with a reduction.
For any $f$ under the current consideration (Borel, $\geq -x^2$), by linearly interpolating $f$ on $\Lambda_m$, we can construct a sequence $f_m$ with $\Engbm(\intvl,\Ff_m)<\infty$ and $\Engbm(\intvl,\Ff_m)\to\Engbm(\intvl,\Ff)$, whether the latter is finite or not.
Given this, hereafter we assume $\Engbm(\intvl,\Ff)<\infty$, which in particular implies $\Ff\in\Csp(\intvl)$.
Fix an arbitrary minimizer $\Mmu=\sum_{\Pgamma\in\Snet}\dens_{\Pgamma}\delta_{\Pgamma}$ of \eqref{e.t.main} and put $\Fh:=\Hfn[\Mmu]$.
Minimizers exist because $\rate$ is a good rate function.

Let us first construct some approximations.
This construction will be used here and in Sections~\ref{s.pfmain.=h*}.
Fix an arbitrary $\eps>0$.
First, note that for any $c<\infty$, the mapping $\Fphi\mapsto\Entbm(\Fphi)$ is lower semicontinuous on $\Entspp{c}$ under the uniform topology over $[-c,c]$; see \eqref{e.Entspp} and \eqref{e.Entbm}.
Combining this with \eqref{e.Ff*m.convg} gives $ \liminf_{m} \Entbm(\Ff_{*\intvl_{m}}) \geq \Entbm(\Ff_*)=\Engbm(\intvl,\Ff)$.
Given this and \eqref{e.Engbm.bd}, fix a large enough $m$ such that
\begin{align}
	\label{e.pfmain.>.m}
	\Entbm(\Ff_{*\intvl_{m}}) \geq \Engbm(\intvl,\Ff) - \eps,
	\qquad
	\FG_m(\eps) > \Engbm(\intvl,\Ff),
	\ \
	\FG_m \text{ defined in } \eqref{e.Engbm.G}.
\end{align}
For this fixed $m$, take an arbitrary $x_0\in\R$ and add it to $\intvl_{m}$ to get $\SY:=\intvl_{m}\cup\{x_0\}$; see Remark~\ref{r.y0}.
For each $y\in\SY$, take an $\Metric_{\Mmu}$ geodesic $(0,0)\xrightarrow{\Pgeo_y}(1,y)$, and expand the set $\Snet$ to include these $\Pgeo_{y}$s: More precisely, take a $\Snet'\supset\Snet$ such that $\graph(\Snet')\supset\graph(\Pgeo_y)$, for all $y\in\SY$.
Set $\dens_{\Pzeta}:=0$ for all $\Pzeta\in\Snet'\setminus\Snet$ so that the resulting measure does not change, namely $\Mmu=\sum_{\Pzeta\in\Snet'}\dens_{\Pzeta}\delta_{\Pzeta}$.
Now construct the piecewise linear approximation $\Mmu_{n}$ of $(\Mmu,\Snet')$ as in Definition~\ref{d.approx}.
Set $\Fh_{n}:=\Hfn[\Mmu_{n}]$.

\begin{rmk}\label{r.y0}
Adding $x_0$ is needed only in Section~\ref{s.pfmain.=h*} and is not in this section.
\end{rmk}

We now use the preceding approximations to prove the desired result.
First, we argue that
\begin{align}
	\label{e.pfmain>.Fhn=Fh}
	\lim_{n\to\infty} \Fh_{n}(1,y) = \Fh(1,y),\qquad \text{for all }y\in\SY:=\intvl_{m}\cup\{x_0\}.
\end{align}
To see why, recall that $\graph(\Pgeo_{y})\subset\graph(\Snet')$ and that $\Mmu_{n}$ is constructed such that \eqref{e.approx} hold for $\{\Pgamma_1,\Pgamma_2,\ldots\}=\Snet'$.
Hence, there exists a path $(0,0)\xrightarrow{\Pgeo_{y,n}}(1,y)$ such that $\lim_{n}|\Pgeo_{y,n}|_{\Metric_{\Mmu_{n}}}=|\Pgeo_{y}|_{\Metric_{\Mmu}}$.
The left-hand side is bounded above by $\liminf_{n}\Fh_{n}(1,y)$, while the right-hand side is equal to $\Fh(1,y)$.
This proves the inequality $\geq$ in \eqref{e.pfmain>.Fhn=Fh}.
The other inequality follows by Lemma~\ref{l.approx.metric} for $(s,y;s',y')=(0,0;1,y)$.
Next, using Proposition~\ref{p.pwlinear}\eqref{p.pwlinear.rate} and Proposition~\ref{p.key} for $\Mmu\mapsto\Mmu_{n}$ gives $\rate(\Metric_{\Mmu}) -\rate(\Ment[\Fh_n]) \geq \Entbm(\Fh_{n}(1,\cdot))$.	
Bound the last term from below by $\Engbmn{|\SY|}(\vec{\SY},\Fh_{n}(\vec{\SY}))$, and take the $n\to\infty$ limit with the aid of \eqref{e.Engbmn.conti} and \eqref{e.pfmain>.Fhn=Fh}.
Doing so gives
\begin{align}
	\label{e.pfmain>.key.}
	\rate(\Metric_{\Mmu}) -\limsup_{n\to\infty} \rate(\Ment[\Fh_n])
	\geq
	\liminf_{n\to\infty} \Entbm(\Fh_{n}(1,\cdot))
	\geq
	\Engbmn{|\SY|}(\vec{\SY},\Fh(\vec{\SY})).
\end{align}
Recall that $\Fh(1,\cdot)|_{\intvl}=\Ff$, so in particular $\Fh(1,\cdot)=\Ff$ on $\intvl_{m}$.
This together with $\SY\supset\intvl_{m}$ gives $\Engbmn{|\SY|}(\vec{\SY},\Fh(\vec{\SY}))\geq\Engbmn{|\intvl_{m}|}(\vec{\intvl_{m}},\Ff_*(\vec{\intvl_{m}}))=\Entbm(\Ff_{*\intvl_{m}})$.
Combining this with \eqref{e.pfmain>.key.} and the first inequality in \eqref{e.pfmain.>.m} gives
\begin{align}
	\label{e.pfmain>.key..}
	\rate(\Metric_{\Mmu}) -\limsup_{n\to\infty} \rate(\Ment[\Fh_n]) 
	\geq 
	\Engbm(\intvl,\Ff) - \eps.
\end{align}
In particular, $\rate(\Metric_{\Mmu}) \geq \Engbm(\intvl,\Ff) - \eps$.
Since $\eps>0$ was arbitrary, $\rate(\Metric_{\Mmu})\geq\Engbm(\intvl,\Ff)$ follows.

\subsection{Proof of Theorem~\ref{t.main}\eqref{t.main.inf}, proving $\Metric_{\Mmu}(0,0;t,x)=\Fh_*(t,x)$}
\label{s.pfmain.=h*}
Continue with the notation in Section~\ref{s.pfmain.>}.
Our goal here is to prove that $\Fh(t,x)=\Fh_*(t,x)$, for all $(t,x)\in(0,1]\times\R$.
Let us first prove this for $t=1$, namely
\begin{align}
	\label{e.pfmain.=h*.=}
	\Fh_*(1,x) = \Ff_*(x),
	\qquad
	\text{for all } x\in\R.
\end{align}
Since $\Mmu$ is a minimizer, the results of Sections~\ref{s.pfmain.<}--\ref{s.pfmain.>} give $\Engbm(\intvl,\Ff)=\rate(\Metric_{\Mmu})$.
Inserting this into \eqref{e.pfmain>.key.} and forgoing the first limsup there give
\begin{align}
	\label{e.pfmain>.key...}
	\Engbm(\intvl,\Ff)
	\geq
	\liminf_{n\to\infty} \Entbm(\Fh_{n}(1,\cdot))
	\geq
	\Engbmn{|\SY|}(\vec{\SY},\Fh(\vec{\SY})).
\end{align}
Further, since $\Fh(1,y)=\Ff(y)$ for all $y\in\intvl_{m}$ and since $\SY:=\intvl_{m}\cup\{x_0\}$, we have $\Engbmn{|\SY|}(\vec{\SY},\Fh(\vec{\SY}))\geq \FG_m(|\Fh(1,x_0)-\Ff_*(x_0)|)$.
Combining this with \eqref{e.pfmain>.key...} gives $ \Engbm(\intvl,\Ff) \geq \FG_m(|\Fh(1,x_0)-\Ff_*(x_0)|) $.
At the same time, the second inequality in \eqref{e.pfmain.>.m} gives $\FG_m(\eps)>\Engbm(\intvl,\Ff)$.
Since $\FG_m(\lambda)$ is nondecreasing in $\lambda$ (readily seen from \eqref{e.Engbm.G}), we must have $|\Fh(1,x_0)-\Ff_*(x_0)|<\eps$, otherwise the last two inequalities can be combined and a contradiction would follow.
Since $x_0\in\R$ and $\eps>0$ were arbitrary, \eqref{e.pfmain.=h*.=} follows.

Let us prepare a bound that will be needed.
Recall $\Fh_{n}:=\Hfn[\Mmu_n]$ from Section~\ref{s.pfmain.>}.
We argue that
\begin{align}
	\label{e.pfmain.=h*.bd}
	\limsup_{n\to\infty} \norm{ \Fh_{n}(1,\cdot) - \Ff_* }_{\Lsp^\infty(\R)} < \eps.
\end{align}
After passing to a subsequence, either $\Fh_{n}(1,x_n)\to\Ff_*(x_0)+\alpha$ for some $\alpha=\pm\limsup_{n} \norm{ \Fh_{n}(1,\cdot) - \Ff_* }_{\Lsp^\infty(\R)}$ and $x_1,x_2,\ldots\to x_0\in\R$, or $\Fh_{n}(1,x_n)\geq \Fhh(1,x_n)+\eps/2$ for some $x_1,x_2,\ldots\to\pm\infty$.
If the first case happens with $|\alpha|=\infty$ or the second case happens, it is not hard to show that $\Entbm(\Fh_{n}(1,\cdot))\to\infty$.
This is prohibited by \eqref{e.pfmain>.key...}.
We hence consider the first case with $|\alpha|<\infty$.
For the $m$ fixed in Section~\ref{s.pfmain.>}, set $\SY'_n:=\intvl_{m}\cup\{x_n\}$, write $\Entbm(\Fh_{n}(1,\cdot)) \geq \Engbmn{|\SY'_n|}(\vec{\SY}'_n,\Fh_{n}(1,\vec{\SY}'_n))$ (see \eqref{e.Engbmn}), and take the $n\to\infty$ limit.
When taking the limit, use \eqref{e.Engbmn.conti}, use \eqref{e.pfmain>.Fhn=Fh} for $y\in\intvl_{m}$, and use $\Fh_{n}(1,x_n)\to\Ff_*(x_0)+\alpha$.
Doing so gives $ \liminf_{n}\Entbm(\Fh_{n}(1,\cdot))\geq\FG_{m}(|\alpha|)$.
Combining this with the first inequality in \eqref{e.pfmain>.key...} gives $\Engbm(\intvl,\Ff)\geq\FG_{m}(|\alpha|)$.
From here, the same argument as the last paragraph gives $|\alpha|<\eps$, so \eqref{e.pfmain.=h*.bd} follows.

Let us prove $\Fh\leq\Fh_*$ on $(0,1]\times\R$.
Take any $(t,x)\in(0,1]\times\R$, and consider any path $(t,x)\xrightarrow{\Peta}\{1\}\times\R$.
The metric $\Metric_{\Mmu}$ satisfies $\Metric_{\Mmu}(0,0;1,\Peta(1))\geq\Metric_{\Mmu}(0,0;t,x)+|\Peta|_{\Metric_{\Mmu}}$ and $|\Peta|_{\Metric_{\Mmu}}\geq|\Peta|_{\diri}.$
Recognize $\Metric_{\Mmu}(0,0;t,x)$ as $\Fh(t,x)$ and recognize $\Metric_{\Mmu}(0,0;1,\Peta(1))$ as $\Fh(1,\Peta(1))$, which is equal to $\Ff_*(\Peta(1))$ by \eqref{e.pfmain.=h*.=}.
We have $\Fh(t,x)\leq -|\Peta|_{\diri}+\Ff_*(\Peta(1))$.
Taking the infimum over such $\Peta$s gives
\begin{align}
	\Fh(t,x)
	\leq
	\HLbk_{1\to t}[\Ff_*](x)
	=
	\Fh_*(t,x).
\end{align}

Let us prove $\Fh\geq\Fh_*$ on $(0,1]\times\R$.
Recall $\Mmu_{n}$ from Section~\ref{s.pfmain.>}.
Take any $(t,x)\in(0,1]\times\R$, and use Proposition~\ref{p.pwlinear}\eqref{p.pwlinear.path} for $\Mmu\mapsto\Mmu_{n}$ to obtain a path $(t,x)\xrightarrow{\Pgeo_n}\{1\}\times\R$ such that $|\Pgeo_{n}|_{\Ment[\Fh_n]}=\Fh_n(1,\Pgeo_n(1))-\Fh_n(t,x)$.
Since $\Mmu$ is a minimizer, the results of Sections~\ref{s.pfmain.<}--\ref{s.pfmain.>} give $\Engbm(\intvl,\Ff)=\rate(\Metric_{\Mmu})$.
Combining this with \eqref{e.pfmain>.key..} gives $\limsup_{n} \rate(\Ment[\Fh_{n}])\leq\eps$.
Given this property, apply Lemma~\ref{l.approx.path} to $\{\Pgeo_n\}_{n}$ and $\{\Ment[\Fh_{n}]\}_{n}$.
The result gives a path $(t,x)\xrightarrow{\Pgeo}\{1\}\times\R$ such that
\begin{align}
\begin{split}
	(1-t)^{1/3}\Big(\frac{3\eps}{4}\Big)^{2/3} + |\Pgeo|_{\diri} 
	&\geq
	\limsup_{n\to\infty} \big( \Fh_n(1,\Pgeo_n(1))-\Fh_n(t,x) \big)
\\
	&\geq
	\liminf_{n\to\infty} \Fh_n(1,\Pgeo_n(1))
	-
	\limsup_{n\to\infty} \Fh_n(t,x).
\end{split}
\end{align}
By \eqref{e.pfmain.=h*.bd} and the fact that $\Ff_*$ is continuous, the liminf is at least $\Ff_*(\Pgeo(1))-\eps$.
By Lemma~\ref{l.approx.metric} for $(s,y;s',y')=(0,0;t,x)$, the last limsup is at most $\Fh(t,x)$.
After rearranging terms, we arrive at 
\begin{align}
\begin{split}
	\Fh(t,x)
	&\geq
	-|\Pgeo|_{\diri} + \Ff_*(\Pgeo(1))
	-
	\eps - (1-t)^{1/3}\Big(\frac{3\eps}{4}\Big)^{2/3}
\\
	&\geq
	\HLbk_{1\to t}[\Ff_*](x)
	-
	\eps - (1-t)^{1/3}\Big(\frac{3\eps}{4}\Big)^{2/3}
	=
	\Fh_*(t,x) - \eps - (1-t)^{1/3}\Big(\frac{3\eps}{4}\Big)^{2/3}.
\end{split}
\end{align}
Since $\eps>0$ was arbitrary, this completes the proof.

\subsection{Proof of Theorem~\ref{t.main}\eqref{t.main.measure}}
\label{s.pfmain.measure}
We begin with a lemma.
\begin{lem}\label{l.measure.support}
	Take any $\Mmu\in\Msp_+$, set $\Hfn[\Mmu]=\Fh$ and $\Fu:=\partial_{x}\Fh$, assume $\Fu$ is a tractable solution with respect to the tractable partition $(\Stime(\Fu),\Sbdy(\Fu),\Sdom(\Fu))$ of a light cone $\cone(c)$, and consider the $\eps$-thickened set of $\graph(\Sbdy(\Fu))$, namely 
	\begin{align}
		\SB_{\eps}(\Fu) := \{ (t,x): |x-\Pl(t)| \leq \eps, t\in[\start(\Pl),\ed(\Pl)] \text{ for some } \Pl\in\Sbdy(\Fu) \}.
	\end{align}
	For every $\eps>0$, $\Fh=\Hfn[\,\Mmu\lfloor_{\SB_{\eps}(\Fu)}\,]$.
\end{lem}
\begin{proof}
First, we claim that, for every $(t,x)\in(0,1]\times\R$, there exists a path $\Sbdy(\Fu)\cup\{(0,0)\}\xrightarrow{\Peta}(t,x)$ such that $|\Peta|_{\diri}=\Fh(t,x)-\Fh(\start,\Peta(\start))$, where $\start=\start(\Peta)$, and the path $\Peta$ is obtained by concatenating the characteristics of $\Fu$.
If $(t,x)\in\graph(\Sbdy(\Fu))$, simply take $\Peta$ as the degenerate path with $[\start(\Peta),\ed(\Peta)]=[t,t]$ and $\Peta(t)=x$.
Otherwise, let us construct $\Peta$ by going backward in time.
Starting from $(t,x)\notin\graph(\Sbdy(\Fu))$, we go back in time by following a characteristic $\Pchi$ of $\Fu$.
\begin{itemize}
	\item If $(t,x)\in\SD$ for some $\SD\in\Sdom(\Fu)$, or if $(t,x)\notin\cone(c)$, such a characteristic $\Pchi$ indeed exists.
	\item Otherwise $(t,x)$ lies along the upper boundary (picturing time as the vertical axis) of some $\SD\in\Sdom(\Fu)$.
	Take a sequence of points $(t_k,x_k)\in\SD$ that approaches $(t,x)$.
	Since $\Fu\in\Lsp^\infty(\SD)$, after passing to a subsequence, $\Fu(t_k,x_k)$ converges to a finite limit.
	Each $(t_k,x_k)$ lies along a characteristic, which has velocity $-\frac{1}{2}\Fu(t_k,x_k)$ by \eqref{e.characteristics}.
	Taking the limit of these characteristics gives a characteristic $\Pchi$ that is connected to $(t,x)$.
\end{itemize}
Follow $\Pchi$ backward in time until it reaches $(0,0)$ or $\partial\SD$, and let $s\in[0,t)$ denote the time when this happens.
Using $\partial_t\Fh=\frac{1}{4}\Fv^2$ outside of $\graph(\Sbdy(\Fu))$ and using \eqref{e.characteristics}, we calculate $\Fh(t,x)-\Fh(s,\Pchi(s))=|\Pchi|_{[s,t]}|_{\diri}$. 
Now, if $(s,\Pchi(s))\notin\graph(\Sbdy(\Fu))\cup\{(0,0)\}$, repeat this process until reaching $\graph(\Sbdy(\Fu))\cup\{(0,0)\}$ and concatenate the resulting characteristics.
Doing so gives $\Peta$.
	
We now turn to proving the statement of this lemma.
Put $\Mnu:=\Mmu\lfloor_{\SB_{\eps}(\Fu)}$.
Since $\Mmu\geq\Mnu$, by \eqref{e.monotone.height} we have $\Fh\geq\Hfn[\Mnu]$.
Hence it suffices to prove $\Hfn[\Mnu]\geq\Fh$.
Taking any $(t_0,x_0)\in(0,1]\times\R$, we prove this by constructing a path $(0,0)\xrightarrow{\Pgeo}(t_0,x_0)$ such that $|\Pgeo|_{\Metric_\Mnu}\geq\Fh(t_0,x_0)$.
First, use the preceding claim to find $\graph(\Sbdy(\Fu))\cup\{(0,0)\}\xrightarrow{\Peta_1}(t_0,x_0)$ such that $|\Peta_1|_{\diri}=\Fh(t_0,x_0)-\Fh(\start_1,\Peta_1(\start_1))$, where $\start_1:=\start(\Peta_1)$.
Next, take an $\Metric_{\Mmu}$ geodesic $(0,0)\xrightarrow{\Pgeo_1}(\start_1,\Peta_1(\start_1))$, let $t_1:=\sup\{ t\in[0,\start_1] : (t,\Pgeo_1(t))\notin\SB_{\eps}(\Fu) \}$, and put $x_1:=\Pgeo_1(t_1)$.
Since $\Pgeo_1$ is an $\Metric_{\Mmu}$ geodesic and since $\Fh=\Hfn[\Mmu]$, we have $|\Pgeo_1|_{[t_1,\start_1]}|_{\Metric_{\Mmu}}=\Fh(\start_1,\Peta_1(\start_1))-\Fh(t_1,x_1)$.
Further, since $\graph(\Pgeo_1|_{[t_1,\start_1]})\subset\SB_{\eps}(\Fu)$, we have $|\Pgeo_1|_{[t_1,\start_1]}|_{\Metric_{\Mmu}}=|\Pgeo_1|_{[t_1,\start_1]}|_{\Metric_{\Mnu}}$.
Now concatenate the paths $\Pgeo_1|_{[t_1,\start_1]}$ and $\Peta_1|_{[\start_1,t_0]}$.
The preceding properties give
\begin{align}
	\big|\, \Pgeo_1|_{[t_1,\start_1]} \cup \Peta_1|_{[\start_1,t_0]}\,\big|_{\Metric_{\Mnu}}
	&\geq
	\big|\Pgeo_1|_{[t_1,\start_1]}\big|_{\Metric_{\Mnu}} + \big|\Peta_1|_{[\start_1,t_0]}\big|_{\diri}
	=
	\Fh(t_0,x_0) - \Fh(t_1,x_1).
\end{align}
Taking $(t_1,x_1)$ as the new starting point and iterating this process gives the desired $\Pgeo$.
\end{proof}

Let us prepare some notation.
Set $\Fu_*:=\partial_{x}\Fh_*$ and let $\Mmu$ be a minimizer as before.
The assumption here (in Theorem~\ref{t.main}\eqref{t.main.measure}) is that $\Ff$ is continuous, and piecewise linear or quadratic on $\intvl$, and our goal is to prove $\Mmu=\M[\Fh_*]$.
The function $\Ff_*\in\Csp(\R)$, obtained from parabola-constrained interpolating the graph of $\Ff$, is piecewise linear or quadratic and equal to $\Fhh$ outside of a bounded interval.
By Lemma~\ref{l.hopflax}\eqref{l.hopflax.bk}, $\Fu_*$ is a tractable solution without entropy shocks.

Next, we perform a few reductions.
The result of Section~\ref{s.pfmain.=h*} gives $\Fh:=\Hfn[\Mmu]=\Fh_*$, so Lemma~\ref{l.measure.support} implies that $\Hfn[\,\Mmu\lfloor_{\SB_{\eps}(\Fu_*)}\,]=\Fh_*$.
This property forces $\Mmu=\Mmu\lfloor_{\SB_{\eps}(\Fu_*)}$, otherwise $\Mmu>\Mmu\lfloor_{\SB_{\eps}(\Fu_*)}$, which implies $\rate(\Metric_{\Mmu})>\rate(\Metric_{\Mmu\lfloor_{\SB_{\eps}(\Fu_*)}})$ by \eqref{e.monotone} and contradicts the condition that $\Mmu$ is a minimizer.
Since $\eps>0$ is arbitrary, it follows that $\Mmu$ is supported on $\graph(\Sbdy(\Fu_*))$.
Next, take any $\Pl\in\Sbdy(\Fu_*)$ and any open interval $\SO\subset(\start(\Pl),\ed(\Pl))$ such that $ \Fu_*|_{\Pl-}>\Fu_*|_{\Pl+}$ on $\SO$.
It suffices to prove that
\begin{align}
	\label{e.p.comparing.measure.goal}
	\text{for any } t\in\SO \text{ and any } \Metric_{\Mmu} \text{ geodesic } (0,0)\xrightarrow{\Pgeo}(t,\Pl(t)),
	\qquad
	\Pgeo|_{[0,t]\cap\SO} = \Pl|_{[0,t]\cap\SO}.
\end{align}
Once \eqref{e.p.comparing.measure.goal} is proven, since $\Fh_*=\Metric_{\Mmu}(0,0;\cdot,\cdot)$, it follows that $\Fh_*(t,\Pl(t)) = | \Pl|_{[s,t]} |_{\Metric_{\Mmu}} + \Fh_*(s,\Pl(s))$, for all $s<t\in \SO$.
This implies $\Mmu\lfloor_{\graph(\Pl|_\SO)}=\M[\Fh_*]\lfloor_{\graph(\Pl|_\SO)}$, and this holds for all $\SO$ described above.
Since $\Mmu$ is supported on $\graph(\Sbdy(\Fu_*))$ and since $\Fu_*$ has no entropy shocks, it follows that $\Mmu=\M[\Fh_*]$.

To prove \eqref{e.p.comparing.measure.goal}, we argue by contradiction.
Assume the opposite of \eqref{e.p.comparing.measure.goal}: There exist $t_0\in\SO$, an $\Metric_{\Mmu}$ geodesic $(0,0)\xrightarrow{\Pgeo}(t_0,\Pl(t_0))$, a region $\SD\in\Sdom(\Fu_*)$, and an $s_3\in\SO\cap(0,t_0]$ such that $\Pgeo|_{(s_3-\lambda,s_3)}\subset\SD$, for some $\lambda>0$, and that $\Pgeo(s_3)=\Pl(s_3)$; see Figure~\ref{f.contradiction1}.
Combining the property that $ \Fu_*|_{\Pl-}>\Fu_*|_{\Pl+}$ on $\SO$ with \eqref{e.rankine.hugoniot} for $\Fv\mapsto\Fu_*$ gives
\begin{align} 
	\label{e.p.comparing.measure.contradict}
	-\tfrac{1}{2}\Fu_*|_{\Pl-} < \dot{\Pl} < -\tfrac{1}{2}\Fu_*|_{\Pl+} \text{ on } \SO.
\end{align}
Since $\Fu_*$ is a $\Csp^1(\SD)$ solution of Burgers' equation, every point in $\SD$ lies along a characteristic of $\Fu_*$.
This property together with \eqref{e.p.comparing.measure.contradict} implies that all points on $\graph(\Pgeo|_{[0,s_3)})\cap\SD$ that are sufficiently close to $(s_3,\Pgeo(s_3))$ are connected back to $\Pl$ by characteristics; see Figure~\ref{f.contradiction1}.
Take one such point $(s_2,\Pgeo(s_2))$ and the corresponding characteristic $(s_1,\Pgeo(s_1))\xrightarrow{\Pchi}(s_2,\Pgeo(s_2))$.
Assume $s_2$ is sufficiently close to $s_3$ such that $\graph(\Pchi|_{(s_1,s_2]})\subset\SD$ and that $s_1\in\SO$.

Given the preceding setup, we now argue for a contradiction.
First, since $\Pgeo$ is an $\Metric_{\Mmu}$ geodesic and since $\Fh_*=\Metric_{\Mmu}(0,0;\cdot,\cdot)$,
\begin{align}
	\label{e.p.comparing.measure.contradict.1}
	\Fh_*(s_3,\Pgeo(s_3)) 
	= 
	\Fh_*(s_2,\Pgeo(s_2)) + \big|\Pgeo|_{[s_2,s_3]}\big|_{\Metric_{\Mmu}}
	=
	\Fh_*(s_2,\Pgeo(s_2)) + \big|\Pgeo|_{[s_2,s_3]}\big|_{\diri},	
\end{align}
where the second equality follows because $\Mmu$ is supported on $\graph(\Sbdy(\Fu_*))$.
Next, since $\Pchi$ is a characteristic, using $\Fh_*=\frac{1}{4}\Fu_*^2$ within $\SD$ and \eqref{e.characteristics}, we have
$
\Fh_*(s_2,\Pgeo(s_2)) - \Fh_*(s_1,\Pgeo(s_1)) = |\Pchi|_{[s_1,s_2]}|_{\diri}.
$
Inserting this to \eqref{e.p.comparing.measure.contradict.1} gives
\begin{align}
	\label{e.p.comparing.measure.contradict.3}
	\Fh_*(s_3,\Pgeo(s_3)) 
	= 
	\Fh_*(s_1,\Pgeo(s_1)) + \big|\Pchi|_{[s_1,s_2]}\cup \Pgeo|_{[s_2,s_3]}\big|_{\diri}.
\end{align}
We finish the proof in two separate cases.
\begin{description}
\item[Case~1. The path $\Peta:=\Pchi\cup \Pgeo|_{[s_2,s_3]}$ is linear]
Consider $\Peta(s_2)<\Pl(s_2)$; the other case is similar.
In this case $\Peta$ is a characteristic and intersects with $\Pgeo$ at $(s_3,\Pgeo(s_3))$, so $\Fu_*|_{\Pl-}(s_3)=\dot{\Peta}$.
Further, $\Peta|_{(s_2,s_3)}$ lies entirely to the left of $\Pl|_{(s_2,s_3)}$ and $\Peta(s_3)=\Pgeo(s_3)$; see Figure~\ref{f.contradiction2}.
These properties together give $\dot{\Peta}\leq\dot{\Pl}(s_3)$.
This, together with $\Fu_*|_{\Pl-}(s_3)=\dot{\Peta}$, contradicts \eqref{e.p.comparing.measure.contradict}.

\item[Case~2. The path $\Peta:=\Pchi\cup \Pgeo|_{[s_2,s_3]}$ is not linear]
In this case, \eqref{e.p.comparing.measure.contradict.3} implies
\begin{align}
	\label{e.p.comparing.measure.contradict.4}
	\Fh_*(s_3,\Pgeo(s_3)) 
	\lneqq
	\sup\big\{ \Fh_*(s_1,\Pgeo(s_1)) + |\Pzeta|_{\diri} : (s_1,\Pgeo(s_1)) \xrightarrow{\Pzeta} (s_3,\Pgeo(s_3))  \big\}.
\end{align}
On the right-hand side, write $\Fh_*(s_1,\Pgeo(s_1))=\Metric_{\Mmu}(0,0;s_1,\Pgeo(s_1))$ and bound $|\Pzeta|_{\diri}\leq|\Pzeta|_{\Metric_{\Mmu}}$.
We see that the right-hand side is bounded by $\Metric_{\Mmu}(0,0;s_3,\Pgeo(s_3))$.
This is equal to $\Fh_*(s_3,\Pgeo(s_3))$, leading to a contradiction.
\end{description}

\begin{figure}
\hfill
\begin{minipage}[t]{.45\linewidth}
\setlength{\fboxsep}{0pt}
\fbox{\includegraphics[width=\linewidth]{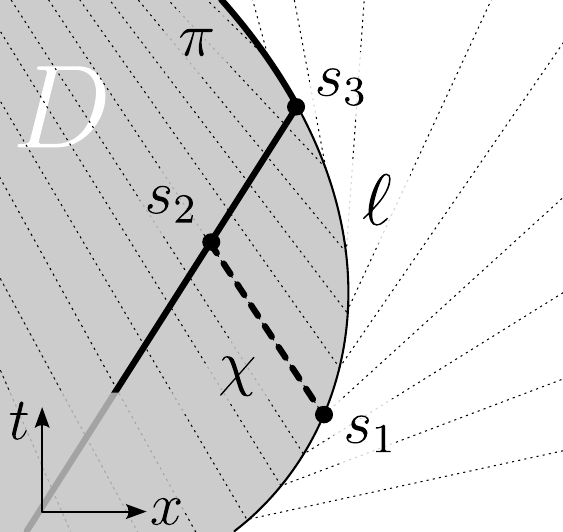}}
\caption{%
	The thin solid path is $\Pl$.
	The thick solid path is $\Pgeo$.
	The dashed paths are characteristics, with the thick one being $\Pchi$.
}
\label{f.contradiction1}
\end{minipage}
\hfill
\begin{minipage}[t]{.45\linewidth}
\setlength{\fboxsep}{0pt}
\fbox{\includegraphics[width=\linewidth]{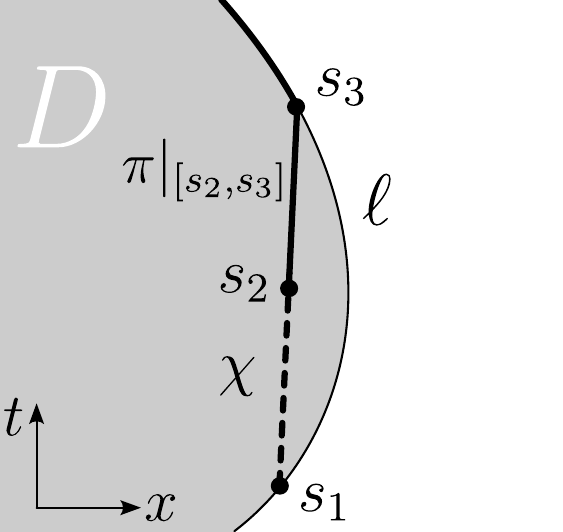}}
\caption{
	The case where $\Peta:=\Pchi\cup\Pgeo|_{[s_2,s_3]}$ is linear.
}
\label{f.contradiction2}
\end{minipage}
\end{figure}

\subsection{Proof of Corollary~\ref{c.main}}
\label{s.pfmain.cor}
Let us prove Part~\eqref{c.main.measure}.
Fix $r'>0$, recall $\EventT(r)$ from Corollary~\ref{c.main}, and consider 
$
\FF(r) := \inf\{ \rate(\metric) : \norm{\metric(0,0,1,\cdot)-\Ff}_{\Lsp^\infty(\intvl)}<r,\  \norm{\metric-\Metric_{\Mmu_*}}_{\Lsp^\infty(\SB)} > r' \}.
$
The LDP from \cite{das24} gives
\begin{align}
	\label{e.c.main.1}
	&\limsup_{r\to 0} \limsup_{\eps\to 0} \big| \eps^{-3/2} \log \P\big[ \EventT(r) \big] + \rate_{0,0;1,\intvl}(\Ff) \big| = 0,
	\\
	\label{e.c.main.2}
	&\liminf_{r\to 0} \liminf_{\eps\to 0} \eps^{-3/2} \log \P\big[ \EventT(r) \cap \{ \norm{\landscape_\eps-\Metric_{\Mmu_*}}_{\Lsp^\infty(\SB)} > r' \} \big] 
	\geq 
	-\liminf_{r\to 0}\FF(r). 
\end{align}
By Theorem~\ref{t.main}\eqref{t.main.measure}, the infimum \eqref{e.t.main} has as unique minimizer $\Metric_{\Mmu_*}$.
Using this and the property that $\rate$ is a good rate function, one deduces that $\liminf_{r\to 0}\FF(r) > \rate_{0,0;1,\intvl}(\Ff)$.
Combining this with \eqref{e.c.main.1}--\eqref{e.c.main.2} gives the desired result.

The proof of Part~\eqref{c.main.inf} is similar.
The major difference is that, for Part~\eqref{c.main.inf}, the uniqueness of $\Mmu$ is not given. 
Yet, by Theorem~\ref{t.main}\eqref{t.main.inf}, $\Hfn[\Mmu]=\Fh_*$ for every minimizer $\Mmu$.
This suffices for the proof, since Part~\eqref{c.main.inf} only concerns the marginal $\landscape_{\eps}(0,0;\cdot,\cdot)$.

\subsection{Proof of Proposition~\ref{p.multi}}
\label{s.pfmain.multi}
Fix any minimizer $\Mmu$ of \eqref{e.rate.multi}.
Since $\Mmu$ satisfies the condition in the infimum \eqref{e.rate.multi}, for each $y\in\SY$, there exists a $z_y\in\SZ$ and an $\Metric_{\Mmu}$ geodesic $(0,z_y)\xrightarrow{\Pgeo_{y}}(1,y)$ such that $|\Pgeo_y|_{\Metric_{\Mmu}}+\Fg(z_y)=\Ff(y)$.
The choice of $(z_y,\Pgeo_{y})$ may not be unique, and we argue that they can be chosen to satisfy the non-intersecting property: $\graph(\Pgeo_{y})\cap\graph(\Pgeo_{y'})=\emptyset$ whenever $z_y\neq z_{y'}$.
To see how, order the $y$s as $Y=\{y_1<\cdots<y_{|Y|}\}$, consider $y_1,y_2$, and write them temporarily as $(y_1,y_2)=(y,y')$ to simplify notation.
If $\pi_y,\pi_{y'}$ do not intersect, we keep them unchanged.
If they intersect, consider the last time $\bar{t}:=\inf\{ t\in[0,1]: \Pgeo_{y}(t)= \Pgeo_{y'}(t) \}$ that happens.
We must have $|\Pgeo_y|_{[0,\bar{t}]}|_{\Metric_{\Mmu}}+\Fg(z_y)=|\Pgeo_{y'}|_{[0,\bar{t}]}|_{\Metric_{\Mmu}}+\Fg(z_{y'})$.
Otherwise, say the inequality $>$ held, adding $|\Pgeo_{y'}|_{[\bar{t},1]}|_{\Metric_{\Mmu}}$ to both sides would give
\begin{align}
	\big|\,\Pgeo_y|_{[0,\bar{t}]}\,\big|_{\Metric_{\Mmu}}
	+
	\big|\,\Pgeo_{y'}|_{[\bar{t},1]}\,\big|_{\Metric_{\Mmu}}+\Fg(z_y)>\Ff(y').
\end{align}
On the left-hand side, concatenating paths $\Peta:=\Pgeo_y|_{[0,\bar{t}]}\cup\Pgeo_{y'}|_{[\bar{t},1]}$ and writing the result as $|\Peta|_{\Metric_{\Mmu}}+\Fg(z_y)$, we see that the left-hand side is bounded above by $\max_{z\in\SZ}\{ \Metric_{\Mmu}(0,z;1,y')+\Fg(z) \}$.
This contradicts the condition in the infimum in \eqref{e.rate.multi}.
Hence $|\Pgeo_y|_{[0,\bar{t}]}|_{\Metric_{\Mmu}}+\Fg(z_y)=|\Pgeo_{y'}|_{[0,\bar{t}]}|_{\Metric_{\Mmu}}+\Fg(z_{y'})$.
Given this equality, we make a new choice for the geodesic for $y'=y_2$ as $(z_{y'},\Pgeo_{y'})\mapsto (z_{y},\Pgeo_{y}|_{[0,\bar{t}]}\cup\Pgeo_{y'}|_{[\bar{t},1]})$.
This way, the geodesics for $(y,y')=(y_1,y_2)$ form a tree.
Repeat the same procedure for $(y,y')=(y_2,y_3),\ldots,(y_{|Y|-1},y_{|Y|})$.
At each step, either $\pi_{y'}$ does not intersect with the previous graphs, or it is modified to become a branch of a previous graph.
Doing so produces $(z_y,\Pgeo_{y})_{y\in\SY}$ that satisfies the non-intersecting property.
Let $\SY_{z}:=\{y: z_{y}=z\}$, $\SQ_z:=\cup_{y\in\SY_{z}}\graph(\Pgeo_{y})$, and $\Mmu_{z}:=\Mmu\lfloor_{\SQ_{z}}$.
By construction, these $\Mmu_{z}$s satisfy Conditions~\eqref{p.mult.disjoint}--\eqref{p.mult.=} in Proposition~\ref{p.multi}, and Condition~\eqref{p.mult.<} holds because $\Mmu_{z}\leq \Mmu$.
Let $\Mnu:=\sum_{z\in\SZ} \Mmu_{z}$, which is $\leq \Mmu$ by construction.
Given the preceding properties, we must have $\Mmu=\Mnu$, otherwise $\Mnu$ satisfies the condition in \eqref{e.rate.multi} with $\Mnu<\Mmu$, which implies $\rate(\Metric_{\Mnu})<\rate(\Metric_{\Mmu})$ by \eqref{e.monotone} and contradicts the assumption that $\Mmu$ minimizes \eqref{e.rate.multi}.
This completes the proof.

\bibliographystyle{alpha}
\bibliography{kpz-ldp,moments}
	
\end{document}